\documentclass[12pt,a4paper]{article}
\usepackage{amsfonts}

\usepackage{mathrsfs}

\usepackage{stmaryrd}
\usepackage{amsmath}
\usepackage{amssymb}

\usepackage{xcolor}
\usepackage{hyperref}
\hypersetup{hidelinks}

\usepackage{bbm}
\usepackage{mathrsfs}
\usepackage{graphicx}
\usepackage{lipsum}

\usepackage{enumerate}
\usepackage{theorem}
\usepackage{indentfirst}

\makeatletter
\def\tank#1{\protected@xdef\@thanks{\@thanks
        \protect\footnotetext[0]{#1}}}
\def\bigfoot{

    \@footnotetext}
\makeatother

\newcommand{\ea}{\end{array}}

\newtheorem{theorem}{Theorem}[section]
\newtheorem{hypothesis}{Hypothesis}[section]

\newtheorem{proposition}{Proposition}[section]

\newtheorem{lemma}{Lemma}[section]
\newtheorem{definition}{Definition}[section]
\newtheorem{remark}{Remark}[section]

\newtheorem{conditionA}{A\kern-1.5mm}

{\theorembodyfont{\rmfamily}
\newtheorem{example}{Example}[section]}

\newenvironment{proof}{Proof.}

\def\RR{\mathbb{R}}
\def\PP{\mathbb{P}}
\def\EE{\mathbb{E}}

\def\vare{{\varepsilon}}
\def \eref#1{\hbox{(\ref{#1})}}

\def\EE{\mathbb{ E}}

\newcommand{\ep}{\ensuremath{\varepsilon}}

\allowdisplaybreaks

\begin{document}
\title{{\Large \bf Central Limit Type Theorem and Large Deviation Principle for Multi-Scale McKean-Vlasov SDEs}}

\author{{Wei Hong$^{a,b}$},~~{Shihu Li$^{a}$},~~{Wei Liu$^{a}$}\footnote{Corresponding author: weiliu@jsnu.edu.cn},~~{Xiaobin Sun$^{a}$}
\\
 \small $a.$ School of Mathematics and Statistics, Jiangsu Normal University, Xuzhou 221116, China \\
 \small $b.$ Center for Applied Mathematics, Tianjin University, Tianjin 300072, China}
\date{}
\maketitle
\begin{center}
\begin{minipage}{145mm}
{\bf Abstract.} The main aim of this work is to study the asymptotic behavior for multi-scale McKean-Vlasov stochastic dynamical systems. Firstly, we obtain a central limit type theorem, i.e. the deviation between the slow component $X^{\varepsilon}$ and the solution $\bar{X}$ of the averaged equation converges weakly to a limiting process. More precisely, $\frac{X^{\varepsilon}-\bar{X}}{\sqrt{\varepsilon}}$ converges weakly in $C([0,T],\mathbb{R}^n)$ to the solution of certain distribution dependent stochastic differential equation, which involves an extra explicit stochastic integral term. Secondly, in order to estimate the probability of deviations away from the limiting process, we further investigate the Freidlin-Wentzell's large deviation principle for multi-scale McKean-Vlasov stochastic system when the small-noise regime parameter $\delta\rightarrow 0$ and the time scale parameter $\varepsilon(\delta)$ satisfies $\varepsilon(\delta)/\delta\to 0$. The main techniques are based on the Poisson equation for central limit type theorem and the weak convergence approach for large deviation principle.

\vspace{3mm} {\bf Keywords:} Central limit type theorem; Large deviation principle; McKean-Vlasov equation; Multi-scale; Poisson equation; Weak convergence approach


\end{minipage}
\end{center}

\tableofcontents
\section{Introduction}
The McKean-Vlasov stochastic differential equations (SDEs) could be used to describe stochastic systems whose evolution is
 influenced by both the microscopic location and the macroscopic distribution of particles, i.e., the coefficients depend  not only on the solution pointwisely but also on its time marginal law. It is also called the distribution dependent SDEs in the literature.
Such kind of system arises naturally in the framework of particle systems where the particles interact in a mean field way. For instance, we consider the dynamics of following $N$-particles system
directed by SDE,
\begin{equation}\label{I03}
d X_t^{N,i} = b(X_t^{N,i},\mu_t^N)dt+\sigma(X_t^{N,i},\mu_t^N)d W_t^i,~\mu_t^N=N^{-1}\sum_{j=1}^N\delta_{X_t^{N,j}},
\end{equation}
where $i=1,\ldots,N$, the mean field interactions here is expressed as the dependency of coefficients on the empirical measure $\mu_t^N$ of the system. Under some regularity conditions on $b$ and $\sigma$ (e.g. globally Lipschitz) and exchangeability assumption on the initial conditions,
 the empirical laws $\mu_t^N$ of particles system (\ref{I03})
will converge weakly to the law of solution
of following McKean-Vlasov SDE as $N\to \infty$,
\begin{equation*}\label{I04}
d X_t = b(X_t, \mathscr{L}_{X_t})dt+\sigma(X_t, \mathscr{L}_{X_t})d W_t,
\end{equation*}
where $\mathscr{L}_{X_t}$ is the law of $X_t$ and $W_t$ is a standard Brownian motion. Therefore, the coefficients of the limiting SDE will naturally depend  on the law of solution.  This was called propagation
of chaos by Kac \cite{K56}, which has been widely studied in the literature, see e.g. \cite{HM14,JW,LWZ,M96,M1,S91}.

In the past decade, the McKean-Vlasov SDEs have been intensively investigated, e.g. the well-posedness (cf.~\cite{BLPR,H21b, RZ21, WFY}), the asymptotic properties such as the large deviations, ergodicity, Harnack inequality and the Bismut formula (cf. \cite{BRW, HLL1, LMW,LSZZ}). On the other hand,  multi-scale environments including slow and fast components have also attracted more and more attentions due to their widespread applications in  material sciences, chemical kinetics, climate dynamics etc. Typically, dynamics of chemical reaction networks often take place on notably different time-scales,  from the order of nanoseconds to the order of several days,  the reader can refer to \cite{An00,C1,C2,CF,EE,MTV} for more background.

Hence, it is also natural to consider the following  multi-scale interacting particle system
\begin{equation}\left\{\begin{array}{l}\label{I05}
\displaystyle
d X_t^{\varepsilon,N,i} = b(X_t^{\varepsilon,N,i},\mu_t^{\varepsilon,N},X_t^{\varepsilon,N,i}/\varepsilon,\nu_t^{\varepsilon,N})dt+\sigma(X_t^{\varepsilon,N,i},\mu_t^{\varepsilon,N},X_t^{\varepsilon,N,i}/\varepsilon,\nu_t^{\varepsilon,N})d W_t^i, \\
\displaystyle \mu_t^{\varepsilon,N}=N^{-1}\sum_{j=1}^N\delta_{X_t^{\varepsilon,N,j}},\nu_t^{\varepsilon,N}=N^{-1}\sum_{j=1}^N\delta_{X_t^{\varepsilon,N,j}/\varepsilon},\\
\end{array}\right.
\end{equation}
where $i=1,\ldots,N$, $\varepsilon$ is a small parameter describing the ratio of time scale,
$X_t^{\varepsilon,N,i}$ and $Y_t^{\varepsilon,N,i}:=X_t^{\varepsilon,N,i}/\varepsilon$ represent the slow and fast components in system (\ref{I05}) respectively. It is very interesting to study the combined mean field and homogenization limits (i.e. $N\rightarrow\infty$ and $\varepsilon\rightarrow0$ ) for the multi-scale interacting particle system. For example, Gomes and Pavliotis \cite{GP} studied the system (\ref{I05}) with $b(x,\mu,y,\nu)= b(x,\mu,y)$, $\sigma= c$ where $c$ is a constant, of weakly interacting diffusions moving in a two-scale locally periodic confining potential. The authors in \cite{GP} proved that, although the mean field limit and homogenization limit commute for finite time,  they do not commute in the long time limit in general.
Moreover, Delgadino et al. \cite{DGP} considered the system (\ref{I05}) with $b(x,\mu,y,\nu)= b(y,\nu)$ and $\sigma= c$ and proved that the mean field and homogenization limits  do not commute if the mean field system constrained to the torus undergoes a
phase transition, i.e. if multiple steady states exist.
Recently, Bezemek and Spiliopoulos \cite{BS} established the large deviation principle of the empirical distribution in system \eref{I05} with $b(x,\mu,y,\nu)= b(x,\mu,y)$ and $\sigma(x,\mu,y,\nu)= \sigma(x,\mu,y)$, as $\varepsilon\rightarrow0$ and $N\rightarrow\infty$ simultaneously.

As mentioned above,
for fixed $\varepsilon>0$ and $b(x,\mu,y,\nu)= b(x,\mu,y)$ (similarly for $\sigma$), one can pass to the mean field limit in (\ref{I05}) as $N\rightarrow\infty$ and get the following multi-scale McKean-Vlasov SDEs
\begin{equation*}\label{I06}
d X_t^{\varepsilon} = b(X_t^{\varepsilon},\mathscr{L}_{X_t^{\varepsilon}},Y_t^{\varepsilon})dt+\sigma(X_t^{\varepsilon},\mathscr{L}_{X_t^{\varepsilon}},Y_t^{\varepsilon})d W_t,~Y_t^{\varepsilon}=X_t^{\varepsilon}/\varepsilon.
\end{equation*}
In the present work, comparing with \cite{BS,DGP,GP}, we focus on the asymptotic behavior of multi-scale McKean-Vlasov SDEs as $\varepsilon\rightarrow0$. More precisely,
we consider the following  multi-scale distribution dependent/McKean-Vlasov SDEs
\begin{equation}\left\{\begin{array}{l}\label{E2}
\displaystyle
d X^{\delta,\varepsilon}_t = b(X^{\delta,\varepsilon}_t, \mathscr{L}_{X^{\delta,\varepsilon}_t}, Y^{\delta,\varepsilon}_t)dt+\sqrt{\delta}\sigma(X^{\delta,\varepsilon}_t, \mathscr{L}_{X^{\delta,\varepsilon}_t})d W^{1}_t, \\
\displaystyle d Y^{\delta,\varepsilon}_t =\frac{1}{\varepsilon}f(X^{\delta,\varepsilon}_t, \mathscr{L}_{X^{\delta,\varepsilon}_t}, Y^{\delta,\varepsilon}_t)dt+\frac{1}{\sqrt{\varepsilon}}g( X^{\delta,\varepsilon}_t, \mathscr{L}_{X^{\delta,\varepsilon}_t}, Y^{\delta,\varepsilon}_t)d W^{2}_t,\\
\displaystyle X^{\delta,\varepsilon}_0=\xi,~Y^{\delta,\varepsilon}_0=\zeta,
\end{array}\right.
\end{equation}
where $\{W^{1}_t\}_{t\geq 0}$ and $\{W^{2}_t\}_{t\geq 0}$ are mutually independent $d_1$ and $d_2$ dimensional standard Brownian motions on a complete filtration probability space $(\Omega, \mathscr{F}, \{\mathscr{F}_{t}\}_{t\geq0}, \mathbb{P})$ respectively, the diffusion coefficients $\sigma$ and $g$ are $n\times d_1$-dimensional and $m\times d_2$-dimensional matrices respectively, $\xi$ and $\zeta$ are $\mathscr{F}_0$-measurable $\RR^n$ and $\RR^m$-valued random variables respectively. The coefficients $b,\sigma,g,f$ satisfy some appropriate conditions, and $\delta,\varepsilon$ are  small positive parameters where $\delta$ describes the intensity of the noise and $\varepsilon$ describes  the ratio of the time scale between the slow component $X^{\delta,\varepsilon}_t$ and fast component $Y^{\delta,\varepsilon}_t$.

 Due to the different time scales and the cross interactions between the fast and slow components, it is very difficult to analyze such kind of stochastic system directly. Therefore, a simplified equation which governs the evolution of the system over a long time scale would be highly desirable and useful for applications. Fix $\delta=1$, R\"{o}ckner et al. \cite{RSX1} proved that  the slow component in system \eref{E2} will converge strongly to the
solution of the so-called averaged equation, i.e.
\begin{equation}\label{Averaging}
\sup_{t\in[0,T]}\EE|X^{\varepsilon}_t-\bar{X}_t|^2\leq C_T\big(1+\EE|\xi|^4+\EE|\zeta|^4\big)\varepsilon,
\end{equation}
where $X^{\varepsilon}$ is the first component of solution to system \eref{E2} with $\delta=1$. Here $\bar{X}_t$ is the solution of the corresponding averaged equation
\begin{equation}\left\{\begin{array}{l}
\displaystyle d\bar{X}_{t}=\bar{b}(\bar{X}_t,\mathscr{L}_{\bar{X}_t})dt+\sigma(\bar{X}_t,\mathscr{L}_{\bar{X}_t})d W^{1}_t,\\
\bar{X}_{0}=\xi,\end{array}\right. \label{1.3}
\end{equation}
where the new averaged drift coefficient is defined by  $$\bar{b}(x,\mu)=\int_{\RR^m}b(x,\mu, y)\nu^{x, \mu}(dy)$$
and $\nu^{x, \mu}$ denotes the unique invariant measure of the following frozen equation (for fixed $x$ and $\mu$),
\begin{equation*}\left\{\begin{array}{l}
\displaystyle
dY_{t}=f(x, \mu, Y_{t})dt+g(x, \mu, Y_{t})d\tilde{W}_{t}^{2},\\
Y_{0}=y\in\RR^m, \\
\end{array}\right.
\end{equation*}
 where $\{\tilde{W}_{t}^{2}\}_{t\geq 0}$ is an $d_2$-dimensional Brownian motion on another complete probability space $(\tilde{\Omega}, \tilde{\mathscr{F}}, \tilde{\mathbb{P}})$. This theory, known as the averaging principle, could be regarded as classical functional law of large numbers and has been intensively investigated
 in the past decades (see e.g. \cite{BM,B2,C1,C2,CF,CL,DSXZ,Gu03,HKP,HL1,K1,L1,PV1,PIX,RX2,SXX,XLLM,V0} and more references therein).

\subsection*{Central limit type theorem}~
The averaged process $\bar{X}_t$ above is valid only in the limiting sense, and it is clear that by \eref{Averaging} the slow process
$X^\varepsilon$ will experience fluctuations around its averaged process $\bar{X}$  for small $\varepsilon$. Meanwhile, it is also proved that the optimal convergence rate is $1/2$. In order to capture the fluctuations, it is important to study the
 asymptotic behavior of the  deviation between $X^{\varepsilon}$ and $\bar{X}$, i.e. to analyze the limit of
$$Z_t^\varepsilon:=\frac{X^{\varepsilon}_t-\bar{X}_t}{\sqrt{\varepsilon}},\ \ \text{as}~\varepsilon\to 0.$$
 Therefore, the first aim of this paper is to identify  that $\{Z^{\varepsilon}\}_{\varepsilon>0}$ converges weakly in $C([0,T];\RR^n)$ (as $\varepsilon\to 0$) to the solution of following equation
\begin{eqnarray}\label{e10}
dZ_t= \!\!\!\!\!\!\!\!&& \partial_x\bar{b}(\bar{X}_t,\mathscr{L}_{\bar{X}_t})\cdot Z_tdt
+\EE\Big[\partial_{\mu} \bar{b}(u,\mathscr{L}_{\bar{X}_t})(\bar{X}_t)\cdot Z_t\Big]\Big|_{u=\bar{X}_t}dt+\left[\partial_x\sigma(\bar{X}_t,\mathscr{L}_{\bar{X}_t})\cdot Z_t\right]dW_t^1
 \nonumber\\
 \!\!\!\!\!\!\!\!&& + \EE\Big[\partial_{\mu} \sigma(u,\mathscr{L}_{\bar{X}_t})(\bar{X}_t)\cdot Z_t\Big]\Big|_{u=\bar{X}_t}dW_t^1+\Theta(\bar{X}_t,\mathscr{L}_{\bar{X}_t})dW_t,~Z_0=0,
\end{eqnarray}
 where $\{W_t\}_{t\geq 0}$ is a $n$ dimensional standard Brownian motion, which is defined on $(\Omega,\mathscr{F},\PP)$  and independent of $\{W^1_t\}_{t\geq 0}$ and $\Theta$ is a new diffusion coefficient (see Theorem \ref{main result 2} for the details). Here $\mathbb{E}$ is  the expectation on $(\Omega,\mathscr{F},\PP)$. We remark that the formulation of the second and fourth terms on the right hand side of (\ref{e10}) essentially follows from the chain rule for the distributions of random variables (cf.~\cite[Theorem 2.1]{BRW}).

Such kind of central limit type theorem is called the normal deviations in the classical multi-scale case. For example, Khasminskii \cite{K2} studied a class of slow-fast stochastic system with vanish noise in the slow equation, Cerrai \cite{C09}  generalized the results significantly in \cite{K2} to the infinite-dimensional case. However, some crucial techniques used in \cite{C09,K2} are not applicable for the fully coupled stochastic system (e.g. (\ref{E2}) here).  Wang and Roberts \cite{W12} study the normal deviations of the stochastic system driven by additive noise in infinite dimension, where the  martingale approach is used to characterize the limiting process.

%
%
%
For the purpose of characterizing  the limiting process, the technique of Poisson equation depending on the parameter measure is employed in this work.
Moreover,  we shall apply the martingale representation theorem by identifying the corresponding quadratic variational process to construct the desired limiting process, here the time discretization technique plays an important role in the proof (see Lemma \ref{L4.5}).
To the best of our knowledge, this is the first result concerning the central limit type theorem for multi-scale McKean-Vlasov SDEs.  Compared with the existing results in the distribution independent case (cf.~\cite{C09,RX2,RXY,W12}), two new terms involving the Lions derivative of the coefficients appear in the limiting equation \eref{e10}. It seems quite natural since the coefficients of (\ref{E2}) also depend on the distribution of the solution, and it highlights the intrinsic difference between McKean-Vlasov SDEs and classical SDEs.

 We want to mention that a very interesting result concerning the central limit type theorem for SDEs was obtain recently in \cite{RX2} by R\"{o}ckner, Xie and Sun. Shortly speaking, the authors in \cite{RX2} get the convergence of deviations  of  a nice function $\varphi$  acting on $Z^\varepsilon_t$ (e.g. $\varphi\in C_b^4(\RR^n)$), and their key technique is based on studying the regularity of the semigroup generated from a linear Kolmogorov equation essentially. However, in this work we mainly investigate  a class of McKean-Vlasov SDEs, which will generate a nonlinear Kolmogorov equation. Therefore,  the strategy of \cite{RX2} seems quite difficult to apply to our case here. Moreover, as we mentioned before, we obtain the convergence of $Z^\varepsilon$ in space $C([0,T];\RR^n)$, whereas the results of \cite{RX2} hold only for fixed time $t\in[0,T]$.  The  reader can refer to e.g. \cite{BGS,HP1,HL,RXY} and the references therein for more results in this direction.

The result of central limit type theorem indicates that typical fluctuations of the slow process around its limiting process
are of order $O(\sqrt{\varepsilon})$. In particular,
we can get the formal asymptotic expansion $$X_t^\varepsilon\approx^D \bar{X}_t+\sqrt{\varepsilon}Z_t,$$
where
$\approx^D$ means approximate equality of probability distributions. Such expansion provides better approximations
for the original system \eref{E2}, which is also called the Van Kampen's approximation in physics and
has been used in the context of stochastic climate models (see e.g. \cite{A01,BK04}).

\subsection*{Large deviation principle}
It follows that the central limit type theorem above does not permit to estimate the probability of deviations of order one away from the limiting process.
On these time-scales, large
deviations are expected to be rare, and their probability could be estimated by a large deviation principle (LDP).
Note that the analysis of averaging principle and central limit type theorem are only for fixed $\delta=1$ in (\ref{E2}). It is natural to ask how the multi-scales influence
the stochastic system (\ref{E2})  when $\delta$ and $\varepsilon$  converge to zero simultaneously.

The LDP theory mainly investigates the asymptotic property of remote tails of a family of probability distributions, which is one of important topics in the probability theory and has been widely applied in many fields such as statistics, information theory and engineering.
Since the well-known work of Freidlin and Wentzell \cite{FW}, the small perturbation type (also called Freidlin-Wentzell type) LDP for SDEs has been widely studied in the literature.  A very powerful tool for studying the Freidlin-Wentzell's LDP is the weak convergence approach (see e.g. \cite{BPZ,MSZ,RZ2}),
which is systematically developed by Budhiraja, Dupuis and Ellis et al. in \cite{BD,BDM,DE}. The key idea of this approach is to prove certain variational
representation formula about the Laplace transform of bounded continuous functionals, which then leads
to the verification of the equivalence between  LDP and  Laplace principle.
The LDP  has also been studied a lot for  multi-scale standard SDEs (i.e. distribution independent case), see e.g. \cite{Gu03,HLL2,KP,SWXY} and reference therein.

To the best of our knowledge, there is no LDP result for multi-scale distribution dependent stochastic system so far. Hence, the second aim of this paper is to further investigate the small noise asymptotic behavior and establish LDP for multi-scale distribution dependent SDEs.
More precisely, assuming suitable condition between $\delta$ and $\varepsilon=\varepsilon(\delta)$ (i.e.~$\varepsilon/\delta\to 0$ as $\delta\to 0$), we prove that $\{X^{\delta,\varepsilon}\}_{\delta>0}$ in \eref{E2}
satisfies the LDP in $C([0,T]; \RR^n)$. Studying the LDP for the multi-scale stochastic system is of practical significance. For instance, Spiliopoulos in \cite{S1} establish the LDP and use it to construct the asymptotically optimal importance sampling for systems of multi-scale motion. Some applications are also given in \cite{S1} such as the short time asymptotics of a process that depends on another fast mean reverting process.

The main difficulty of verifying LDP for multi-scale McKean-Vlasov stochastic system is
due to the dependence of distribution and the cross interactions between the fast and slow component in the model. Our strategy here is still based on the powerful weak convergence approach.  As an important part of the proof, we need to formulate the correct form of the skeleton equation.
Intuitively, as the parameter $\delta\rightarrow0$ in system (\ref{E2}) (hence $\varepsilon\rightarrow0$), the drift term is averaged and the noise term vanishes, then we can get the following ordinary differential equation
\begin{equation*}
\frac{d\bar{X}_t^0}{dt}=\bar{b}(\bar{X}_t^0,\mathscr{L}_{\bar{X}_t^0}),~\bar{X}_0^0=x,
\end{equation*}
where $\mathscr{L}_{\bar{X}^0_t}=\delta_{\bar{X}^0_t}$ is the Dirac measure of $\bar{X}^0_t$.  Then for any $h$ belongs to the Cameron-Martin space $\mathcal{H}_0$ (see \eref{CM1} below), we define the following skeleton equation with respect to the slow equation of (\ref{E2}),
\begin{equation}\label{ske}
\frac{d \bar{X}^{h}_t}{dt}=\bar{b}(\bar{X}^{h}_t,\mathscr{L}_{\bar{X}^{0}_t})+\sigma(\bar{X}^{h}_t,\mathscr{L}_{\bar{X}^{0}_t})P_1\dot{h}_t,~\bar{X}^{h}_0=x,
\end{equation}
where $P_1$ is a projection operator (see \eref{bm} below). Note that $\mathscr{L}_{\bar{X}^0_t}$ (instead of  $\mathscr{L}_{\bar{X}_t^h}$) shows up in the skeleton equation \eref{ske} and is used to define the rate function of LDP. Since the measure $\mathscr{L}_{X_t^{\delta,\varepsilon}}$ is not random, so that the convergence of these
measures is independent of the occurrence of a rare event for the random variable $X_t^{\delta,\varepsilon}$. This is also an important difference between McKean-Vlasov SDEs and standard SDEs.

In fact, if one formulate the correct form of the controlled process $X_t^{\delta,\varepsilon, h^\delta}$ as follows \begin{equation}\label{control1}
\left\{ \begin{aligned}
&dX^{\delta,\varepsilon,h^\delta}_t=b(X^{\delta,\varepsilon,h^\delta}_t,\mathscr{L}_{X^{\delta,\varepsilon}_t},Y^{\delta,\varepsilon,h^\delta}_t)dt
+\sigma(X^{\delta,\varepsilon,h^\delta}_t,\mathscr{L}_{X^{\delta,\varepsilon}_t})P_1\dot{h}^\delta_t dt\\
&~~~~~~~~~~~~+\sqrt{\delta}\sigma(X^{\delta,\varepsilon,h^\delta}_t,\mathscr{L}_{X^{\delta,\varepsilon}_t})dW_t^1,\\
&dY^{\delta,\varepsilon,h^\delta}_t=\frac{1}{\varepsilon}f(X^{\delta,\varepsilon,h^\delta}_t,\mathscr{L}_{X^{\delta,\varepsilon}_t},Y^{\delta,\varepsilon,h^\delta}_t)dt
+\frac{1}{\sqrt{\delta\varepsilon}}g(X^{\delta,\varepsilon,h^\delta}_t,\mathscr{L}_{X^{\delta,\varepsilon}_t},Y^{\delta,\varepsilon,h^\delta}_t)P_2\dot{h}^\delta_t dt\\
&~~~~~~~~~~~~+\frac{1}{\sqrt{\varepsilon}}g(X^{\delta,\varepsilon,h^\delta}_t,\mathscr{L}_{X^{\delta,\varepsilon}_t},Y^{\delta,\varepsilon,h^\delta}_t)d W_t^2,\\
&X^{\delta,\varepsilon,h^\delta}_0=x, Y^{\delta,\varepsilon,h^\delta}_0=y,
\end{aligned} \right.
\end{equation}
where $P_1,P_2$ are the projection operations defined in (\ref{bm}) below. Then let $\delta\rightarrow0$, it's natural to derive the skeleton equation (\ref{ske}). Furthermore, it turns out that the correct rate function for the multi-scale McKean-Vlasov SDEs is defined through the solution of skeleton equation (\ref{ske}). We mention that in the single-scale distribution dependent case, the work \cite[Section 4.1]{DST} also state that the Dirac measure $\mathscr{L}_{\bar{X}^0_t}$ is a good  approximation of $\mathscr{L}_{X_t^\delta}$, where they first replace the distribution $\mathscr{L}_{X_t^\delta}$  by $\mathscr{L}_{\bar{X}^0_t}$.
Then they use some discrete approximation and exponential equivalence arguments to show the LDP. Different from the method in \cite{DST}, in this work we employ the weak convergence method to prove the LDP for the distribution dependent SDEs in the multi-scale case.
Moreover, in order to overcome the difficulty caused by the cross interactions between the fast and slow component, the time discretization technique will be employed frequently to obtain some crucial estimates.

Note that if $\delta$ and $\varepsilon$ converge to zero in different speed, we might have three different regimes of interaction, i.e.
\begin{equation*}
\lim_{\delta\rightarrow0}\varepsilon/\delta=\left\{
  \begin{array}{ll}
    0, & \hbox{Regime 1;} \\
    \gamma, & \hbox{Regime 2;} \\
    \infty, & \hbox{Regime 3.}
  \end{array}
\right.
\end{equation*}
In the classical SDE case with periodic coefficients, Dupuis and Spiliopoulos \cite{DS12} proved LDP of multi-scale SDE for all three regimes  using the characterization of optimal controls through solutions to Hamilton-Jacobi-Bellman equations.
In Regime 1, such regime allows to decouple the invariant measure of the associated frozen equation and the control function from the limiting occupation measures. Roughly speaking, as $\delta\to 0$, we need $\varepsilon/\delta\to 0$,  thus the coefficient of the multi-scale stochastic system is averaged first, and then for the averaged equation with small noise, the large deviation principle is obtained by using the weak convergence method.
However, in the Regimes 2 and 3, such invariant measure will depend on the control function essentially and this dependence will make the analysis very challenging in the McKean-Vlasov structure.

\vspace{6mm}

The remainder of this paper is organized as follows. In Section 2, we introduce the detailed framework and state the main results of this work. In Section 3, we first establish some estimates of solution to the slow-fast system, and introduce the Poisson equation with measure dependence. Secondly, we give the detailed proof of central limit type theorem. Section 4 is devoted to proving the LDP via weak convergence approach, and  we postpone the proofs of some crucial estimates to Section 5 as Appendix.
Note that throughout this paper $C$ and $C_T$  denote positive constants which may change from line to line, where the subscript $T$ is used to emphasize that the constant depends on certain parameter $T$.

\section{Main results}\label{sec.prelim}
\setcounter{equation}{0}
 \setcounter{definition}{0}
\subsection{Main framework}
Now we introduce some notations which will be used frequently throughout this paper. Denote by $|\cdot|$ the Euclidean vector norm, $\langle\cdot, \cdot\rangle$ the Euclidean inner product and $\|\cdot\|$ the matrix norm or the operator norm if there is no confusion possible. The tensor product $\mathbb{R}^{l_1}\otimes\mathbb{R}^{l_2}$ represents the space of all ${l_1}\times {l_2}$-dimensional matrix for $l_1,l_2\in \mathbb{N}_{+}$. A matrix-valued function $u(x,y)$ defined on $\RR^n\times\RR^m$, we use $\partial_x u$ and $\partial_y u$ to denote the first order partial derivative of $u$ with respect to (w.r.t.)  $x$ and $y$ respectively, $\partial^2_{xx} u$, $\partial^2_{yy} u$ and $\partial^2_{xy} u$ to denote its second order derivatives of $u$ w.r.t. $x$ twice, $y$ twice and $x,y$ for each respectively.

\vspace{0.1cm}
Define $\mathscr{P}$ the set of all probability measures on $(\RR^n, \mathscr{B}(\RR^n))$ and $\mathscr{P}_2$ by
$$
\mathscr{P}_2:=\Big\{\mu\in \mathscr{P}: \mu(|\cdot|^2):=\int_{\RR^n}|x|^2\mu(dx)<\infty\Big\},
$$
 then $\mathscr{P}_2$ is a Polish space under the $L^2$-Wasserstein distance
$$
\mathbb{W}_2(\mu_1,\mu_2):=\inf_{\pi\in \mathscr{C}_{\mu_1,\mu_2}}\left[\int_{\RR^n\times \RR^n}|x-y|^2\pi(dx,dy)\right]^{1/2}, \quad \mu_1,\mu_2\in\mathscr{P}_2,
$$
where $\mathscr{C}_{\mu_1,\mu_2}$ is the set of all couplings for $\mu_1$ and $\mu_2$. Therefore for any $X,Y\in L^2(\Omega,\PP;\RR^n)$, by the definition of the $L^2$-Wasserstein distance, we have
\begin{eqnarray}
\mathbb{W}_2(\mathscr{L}_{X},\mathscr{L}_{Y})^2\leq\EE|X-Y|^2.\label{W2}
\end{eqnarray}

In what follows, we recall the definition of Lions derivative on $\mathscr{P}_2$. Due to the strategy in \cite[Section 6]{C}, for $u: \mathscr{P}_2\rightarrow \RR$ we denote by $U$ its ``extension" to $L^2(\Omega, \PP;\RR^n)$ defined by
$$
U(X)=u(\mathscr{L}_{X}),\quad X\in L^2(\Omega,\PP;\RR^n).
$$
Then we say that $u$ is differentiable at $\mu\in\mathscr{P}_2$ if there exists $X\in L^2(\Omega,\PP;\RR^n)$ such that $\mathscr{L}_{X}=\mu$ and $U$ is Fr\'echet differentiable at $X$. By means of the Riesz theorem, the Fr\'echet derivative $DU(X)$, viewed as an element of $L^2(\Omega,\PP;\RR^n)$, could be represented as
$$
DU(X)=\partial_{\mu}u(\mathscr{L}_{X})(X),
$$
where $\partial_{\mu}u(\mathscr{L}_{X}):\RR^n\rightarrow \RR^n$, which is called Lions derivative of $u$ at $\mu= \mathscr{L}_{X}$. It should be mentioned that the mapping $\partial_{\mu}u(\mathscr{L}_{X})$ depends only on $\mathscr{L}_X$, not on $X$.
In addition, $\partial_{\mu}u(\mu)\in L^2(\mu;\RR^n)$ for $\mu\in\mathscr{P}_2$. Furthermore, if $\partial_{\mu}u(\mu)(z):\RR^n\rightarrow \RR^n$ is differentiable at $z\in\RR^n$, we denote its derivative by $\partial_{z}\partial_{\mu}u(\mu)(z):\RR^n\rightarrow \RR^n\times\RR^n$.

\vspace{0.1cm}
We say that a matrix-valued function $u(\mu)=(u_{ij}(\mu))$ differentiable at $\mu\in\mathscr{P}_2$, if all its components are  differentiable at $\mu$, and set $$ \partial_{\mu}u(\mu)=(\partial_{\mu}u_{ij}(\mu)), ~~   \|\partial_{\mu}u(\mu)\|^2_{L^2(\mu)}=\sum_{i,j}\int_{\RR^n}|\partial_{\mu}u_{ij}(\mu)(z)|^2\mu(dz). $$
 Moreover, we say $\partial_{\mu}u(\mu)(z)$ differentiable at $z\in\RR^n$, if all its components are differentiable at $z$, and set $$\partial_{z}\partial_{\mu}u(\mu)(z)=(\partial_{z}\partial_{\mu}u_{ij}(\mu)(z)), ~~ \|\partial_{z}\partial_{\mu}u(\mu)\|^2_{L^2(\mu)}=\sum_{i,j}\int_{\RR^n}\|\partial_{z}\partial_{\mu}u_{ij}(\mu)(z)\|^2\mu(dz).$$

\vspace{0.1cm}
For the reader's convenience, we recall the following definitions.

 \begin{definition} For a  map $u(\cdot): \mathscr{P}_2 \to \RR$, we say $u\in C^{(1,1)}(\mathscr{P}_2; \RR)$, if this map is continuously differentiable at any $\mu\in\mathscr{P}_2$ and its derivative $\partial_{\mu}u(\mu)(z):\RR^n\rightarrow \RR^n$ is continuously differentiable at any $z\in\RR^n$. We say $u\in C^{(1,1)}_b(\mathscr{P}_2; \RR)$, if $u\in C^{(1,1)}(\mathscr{P}_2; \RR)$, moreover the derivatives $\partial_{\mu}u(\mu)(z)$ and $\partial_z\partial_{\mu}u(\mu)(z)$ are jointly continuous at any $(\mu,z)$, and uniformly bounded w.r.t.~$(\mu,z)$,~i.e., $\sup_{\mu\in\mathscr{P}_2,z\in\RR^n}|\partial_{\mu}u(\mu)(z)|<\infty$ and $\sup_{\mu\in\mathscr{P}_2,z\in\RR^n}\|\partial_z\partial_{\mu}u(\mu)(z)\|<\infty$. For a matrix-valued map $u(\cdot): \mathscr{P}_2 \to \RR^{l_1}\otimes \RR^{l_2}$, where $l_1,l_2\in \mathbb{N}_{+}$, we say $u\in C^{(1,1)}(\mathscr{P}_2;\RR^{l_1}\otimes \RR^{l_2})$ (resp. $C^{(1,1)}_b(\mathscr{P}_2;\RR^{l_1}\otimes \RR^{l_2})$) if all the components belong to $C^{(1,1)}(\mathscr{P}_2;\RR)$ (resp. $C^{(1,1)}_b(\mathscr{P}_2;\RR)$).
\end{definition}

\begin{definition} For a  map $u(\cdot): \RR^n \to \RR$, we say $u\in C^{2}_b(\RR^n; \RR)$, if the derivatives $\partial_{x} u(x)$,  $\partial^2_{xx} u(x)$ are bounded and continuous at any $x$. For a  map $u(\cdot,\cdot): \RR^n\times\RR^m \to \RR$, we say $u\in C^{2,2}(\RR^n\times\RR^m; \RR)$, if the partial derivatives $\partial_{x} u(x,y)$, $\partial_{y} u(x,y)$, $\partial^2_{xx} u(x,y)$, $\partial^2_{xy} u(x,y)$ and $\partial^2_{yy} u(x,y)$ exist at any $(x,y)$. We say $u\in C^{2,2}_b(\RR^n\times\RR^m; \RR)$, if $u\in C^{2,2}(\RR^n\times\RR^m; \RR)$ and the partial derivatives $\partial_{x} u(x,y)$, $\partial_{y} u(x,y)$, $\partial^2_{xx} u(x,y)$, $\partial^2_{xy} u(x,y)$ and $\partial^2_{yy} u(x,y)$ are jointly continuous at any $(x,y)$ and uniformly bounded w.r.t.  $(x,y)$. For a matrix-valued map $u(\cdot): \RR^n \to \RR^{l_1}\otimes \RR^{l_2}$, we say $u\in C^{2}_b(\RR^n;\RR^{l_1}\otimes \RR^{l_2})$ if all the components belong to $C^{2}_b(\RR^n;\RR)$. Similarly, we say $u\in C^{2,2}(\RR^n\times\RR^m;\RR^{l_1}\otimes \RR^{l_2})$ (resp. $C^{2,2}_b(\RR^n\times\RR^m;\RR^{l_1}\otimes \RR^{l_2})$) if all the components belong to $C^{2,2}(\RR^n\times\RR^m;\RR)$ (resp. $C^{2,2}_b(\RR^n\times\RR^m;\RR)$).

\end{definition}

\begin{definition}
For a matrix-valued map $u(\cdot,\cdot): \RR^n\times\mathscr{P}_2  \to \RR^{l_1}\otimes \RR^{l_2}$, we say $u\in C^{2,(1,1)}_b(\RR^n\times\mathscr{P}_2;\RR^{l_1}\otimes \RR^{l_2})$ if $u(x,\cdot)\in C^{(1,1)}_b(\mathscr{P}_2;\RR^{l_1}\otimes \RR^{l_2})$ for any $x\in\RR^n$ and $u(\cdot,\mu)\in C^{2}_b(\RR^n;\RR^{l_1}\otimes \RR^{l_2})$ for any $\mu\in\mathscr{P}_2$, moreover, the derivatives $\partial_{x} u(x,\mu)$, $\partial^2_{xx} u(x,\mu)$,  $\partial_{\mu} u(x,\mu)(z)$, $\partial_{z}\partial_{\mu} u(x,\mu)(z)$ are jointly continuous at any $(x,\mu,z)$ and uniformly bounded w.r.t.~$(x,\mu,z)$.
\end{definition}

\begin{definition} For a matrix-valued map $u(\cdot,\cdot,\cdot): \RR^n\times\mathscr{P}_2 \times\RR^m \to \RR^{l_1}\otimes \RR^{l_2}$, we say $u\in C^{2,(1,1),2}(\RR^n\times\mathscr{P}_2\times\RR^m;\RR^{l_1}\otimes \RR^{l_2})$ if $u(x,\cdot,y)\in C^{(1,1)}(\mathscr{P}_2;\RR^{l_1}\otimes \RR^{l_2})$ for any $(x,y)\in\RR^n\times \RR^m$ and $u(\cdot,\mu,\cdot)\in C^{2,2}(\RR^n\times\RR^m;\RR^{l_1}\otimes \RR^{l_2})$ for any $\mu\in\mathscr{P}_2$. We say $u\in C^{2,(1,1),2}_b(\RR^n\times\mathscr{P}_2\times\RR^m;\RR^{l_1}\otimes \RR^{l_2})$, if $u\in C^{2,(1,1),2}(\RR^n\times\mathscr{P}_2\times\RR^m;\RR^{l_1}\otimes \RR^{l_2})$, moreover the partial derivatives $\partial_{x} u(x,\mu,y)$, $\partial_{y} u(x,\mu,y)$, $\partial^2_{xx} u(x,\mu,y)$, $\partial^2_{xy} u(x,\mu,y)$, $\partial^2_{yy} u(x,\mu,y)$, $\partial_{\mu}u(x,\mu,y)(z)$ and $\partial_z\partial_{\mu}u(x,\mu,y)(z)$ are uniformly bounded w.r.t.  $(x,\mu,y,z)$ and uniformly continuous on $\RR^n\times\mathscr{P}_2 \times\RR^m\times\RR^n$.
\end{definition}

\vspace{0.1cm}
We assume that the maps
\begin{eqnarray*}
&&b: \RR^n\times\mathscr{P}_2\times\RR^m \rightarrow \RR^{n};\\
&& \sigma: \RR^n\times \mathscr{P}_2\rightarrow \RR^{n}\otimes\RR^{d_1};\\
&&f:\RR^n\times\mathscr{P}_2\times\RR^m\rightarrow \RR^{m};\\
&&g:\RR^n\times\mathscr{P}_2\times\RR^m\rightarrow \RR^{m}\otimes\RR^{d_2}
\end{eqnarray*}
satisfy the following conditions.

\smallskip
\noindent
\begin{conditionA}\label{A1} Suppose that there exist constants $C,\gamma>0$ such that for all $x,x_1,x_2\in\RR^n, \mu,\mu_1,\mu_2\in \mathscr{P}_2, y,y_1,y_2\in\RR^m$,
\begin{eqnarray}
&&|b(x_1, \mu_1, y_1)-b(x_2, \mu_2, y_2)|+\|\sigma(x_1,\mu_1)-\sigma(x_2,\mu_2)\|\nonumber\\
\leq \!\!\!\!\!\!\!\!&&C\big[|x_1-x_2|+|y_1-y_2|+\mathbb{W}_2(\mu_1, \mu_2)\big] , \label{A11}
\end{eqnarray}
\begin{eqnarray}
&&|f(x_1,\mu_1, y_1)-f(x_2, \mu_2,y_2)|+\|g(x_1, \mu_1,y_1)-g(x_2, \mu_2, y_2)\|\nonumber\\
\leq \!\!\!\!\!\!&&C\big[|x_1-x_2|+|y_1-y_2|+\mathbb{W}_2(\mu_1, \mu_2)\big].\label{A21}
\end{eqnarray}
Moreover,
\begin{equation}
2\langle f(x,\mu, y_1)-f(x, \mu,y_2), y_1-y_2\rangle\!+5\|g(x, \mu,y_1)-g(x,\mu, y_2)\|^2\!\leq -\gamma|y_1-y_2|^2.\label{sm}
\end{equation}

\end{conditionA}

\smallskip
\noindent
\begin{conditionA}\label{A2}
Suppose that $b\in C_b^{2,(1,1),2}(\RR^n\times \mathscr{P}_2\times\RR^m;\RR^n)$, $\sigma\in C_b^{2,(1,1)}(\RR^n\times \mathscr{P}_2;\RR^{n}\otimes\RR^{d_1})$, $f\in C^{2,(1,1),2}_b(\RR^n\times \mathscr{P}_2\times\RR^m;\RR^m)$ and $g\in C^{2,(1,1),2}_b(\RR^n\times \mathscr{P}_2\times\RR^m;\RR^{m}\otimes\RR^{d_2})$. Moreover, there exist constants $C>0$  and $\gamma_1\in (0,1]$ such that for all $y_1,y_2\in\RR^m$,
\begin{eqnarray}
\sup_{x\in\RR^n,\mu\in\mathscr{P}_2,z\in\RR^m}\|\partial_{\mu} F(x, \mu, y_1)(z)-\partial_{\mu} F(x, \mu, y_2)(z)\|\leq C|y_1-y_2|^{\gamma_1} , \label{A40}
\end{eqnarray}
\begin{eqnarray}
\sup_{x\in\RR^n,\mu\in\mathscr{P}_2}\|\partial^2_{xx} F(x, \mu, y_1)-\partial^2_{xx} F(x, \mu, y_2)\|\leq C|y_1-y_2|^{\gamma_1} , \label{A41}
\end{eqnarray}
\begin{eqnarray}
\sup_{x\in\RR^n,\mu\in\mathscr{P}_2}\|\partial^2_{xy} F(x, \mu, y_1)-\partial^2_{xy} F(x, \mu, y_2)\|\leq C|y_1-y_2|^{\gamma_1} , \label{A42}
\end{eqnarray}
\begin{eqnarray}
\sup_{x\in\RR^n,\mu\in\mathscr{P}_2}\|\partial^2_{yy} F(x, \mu, y_1)-\partial^2_{yy} F(x, \mu, y_2)\|\leq C|y_1-y_2|^{\gamma_1} , \label{A431}
\end{eqnarray}
\begin{eqnarray}
\sup_{x\in\RR^n,\mu\in\mathscr{P}_2}\|\partial_{z}\partial_{\mu} F(x, \mu, y_1)-\partial_{z}\partial_{\mu} F(x, \mu, y_2)\|_{L^2(\mu)}\leq C|y_1-y_2|^{\gamma_1} , \label{A44}
\end{eqnarray}
\begin{eqnarray}
\sup_{x\in\RR^n,\mu\in\mathscr{P}_2,y\in\RR^m}\|\partial_{\mu}\partial_y F(x,\mu,y)\|_{L^2(\mu)}\leq C , \label{A45}
\end{eqnarray}
where $F$ represents $b, f$ and $g$ respectively.
\end{conditionA}

\smallskip
\noindent
\begin{conditionA}\label{A3}
Suppose that $f$ and $g$ satisfy
\begin{equation}\label{32}
~\sup_{x\in\RR^n,\mu\in\mathscr{P}_2}|f(x,\mu,0)|<\infty, ~
\sup_{x\in\RR^n,\mu\in\mathscr{P}_2}\|g(x,\mu,0)\|<\infty.
\end{equation}
\end{conditionA}

\begin{conditionA}\label{A4}
There exists a constant $C>0$ such that for any $x\in \RR^n,\mu\in\mathscr{P}_2$,
\begin{equation}\label{h6}
\sup_{y\in \RR^m}\|g(x,\mu,y)\|\leq C\left(1+|x|+[\mu(|\cdot|^2)]^{\frac{1}{2}}\right).
\end{equation}

\end{conditionA}

\begin{conditionA}\label{A5}
The scale parameters $\delta$ and $\varepsilon=\varepsilon(\delta)$ satisfy
\begin{equation}\label{h5}
\lim_{\delta\to 0}  \varepsilon/\delta = 0 .
\end{equation}
\end{conditionA}

In the following, we give some comments on the conditions above.
\begin{remark}\label{R2.1}
(i) By the same argument as in the proof of \cite[Theorem 2.2]{RSX1}, for any fixed $\varepsilon, \delta>0$ and $\mathscr{F}_0$-measurable initial values $\xi\in L^2(\Omega;\RR^n), \zeta\in L^2(\Omega;\RR^m)$, \eref{A11} and (\ref{A21}) ensure the existence and uniqueness of strong solutions $\{(X^{\varepsilon,\delta}_t,Y^{\varepsilon,\delta}_t)\}_{t\geq 0}$ to system \eref{E2}.

(ii) Conditions \eref{A21} and \eref{sm} imply that for any $\beta\in (0,\gamma)$, there exists $C_{\beta}>0$  such that for any $x\in\RR^n, y\in\RR^m$, $\mu\in\mathscr{P}_2$,
\begin{eqnarray}
 2\langle f(x,\mu,y), y\rangle+5\|g(x,\mu, y)\|^2\leq -\beta|y|^2+C_{\beta}\left[1+|x|^2+\mu(|\cdot|^2)\right].\label{RE3}
\end{eqnarray}
Moreover, condition \eref{sm} can guarantee the existence and uniqueness of invariant measures for the corresponding frozen equation and the solution of \eref{E2} has finite sixth moment. If $g$ satisfies the sublinear growth w.r.t.~variable $y$, i.e. there is $\kappa\in[0,1)$ such that
\begin{eqnarray*}
\|g(x, \mu,y)\|\!\leq C(1+|x|+(\mu(|\cdot|^2))^{\frac{1}{2}}+|y|^\kappa),
\end{eqnarray*}
then (\ref{sm}) could  be modified into a more natural form as follows
\begin{equation*}
2\langle f(x,\mu, y_1)-f(x, \mu,y_2), y_1-y_2\rangle\!+\|g(x, \mu,y_1)-g(x,\mu, y_2)\|^2\!\leq -\gamma|y_1-y_2|^2.
\end{equation*}

(iii) Condition ${\mathbf{A\ref{A2}}}$ is used to study the regularity of solutions of Poisson equation (cf.~\cite[Proposition 4.1]{RSX1}), which plays an important role in the proof of central limit type theorem. Motivated from \cite[Example 2.16]{Lacker} and \cite[Example 5.1]{RSX1}, we here present a simple example of distribution dependent coefficients such that \eref{A40}-\eref{A45} hold.

Let $F_0:\RR^n\times\RR^m\rightarrow \RR^n$ satisfy that all its first and second order partial derivatives $\partial_{x}F_0(x,y)$, $\partial_{y}F_0(x,y)$, $\partial^2_{xx}F_0(x,y)$, $\partial^2_{xy}F_0(x,y)$ and $\partial^2_{yy}F_0(x,y)$ are uniformly bounded and Lipschitz continuous w.r.t. $(x,y)\in \RR^n\times\RR^m$.
Now we define the following map
$$F(x,\mu,y):=\int_{\RR^n}F_0(x+z,y)\mu(dz),$$
then we have $\partial_{\mu}F(x,\mu,y)(\cdot)=\partial_x F_0(x+\cdot,y)$, $\partial_{z}\partial_{\mu}F(x,\mu,y)(z)=\partial^2_{xx} F_0(x+z,y)$ and $\partial_{\mu}\partial_y F(x,\mu,y)=\partial^2_{xy} F_0(x+\cdot,y)$.
Hence, we can easily check that the map $F$ satisfies \eref{A40}-\eref{A45} with $\gamma_1=1$.

(iv) Condition ${\mathbf{A\ref{A3}}}$ together with the uniform boundedness of $\partial_{y}f(x,\mu,y)$ and $\partial_{y}g(x,\mu,y)$ ensure that $f(x,\mu,y)$ and $g(x,\mu,y)$ are uniformly bounded and linear growth w.r.t $(x,\mu)$ and $y$ respectively, i.e. there exists $C>0$ such that for any $x\in\RR^n,\mu\in \mathscr{P}_2,y\in\RR^m$,
$$
|f(x,\mu,y)|+\|g(x,\mu,y)\|\leq C(1+|y|),
$$
which is used to estimate $\mathbb{E}\Big[\sup_{t\in [0, T]}|Y_{t}^{\varepsilon}|^{4}\Big]$ (see details in Lemma \ref{PMY} below).

(v) Conditions ${\mathbf{A\ref{A4}}}$-${\mathbf{A\ref{A5}}}$ are used to investigate the LDP for the McKean-Vlasov system (\ref{E2}). More precisely, ${\mathbf{A\ref{A4}}}$ is used to derive some a priori estimates for the control equation (\ref{control1}). For the scale assumption ${\mathbf{A\ref{A5}}}$, we have already mentioned in the introduction.

\end{remark}

In the following, we give some comments on the applications of the slow-fast system (\ref{E2}) with $\delta=1$, in particular, we give some examples on the confining or interaction potentials from the perspective of the multi-scale interacting particle systems.
\begin{example}

The multi-scale interacting particle systems have received more and more attention in the literature. For example,
Gomes and Pavliotis \cite{GP} studied the following multi-scale interacting diffusions
\begin{equation}\label{ree1}
d X_t^{\varepsilon,i} =-\nabla V^\varepsilon(X_t^{\varepsilon,i})dt-\frac{\theta}{N}\sum_{j=1}^N\nabla F(X_t^{\varepsilon,i}-X_t^{\varepsilon,j})dt+\sqrt{2\beta^{-1}}d W^i_t,
\end{equation}
where $V^\varepsilon(x):=V(x,x/\sqrt{\varepsilon}):\mathbb{R}\rightarrow\mathbb{R}$ is the confining potential with a fast fluctuating,
$F:\mathbb{R}\rightarrow\mathbb{R}$ is the interaction potential, $\theta,\beta$ are some constants.

The McKean-Vlasov SDE associated with (\ref{ree1}) is
\begin{equation}\label{re05}
d X_t^{\varepsilon} =\big(-\nabla V^\varepsilon(X_t^{\varepsilon})-\theta \mathbb{E}\big[\nabla F(x-X_t^{\varepsilon})\big]|_{x=X_t^{\varepsilon}}\big)dt+\sqrt{2\beta^{-1}}d W_t.
\end{equation}
Take the auxiliary variable $Y_t^\varepsilon:=X_t^\varepsilon/\sqrt{\varepsilon}$,  we can rewrite (\ref{re05}) as the following multi-scale McKean-Vlasov stochastic system
\begin{equation}\left\{\begin{array}{l}\label{res2}
\displaystyle
d X^{\varepsilon}_t = \Big[-\partial_x V(X_t^{\varepsilon},Y_t^{\varepsilon})- \partial_y V(X_t^{\varepsilon},Y_t^{\varepsilon})-\theta \mathbb{E}\big[\nabla F(x-X_t^{\varepsilon})\big]|_{x=X_t^{\varepsilon}}\Big]dt+\sqrt{2\beta^{-1}}d W_t, \\
d Y^{\varepsilon}_t = \Big[-\frac{1}{\varepsilon}\partial_y V(X_t^{\varepsilon},Y_t^{\varepsilon})-\frac{1}{\sqrt{\varepsilon}}\partial_x V(X_t^{\varepsilon},Y_t^{\varepsilon})-\frac{\theta}{\sqrt{\varepsilon}} \mathbb{E}\big[\nabla F(x-X_t^{\varepsilon})\big]|_{x=X_t^{\varepsilon}}\Big]dt+\sqrt{\frac{2\beta^{-1}}{\varepsilon}}d W_t.
\end{array}\right.
\end{equation}

In \cite{GP}, some typical examples of potentials are also given. For the multi-scale potential in (\ref{re05}), the authors considered the confining potential perturbed by a fast periodic fluctuations such as
$$V^\varepsilon(x):=\frac{x^4}{4}-\frac{x^2}{2}+c_0\cos(x/\sqrt{\varepsilon})~~\text{or}~~V^\varepsilon(x):=\frac{x^2}{2}+c_0\cos(x/\sqrt{\varepsilon}).$$
For the interaction potential in (\ref{re05}), the authors considered the Curie-Weiss quadratic interaction potential $$F(x):=\frac{x^2}{2}.$$

Moreover, Delgadino at al.~in \cite{DGP} studied the multi-scale interacting particle system corresponding to the following McKean-Vlasov equation
\begin{equation*}
d X_t^{\varepsilon} =\big(-\nabla V(X_t^{\varepsilon}/\sqrt{\varepsilon})-\theta \mathbb{E}\big[\nabla F\big((x-X_t^{\varepsilon})/\sqrt{\varepsilon}\big)\big]|_{x=X_t^{\varepsilon}}\big)dt+\sqrt{2\beta^{-1}}d W_t,
\end{equation*}
where $F:\mathbb{R}^{d}\rightarrow\mathbb{R}$ and $V:\mathbb{R}^{d}\rightarrow\mathbb{R}$ are smooth and 1-periodic interaction and confining potentials. In particular, the authors considered the following
potentials
$$V(x):=-\eta\cos(2\pi x),~~~F(x):=-\cos(2\pi x),~~0<\eta<1.$$
Motivated by the aforementioned models, we can investigate the following multi-scale McKean-Vlasov stochastic system
\begin{equation}\left\{\begin{array}{l}\label{res3}
\displaystyle
d X^{\varepsilon}_t = \Big[-\nabla V_1(X_t^{\varepsilon})- \nabla V_2(Y_t^{\varepsilon})-\theta \mathbb{E}\big[\nabla F_1(x-X_t^{\varepsilon})\big]|_{x=X_t^{\varepsilon}}\Big]dt+\sqrt{2\beta^{-1}}d W_t^1, \\
d Y^{\varepsilon}_t = \Big[-\frac{1}{\varepsilon}\nabla V_3(X_t^{\varepsilon})-\frac{1}{\varepsilon}\nabla V_4(Y_t^{\varepsilon})-\frac{\theta}{\varepsilon} \mathbb{E}\big[\nabla F_2(x-X_t^{\varepsilon})\big]|_{x=X_t^{\varepsilon}}\Big]dt+\sqrt{\frac{2\beta^{-1}}{\varepsilon}}d W_t^2,
\end{array}\right.
\end{equation}
where $V_i:\mathbb{R}^d\rightarrow\mathbb{R},i=1,\cdots4$, $F_j:\mathbb{R}^d\rightarrow\mathbb{R},j=1,2$, $\{W^{1}_t\}_{t\geq 0}$ and $\{W^{2}_t\}_{t\geq 0}$ are independent $d$-dimensional standard Brownian motions.

One concrete example satisfying the assumptions ${\mathbf{A\ref{A1}}}$-${\mathbf{A\ref{A4}}}$  for the confining potentials is
$$V_i(x):=\frac{x^2}{2}+c_0\cos(x),i=1,2,~V_3(x):=c_0\cos(x),~V_4(x):=\frac{x^2}{2}.$$
For the interaction potential, we can take the Curie-Weiss quadratic interaction potential as in \cite{GP} and the smooth, 1-periodic potential as in \cite{DGP}, i.e.
$$F_1(x):=\frac{x^2}{2},~~F_2(x):=\cos(2\pi x).$$

\end{example}

\begin{remark}\label{re11}
 In the system \eref{E2},
 when the diffusion coefficient $\sigma$ also depends on the fast component $Y$, the strong convergence ($L^2$-sense) of the averaging principle in general does not hold for the slow-fast system (see Section 4.1 in \cite{Liud10} for a counterexample), and in this work we will use the strong convergence result (Theorem 2.5 in \cite{RSX1}) in our proof, thus the diffusion coefficient $\sigma$ considered in this work does not depend on $Y$.
 Here, the coefficients do not depend on the law of $Y$ mainly because of the technical difficulties in the proof. In fact, if the coefficients $f,g$ of the fast component depend on the law of $Y$, then the corresponding frozen equation has the following form
 $$dY_t=f(x,\mu,Y_t,\mathscr{L}_{Y_t})dt+g(x,\mu,Y_t,\mathscr{L}_{Y_t})d\tilde{W}^2_t,$$
thus the solution is a  time-inhomogeneous Markov process. However, the techniques used in this paper strongly rely on the time-homogeneous property.
\end{remark}

\subsection{Central limit type theorem}
In this subsection, we investigate the central limit type theorem for slow-fast McKean-Vlasov stochastic system. Recall that $(X^{\varepsilon},Y^{\varepsilon})$ is the solution of system \eref{E2} with $\delta=1$.
The authors \cite{RSX1} have proved the following error estimate of the averaging principle for stochastic system \eref{E2},
\begin{eqnarray}\label{33}
\sup_{t\in[0,T]}\EE|X^{\varepsilon}_t-\bar{X}_t|^2\leq C_T\big(1+\EE|\xi|^4+\EE|\zeta|^4\big)\varepsilon.
\end{eqnarray}
In this paper, we want to further analyze the deviation between the slow component $X^{\varepsilon}$ and the solution $\bar{X}$ of the averaged equation, more precisely, we show that the process
$$Z_t^\varepsilon:=\frac{X^{\varepsilon}_t-\bar{X}_t}{\sqrt{\varepsilon}}$$
converges weakly to a limiting process $Z$ as $\varepsilon\to 0$ in $C([0,T];\RR^n)$.

In order to state the main result, we introduce the following Poisson equation which depends on parameters $(x,\mu)\in \RR^n\times \mathscr{P}_2$,
\begin{equation}
-\mathcal{L}_{2}(x,\mu)\Phi(x,\mu,y)=b(x,\mu,y)-\bar{b}(x,\mu),\label{PE}
\end{equation}
where $\Phi(x,\mu,y):=(\Phi_1(x,\mu,y),\ldots, \Phi_n(x,\mu,y))$,
$$\mathcal{L}_{2}(x,\mu)\Phi(x,\mu,y):=(\mathfrak{L}_{2}(x,\mu)\Phi_1(x,\mu,y),\ldots, \mathfrak{L}_{2}(x,\mu)\Phi_n(x,\mu,y))$$
and for any $k=1,\ldots,n,$
\begin{eqnarray}\label{inf1}
\mathfrak{L}_{2}(x,\mu)\Phi_k(x,\mu,y):=\langle f(x,\mu,y), \partial_y \Phi_k(x,\mu,y)\rangle+\frac{1}{2}\text{Tr}[g g^{*}(x,\mu,y)\partial^2_{yy} \Phi_k(x,\mu,y)].
\end{eqnarray}
According to \cite[Proposition 4.1]{RSX1}, under the conditions ${\mathbf{A\ref{A1}}}$ and ${\mathbf{A\ref{A2}}}$, (\ref{PE}) admits a unique solution $\Phi(x,\mu,y)$ satisfying $\Phi(\cdot,\mu,\cdot)\in C^{2,2}(\RR^n\times\RR^m;\RR^n)$ and $\Phi(x,\cdot,y)\in C^{(1,1)}(\mathscr{P}_2; \RR^n)$. We denote
\begin{eqnarray}
\partial_y\Phi_g(x,\mu,y)u:=\partial_y\Phi(x,\mu,y)\cdot g(x,\mu,y)u,\quad u\in\RR^{d_2}, \label{partial Phig}
\end{eqnarray}
then $\partial_y\Phi_g(x,\mu,y)\in \RR^{n}\otimes\RR^{d_2}$ and
\begin{equation}\label{14}
\overline{(\partial_y\Phi_g)(\partial_y\Phi_g)^*}(x,\mu):=\int_{\RR^m}(\partial_y\Phi_g(x,\mu,y))(\partial_y\Phi_g(x,\mu,y))^*\nu^{x,\mu}(dy).
\end{equation}

\vskip 0.2cm
Now we state our first main result, which could be viewed as a central limit type theorem.

\begin{theorem}\label{main result 2}
Suppose that conditions  ${\mathbf{A\ref{A1}}}$-${\mathbf{A\ref{A3}}}$ hold.
Then for any initial values $\xi\in L^6(\Omega;\RR^n)$, $\zeta\in L^6(\Omega;\RR^m)$,   $\{Z^{\varepsilon}\}_{\varepsilon>0}$ converges weakly in $C([0,T];\RR^n)$  to the solution of following equation as $\varepsilon\to 0$,
\begin{eqnarray}\label{e5}
dZ_t= \!\!\!\!\!\!\!\!&& \partial_x\bar{b}(\bar{X}_t,\mathscr{L}_{\bar{X}_t})\cdot Z_tdt
+\EE\Big[\partial_{\mu} \bar{b}(u,\mathscr{L}_{\bar{X}_t})(\bar{X}_t)\cdot Z_t\Big]\Big|_{u=\bar{X}_t}dt+\left[\partial_x\sigma(\bar{X}_t,\mathscr{L}_{\bar{X}_t})\cdot Z_t\right]dW_t^1
 \nonumber\\
 \!\!\!\!\!\!\!\!&& + \EE\Big[\partial_{\mu} \sigma(u,\mathscr{L}_{\bar{X}_t})(\bar{X}_t)\cdot Z_t\Big]\Big|_{u=\bar{X}_t}dW_t^1+\Theta(\bar{X}_t,\mathscr{L}_{\bar{X}_t})dW_t,~Z_0=0,
\end{eqnarray}
where $\Theta(x,\mu):=\Big(\overline{(\partial_y\Phi_g)(\partial_y\Phi_g)^*}\Big)^{\frac{1}{2}}(x,\mu)$, $\bar{X}$ is the solution of equation \eref{1.3} and $\{W_t\}_{t\geq 0}$ is a $n$-dimensional standard Brownian motion which is defined on probability space $(\Omega, \mathscr{F},\mathbb{P})$ and independent of $\{W^1_t\}_{t\geq 0}$.
\end{theorem}

\begin{remark}
In contrast to the existing works \cite{C09,K2},  where the authors considered the slow-fast system with vanishing  noise in the slow equation and the fast equation does not depend on slow variable. In the present work, we consider more general case, i.e. both slow and fast components are fully coupled and driven by multiplicative noise, and meanwhile the coefficients also depend on the distribution. Moreover, we want to point out that in \cite{C09, K2} the extra term in the limiting process is a Gaussian process characterized by its mean and covariance, here  we provide explicit representation of this extra term in (\ref{e5}) by using the solution to Poisson equation (\ref{PE}).
\end{remark}

\subsection{Large deviation principle}

For reader's convenience, we recall some necessary notations. Let us define a standard Brownian motion $W_t$ on
$\RR^{d_1+ d_2}$,
$$W_t:=\sum_{k=1}^{d_1+d_2}\beta_k(t)e_k,$$
where $\{e_k\}_{1\leq k\leq d_1+d_2}$ is the canonical basis of $\RR^{d_1+ d_2}$, $\{\beta_k\}_{1\leq k\leq d_1+d_2}$ is a sequence of independent one-dimensional standard Brownian motions. Then we can choose the projection operators $P_1:\RR^{d_1+ d_2}\to \RR^{d_1}$, $P_2:\RR^{d_1+ d_2}\to \RR^{d_2}$ such that
\begin{equation}\label{bm}
W_t^1:=P_1W_t,~W_t^2:=P_2W_t.
\end{equation}
The Cameron-Martin space $\mathcal{H}_0$ associated with $\{W_t,t\in[0,T]\}$ is given as follows
\begin{eqnarray}\label{CM1}
\mathcal{H}_0:=
\!\!\!\!\!\!\!\!&&\Bigg\{h:[0,T]\to\RR^{d_1+ d_2}\Big|~\text{there exists}~\dot{h}\in L^2([0,T];\RR^{d_1+d_2})
\nonumber \\
\!\!\!\!\!\!\!\!&&~~~~~\text{such that}~ h_t=\int_0^t\dot{h}_s d s,t\in[0,T] \Bigg\},
\end{eqnarray}
where the $\dot{h}$ denotes the weak derivative of $h$, then $\mathcal{H}_0$ is a Hilbert space with the scalar product
$$\langle h^1, h^2\rangle_0 :=\int_0^T\langle\dot{ h}^1_t,\dot{ h}^2_t\rangle d t .$$

The main result of LDP is formulated as follows.
\begin{theorem}\label{t3}
Suppose that the conditions ${\mathbf{A\ref{A1}}}$, ${\mathbf{A\ref{A4}}}$ and ${\mathbf{A\ref{A5}}}$ hold.
Then for any initial values $x\in\RR^n,y\in\RR^m$, $\{X^{\delta,\varepsilon}\}_{\delta>0}$
satisfies the LDP in $C([0,T]; \RR^n)$ with the
good rate function $I$ given by
\begin{equation}\label{rf1}
I(f)=\inf_{\left\{h\in \mathcal{H}_0:\  f=\mathcal{G}^0(h)\right\}}\left\lbrace\frac{1}{2}
\int_0^T|\dot{h}_t|^2 d t \right\rbrace,
\end{equation}
where infimum over an empty set is taken as $+\infty$, and map $\mathcal{G}^0: C([0,T]; \RR^{d_1+d_2})\rightarrow C([0,T]; \RR^n)$ is defined by
\begin{equation*}\label{g1}
\mathcal{G}^0(h):=\left\{ \begin{aligned}
&\bar{X}^{h},~~h\in\mathcal{H}_0,\\
&0,~~~~\text{otherwise}.
\end{aligned} \right.
\end{equation*}
Here $\bar{X}^{h}$ is the solution of (\ref{ske}).

\end{theorem}

\begin{remark}
In order to prove the LDP result for the multi-scale McKean-Vlasov SDEs (\ref{E2}), we want to employ the weak convergence approach developed by Budhiraja and Dupuis (cf.~\cite{BD,BDM}). However, once we consider the McKean-Vlasov SDEs, the classical Yamada-Watanabe theorem  is not directly applicable (cf.~\cite{WFY}). Thus a key step is to find a measurable map such that the solution could be represented by a functional of Brownian motions (see subsection \ref{sec5.1} for the details).
\end{remark}

\section{Proof of central limit type theorem}
\setcounter{equation}{0}
 \setcounter{definition}{0}
In this section, we give a detailed proof of Theorem \ref{main result 2}.  To do this, we first derive some crucial a priori estimates for process  $(X^{\varepsilon}, Y^{\varepsilon})$, and establish some regularity result of the solution to the Poisson equation (\ref{PE})(see subsection \ref{S4.0}). Secondly, we construct an auxiliary $\eta^{\varepsilon}$ and show that  the difference $Z^\varepsilon-\eta^{\varepsilon}$ converges to $0$ in $C([0,T],\RR^n)$, as $\varepsilon\rightarrow 0$ (see subsection \ref{S4.1}). Thirdly, we show that $\eta^{\varepsilon}$ is tight in $C([0,T],\RR^n)$ (see subsection \ref{S4.2}). Finally, we identify the weak limiting process of any subsequence $\{\eta^{\varepsilon_n}\}_{n\geq 1}$ with $\varepsilon_n \downarrow 0$ (see subsection \ref{S4.3}).

 \subsection{A priori estimates}\label{S4.0}
\begin{lemma} \label{PMY}
Suppose that the condition ${\mathbf{A\ref{A1}}}$ holds. For any $T>0$ with $0\leq t\leq t+h\leq T$, there exists  a constant $ C_T>0$ such that
\begin{eqnarray}
&&\sup_{\varepsilon\in(0,1)}\mathbb{E}\left[\sup_{t\in [0, T]}|X_{t}^{\varepsilon}|^{6}\right]\leq C_{T}(1+\EE|\xi|^{6}+\EE|\zeta|^{6});\label{X}\\
&&\sup_{\varepsilon\in(0,1)}\sup_{t\in [0, T]}\mathbb{E}|Y_{t}^{\varepsilon}|^{6}\leq C_{T}(1+\EE|\xi|^{6}+\EE|\zeta|^{6});\label{Y}\\
&&\sup_{\varepsilon\in(0,1)}\mathbb{E}|X_{t+h}^{\varepsilon}-X_{t}^{\varepsilon}|^{6}\leq C_{T}(1+\EE|\xi|^6+\EE|\zeta|^6)h^3.\label{COX}
\end{eqnarray}
Furthermore, if the condition (\ref{32}) holds, then there exists $C>0$ such that for small enough $\ep$ we have
\begin{eqnarray}\label{Y1}
\mathbb{E}\left[\sup_{t\in [0, T]}|Y_{t}^{\varepsilon}|^{4}\right]\leq \frac{C(1+\EE|\zeta|^4)T}{\ep}.
\end{eqnarray}
\end{lemma}

\begin{proof}
By a small modification in \cite[Lemmas 3.1, 3.2]{RSX1}, we can easily prove \eref{X}, \eref{Y} and \eref{COX}. Here we only need to prove (\ref{Y1}).

We consider process $\tilde{Y}^\varepsilon_t:=Y^{\varepsilon}_{\varepsilon t}$ that solves the following equation
$$d \tilde{Y}^\varepsilon_t=f(X_{t\varepsilon}^\varepsilon,\mathscr{L}_{X_{t\varepsilon}^\varepsilon},\tilde{Y}^\varepsilon_t)dt+g(X_{t\varepsilon}^\varepsilon,\mathscr{L}_{X_{t\varepsilon}^\varepsilon},\tilde{Y}^\varepsilon_t)d\tilde{W}^2_t,~\tilde{Y}^\varepsilon_0=\zeta,$$
where $\tilde{W}^2_t:=\frac{1}{\sqrt{\varepsilon}}W^2_{t\varepsilon}$ that coincides in law with $W^2_t$.

By It\^{o}'s formula, we have
\begin{eqnarray}
|\tilde{Y}_{t}^{\varepsilon}|^{4}=\!\!\!\!\!\!\!\!&&|\zeta|^{4}+4\int_{0} ^{t}|\tilde{Y}_{s}^{\varepsilon}|^{2}\langle f(X_{s\ep}^{\ep},\mathscr{L}_{X_{s\ep}^{\ep}},\tilde{Y}_{s}^{\ep}),\tilde{Y}_{s}^{\ep}\rangle ds+4\int_{0} ^{t}|\tilde{Y}_{s}^{\varepsilon}|^{2}\langle \tilde{Y}_{s}^{\ep}, g(X_{s\ep}^{\ep},\mathscr{L}_{X_{s\ep}^{\ep}},\tilde{Y}_{s}^{\ep})d\tilde{W}^2_s\rangle  \nonumber\\
&&+2\int_{0} ^{t}|\tilde{Y}_{s}^{\ep}|^{2}\|g(X_{s\ep}^{\ep}, \mathscr{L}_{X_{s\ep}^{\ep}}, \tilde{Y}_{s}^{\ep})\|^2ds\!+4\int_{0} ^{t}|g^{*}(X^{\varepsilon}_{s\varepsilon},\mathscr{L}_{X^{\varepsilon}_{s\varepsilon}}, \tilde{Y}^{\varepsilon}_{s})\tilde{Y}^{\varepsilon}_{s}|^2 ds.\label{ITO}
\end{eqnarray}
Taking expectation on both sides of (\ref{ITO}), we get
\begin{eqnarray*}
\frac{d}{dt}\mathbb{E}|\tilde{Y}_{t}^{\varepsilon}|^{4}= \!\!\!\!\!\!\!\!&&   4\mathbb{E}\Big[|\tilde{Y}_{t}^{\varepsilon}|^{2}\langle f(X_{t\ep}^{\ep},\mathscr{L}_{X_{t\ep}^{\ep}},\tilde{Y}_{t}^{\ep}),\tilde{Y}_{t}^{\ep}\rangle\Big]
+2\mathbb{E}\Big[|\tilde{Y}_{t}^{\ep}|^{2}\|g(X_{t\ep}^{\ep}, \mathscr{L}_{X_{t\ep}^{\ep}}, \tilde{Y}_{t}^{\ep})\|^2\Big]\nonumber\\
\!\!\!\!\!\!\!\!&&+4\mathbb{E}|g^{*}(X^{\varepsilon}_{t\varepsilon},\mathscr{L}_{X^{\varepsilon}_{t\varepsilon}}, \tilde{Y}^{\varepsilon}_{t})\tilde{Y}^{\varepsilon}_{t}|^2.
\end{eqnarray*}

In view of the conditions \eref{sm} and \eref{32}, we can deduce that for any $\beta\in(0,\gamma)$, there exists $C_{\beta}>0$,
\begin{eqnarray}\label{RE3}
 2\langle f(x,\mu,y), y\rangle+3\|g(x,\mu, y)\|^2\leq -\beta |y|^2+C_{\beta}.
\end{eqnarray}
Note that
$$|g^{*}(X^{\varepsilon}_{t\varepsilon},\mathscr{L}_{X^{\varepsilon}_{t\varepsilon}}, \tilde{Y}^{\varepsilon}_{t})\tilde{Y}^{\varepsilon}_{t}|^2 \leq|\tilde{Y}_{t}^{{\varepsilon}}|^{2}\|g(X_{t{\varepsilon}}^{{\varepsilon}}, \mathscr{L}_{X_{t{\varepsilon}}^{{\varepsilon}}}, \tilde{Y}_{t}^{{\varepsilon}})\|^2.$$
Then by (\ref{RE3}) and Young's inequality, we have
\begin{eqnarray}
\frac{d}{dt}\mathbb{E}|\tilde{Y}_{t}^{\varepsilon}|^{4}\leq\!\!\!\!\!\!\!\!&&
2\mathbb{E}\Big[|\tilde{Y}_{t}^{{\varepsilon}}|^{2}\Big(\langle 2f(X_{t{\varepsilon}}^{{\varepsilon}},\mathscr{L}_{X_{t{\varepsilon}}^{{\varepsilon}}},\tilde{Y}_{t}^{{\varepsilon}}),\tilde{Y}_{t}^{{\varepsilon}}\rangle +3\|g(X_{t{\varepsilon}}^{{\varepsilon}}, \mathscr{L}_{X_{t{\varepsilon}}^{{\varepsilon}}}, \tilde{Y}_{t}^{{\varepsilon}})\|^2\Big)\Big]
\nonumber\\\leq\!\!\!\!\!\!\!\!&&-2\beta\mathbb{E}|\tilde{Y}_{t}^{{\varepsilon}}|^{4}+C_{\beta}\mathbb{E}|\tilde{Y}_{t}^{{\varepsilon}}|^{2} \nonumber\\\leq\!\!\!\!\!\!\!\!&&-\beta\mathbb{E}|\tilde{Y}_{t}^{\varepsilon}|^{4}+C_{\beta}. \nonumber
\end{eqnarray}
The Gronwall's lemma implies that
\begin{eqnarray*}
\mathbb{E}|\tilde{Y}_{t}^{\varepsilon}|^{4}\leq \EE|\zeta|^{4}e^{-\beta t}+C\int^t_0 e^{-\beta(t-s)}ds\leq C(1+\EE|\zeta|^{4}).
\end{eqnarray*}
Hence we get
\begin{eqnarray}
\sup_{t\geq 0}\EE|\tilde{Y}^{\varepsilon}_{t}|^{4}\leq C(1+\EE|\zeta|^4).\label{Prior E}
\end{eqnarray}

Note that \eref{ITO} and \eref{RE3} imply that
\begin{eqnarray}
|\tilde{Y}_{t}^{\varepsilon}|^{4}\leq |\zeta|^{4}+Ct+C\left|\int_{0} ^{t}|\tilde{Y}_{s}^{\varepsilon}|^{2}\langle \tilde{Y}_{s}^{\ep}, g(X_{s\ep}^{\ep},\mathscr{L}_{X_{s\ep}^{\ep}},\tilde{Y}_{s}^{\ep})d\tilde{W}^2_s\rangle\right|.
\end{eqnarray}
By Burkholder-Davis-Gundy's inequality and Young's inequality we get
\begin{eqnarray*}
\mathbb{E}\Big[\sup_{t\in[0,T]}|\tilde{Y}_{t}^{\varepsilon}|^{4}\Big]\leq\!\!\!\!\!\!\!\!&&\EE|\zeta|^{4}+CT +C\EE\left[\int_0^T|\tilde{Y}^\varepsilon_s|^{6}(|\tilde{Y}^\varepsilon_s|^{2}+1)ds\right]^{\frac{1}{2}}
\nonumber\\
\leq\!\!\!\!\!\!\!\!&&\EE|\zeta|^{4}+CT +C\EE\left[\sup_{s\in [0,T]}|\tilde{Y}^\varepsilon_s|^{4}\cdot\int_0^T|\tilde{Y}^\varepsilon_s|^{2}(|\tilde{Y}^\varepsilon_s|^{2}+1)ds\right]^{\frac{1}{2}}
\nonumber\\
\leq\!\!\!\!\!\!\!\!&&\EE|\zeta|^{4}+CT +\frac{1}{2}\mathbb{E}\left[\sup_{t\in[0,T]}|\tilde{Y}_{t}^{\varepsilon}|^{4}\right]+\frac{C}{2}\EE\int_0^T|\tilde{Y}^\varepsilon_s|^{2}(|\tilde{Y}^\varepsilon_s|^{2}+1)ds
\nonumber\\
\leq\!\!\!\!\!\!\!\!&&\EE|\zeta|^{4}+CT+\frac{1}{2}\mathbb{E}\left[\sup_{t\in[0,T]}|\tilde{Y}_{t}^{\varepsilon}|^{4}\right]+C\int_0^T\left(\EE|\tilde{Y}^\varepsilon_s|^{4}+1\right)ds.
\end{eqnarray*}
For the last term above, by \eref{Prior E}, we have
$$
C\int_0^T\left(\EE|\tilde{Y}^\varepsilon_s|^{4}+1\right)ds\leq C(1+\EE|\zeta|^4)T.
$$
Hence, we obtain that for any $T\geq 1$,
\begin{eqnarray*}
\mathbb{E}\Big[\sup_{t\in[0,T]}|\tilde{Y}_{t}^{\varepsilon}|^{4}\Big]\leq C(1+\EE|\zeta|^4)T.
\end{eqnarray*}
Hence, it follows that for any $T>0$ and $\varepsilon$ small enough,
\begin{eqnarray*}
\mathbb{E}\Big[\sup_{t\in[0,T]}|Y_{t}^{\varepsilon}|^{4}\Big]=\mathbb{E}\Big[\sup_{t\in\left[0,\frac{T}{\varepsilon}\right]}|\tilde{Y}_{t}^{\varepsilon}|^{4}\Big]\leq\frac{C(1+\EE|\zeta|^4)T}{\ep},
\end{eqnarray*}
which yields (\ref{Y1}).  Hence the proof is complete.       \hspace{\fill}$\Box$
\end{proof}

\begin{remark}
Instead of estimating the fourth moment of solutions in \cite{RSX1}, we here estimate the sixth moment of solutions, since it will be used in the proof of Lemma \ref{L4.5} below. The two estimates \eref{Y} and \eref{Y1} are an exchange of $\sup_{t\in [0,T]}$ and expectation essentially, while their boundedness and proofs are quite different, such as the latter needs additional condition (\ref{32}).
Moreover, we point out that the estimate (\ref{Y1}) is essential in studying the weak limit of  $Z^\varepsilon$ in $C([0,T];\mathbb{R}^n)$ (see Lemma \ref{th1} and Proposition \ref{th2}).
\end{remark}

Recall the frozen equation for any fixed $x\in\RR^n$ and $\mu\in\mathscr{P}_2$,
\begin{equation}\left\{\begin{array}{l}\label{FEQ2}
\displaystyle
dY_{t}=f(x, \mu, Y_{t})dt+g(x, \mu, Y_{t})d\tilde{W}_{t}^{2},\\
Y_{0}=y\in\RR^m \\
\end{array}\right.
\end{equation}
where $\{\tilde{W}_{t}^{2}\}_{t\geq 0}$ is a $d_2$-dimensional Brownian motion on another complete probability space $(\tilde{\Omega}, \tilde{\mathscr{F}}, \tilde{\mathbb{P}})$. Under the conditions \eref{A21} and \eref{sm}, we obtain that
$(\ref{FEQ2})$ has a unique strong solution $\{Y_{t}^{x,\mu,y}\}_{t\geq 0}$ and admits a unique invariant measure $\nu^{x,\mu}$ (cf.~\cite[Theorem 4.3.9]{LR1}).

In order to prove our first main result, we need to study the regularity of the solution of Poisson equation (\ref{PE}) w.r.t.~parameters $(x,\mu)$.

\begin{proposition}\label{P3.6}
Suppose that ${\mathbf{A\ref{A1}}}$ and ${\mathbf{A\ref{A2}}}$ hold. Define
\begin{eqnarray}
\Phi(x,\mu,y)=\int^{\infty}_{0} \left[\tilde\EE b(x,\mu,Y^{x,\mu,y}_s)-\bar{b}(x,\mu)\right]ds,\label{SPE}
\end{eqnarray}
where $\tilde{\EE}$ is the expectation on $(\tilde{\Omega}, \tilde{\mathscr{F}}, \tilde{\mathbb{P}})$. Then $\Phi(x,\mu,y)$ is the unique solution of \eref{PE},  $\Phi\in C^{2,(1,1),2}(\RR^n\times\mathscr{P}_2 \times\RR^m; \RR^n)$ and $\partial_{\mu}\partial_y\Phi(x,\mu,y)$ exists for any $(x,\mu,y)\in \RR^n\times\mathscr{P}_2\times\RR^m $. Moreover, there exists $C>0$ such that
\begin{eqnarray}\label{E1}
&&\!\!\!\!\!\!\!\!\max\left\{|\Phi(x,\mu,y)|,\|\partial_x \Phi(x,\mu,y)\|, \|\partial_{\mu}\Phi(x,\mu,y)\|_{L^2(\mu)},\|\partial^2_{xx} \Phi(x,\mu,y)\|\right.\nonumber\\
&& \quad\left. \|\partial_{z}\partial_{\mu}\Phi(x,\mu,y)(\cdot)\|_{L^2(\mu)}\right \}\leq C\left \{1+|x|+|y|+[\mu(|\cdot|^2)]^{1/2}\right\}.~
\end{eqnarray}
and
\begin{eqnarray}\label{E3}
&&\!\!\!\!\!\!\!\!\sup_{x\in\RR^n,\mu\in\mathscr{P}_2, y\in \RR^m}\max\left\{\|\partial_y \Phi(x,\mu,y)\|,\|\partial^2_{yy} \Phi(x,\mu,y)\|,\|\partial^2_{xy} \Phi(x,\mu,y)\|, \right.\nonumber\\
&&\quad\quad\quad\quad\quad\quad \quad\quad \left.\|\partial_{\mu}\partial_y\Phi(x,\mu,y)\|_{L^2(\mu)}\right\}\leq C.
\end{eqnarray}

\end{proposition}

\begin{proof}
Note that $\mathcal{L}_{2}(x,\mu)$ is the infinitesimal generator of the frozen process $\{Y^{x,\mu}_t\}_{t\geq 0}$, we can easily check  that \eref{SPE} is the unique solution of Poisson equation \eref{PE} under the assumptions ${\mathbf{A\ref{A1}}}$ and ${\mathbf{A\ref{A2}}}$. Moreover, by a straightforward computation and following a similar argument as used in \cite[section 6]{BLPR}, we can prove that $\Phi\in C^{2,(1,1),2}(\RR^n\times\mathscr{P}_2\times\RR^m,\RR^n)$ and $\partial_{\mu}\partial_y\Phi(x,\mu,y)$ exists for any $(x,\mu,y)\in \RR^n\times\mathscr{P}_2\times\RR^m $. We omit the detailed proofs for its lengthy.

The proof of  \eref{E1} and the first estimate in \eref{E3} could be founded in \cite[Proposition 4.1]{RSX1}. We only show the remaining estimates in (\ref{E3}),
which will be proved by the following three steps.

(i) In view of (\ref{SPE}), we deduce that for any $h,k\in \RR^m$,
$$\partial_y\Phi(x,\mu,y)\cdot h=\int_0^{\infty}\tilde \EE\left[\partial_yb(x,\mu,Y^{x,\mu,y}_t)\cdot\big(\partial_yY^{x,\mu,y}_t\cdot h\big)\right]dt,$$
which implies that
\begin{eqnarray*}
\partial^2_{yy}\Phi(x,\mu,y)\cdot(h,k)=\!\!\!\!\!\!\!\!&&\int_0^{\infty}\tilde \EE\big[\partial^2_{yy}b(x,\mu,Y^{x,\mu,y}_t)\cdot\big(\partial_yY^{x,\mu,y}_t\cdot h, \partial_yY^{x,\mu,y}_t\cdot k\big)\nonumber\\
&&+\partial_yb(x,\mu,Y^{x,\mu,y}_t)\cdot(\partial^2_{yy}Y^{x,\mu,y}_t\cdot(h,k))\big]dt,
\end{eqnarray*}
where $\partial_yY^{x,\mu,y}_t\cdot h$ and $\partial^2_{yy}Y^{x,\mu,y}_t\cdot(h,k)$ fulfill the following two equations respectively,
\begin{equation*}
\left\{ \begin{aligned}
d\partial_{y}Y^{x,\mu,y}_t\cdot h=&~\big[\partial_{y}f(x,\mu,Y^{x,\mu,y}_t)\cdot \left(\partial_yY^{x,\mu,y}_t\cdot h\right)\big]dt\\
&+\left[\partial_yg(x,\mu,Y^{x,\mu,y}_t)\cdot\left(\partial_{y}Y^{x,\mu,y}_t \cdot h\right) \right]d\tilde{W}^2_t,\\
\partial_{y}Y^{x,\mu,y}_0\cdot h=&~h,
\end{aligned} \right.
\end{equation*}
and
\begin{equation*}
\left\{ \begin{aligned}
d\partial^2_{yy}Y^{x,\mu,y}_t\cdot(h,k)=&~\big[\partial^2_{yy}f(x,\mu,Y^{x,\mu,y}_t)\cdot (\partial_yY^{x,\mu,y}_t\cdot h, \partial_yY^{x,\mu,y}_t\cdot k)\\
&+ \partial_yf(x,\mu,Y^{x,\mu,y}_t)\cdot\left(\partial^2_{yy}Y^{x,\mu,y}_t \cdot(h,k)\right) \big]dt\\
&+\left[\partial^2_{yy}g(x,\mu,Y^{x,\mu,y}_t)\cdot\big(\partial_yY^{x,\mu,y}_t\cdot h, \partial_yY^{x,\mu,y}_t\cdot k\big)\right.\\
&+\left.\partial_yg(x,\mu,Y^{x,\mu,y}_t)\cdot\left(\partial^2_{yy}Y^{x,\mu,y}_t \cdot(h,k)\right) \right]d\tilde{W}^2_t,\\
\partial^2_{yy}Y^{x,\mu,y}_0\cdot(h,k)=&~0.
\end{aligned} \right.
\end{equation*}
Under the conditions ${\mathbf{A\ref{A1}}}$ and ${\mathbf{A\ref{A2}}}$, by a straightforward computation, we can prove that
\begin{eqnarray*}
&&\sup_{x\in\RR^n,\mu\in\mathscr{P}_2,y\in\RR^m}\tilde\EE|\partial_yY^{x,\mu,y}_t\cdot h|^2\leq Ce^{-\beta t}|h|^2,\label{9}\\
&&\sup_{x\in\RR^n,\mu\in\mathscr{P}_2,y\in\RR^m}\tilde\EE|\partial^2_{yy}Y^{x,\mu,y}_t\cdot(h,k)|^2\leq Ce^{-\beta t}|h|^2|k|^2,\label{10}
\end{eqnarray*}
which together with the boundedness of $\|\partial_yb\|$ and $\|\partial^2_{yy}b\|$, it follows that
$$\sup_{x\in\RR^n,\mu\in\mathscr{P}_2,y\in\RR^m}|\partial_{yy}^2 \Phi(x,\mu,y)\cdot(h,k)|\leq C|h||k|.$$

(ii) Similarly, for any $h,k\in \RR^m$, $\partial^2_{xy}\Phi(x,\mu,y)\cdot(h,k)$ could be represented by
\begin{eqnarray*}
\partial^2_{xy}\Phi(x,\mu,y)\cdot(h,k)=\!\!\!\!\!\!\!\!&&\int_0^{\infty}\tilde \EE\big[\partial_{xy}b(x,\mu,Y^{x,\mu,y}_t)\cdot\big(\partial_yY^{x,\mu,y}_t\cdot h, k\big)\\
&&\quad+\partial^2_{yy}b(x,\mu,Y^{x,\mu,y}_t)\cdot\big(\partial_yY^{x,\mu,y}_t\cdot h, \partial_yY^{x,\mu,y}_t\cdot k\big)\nonumber\\
&&\quad+\partial_yb(x,\mu,Y^{x,\mu,y}_t)\cdot(\partial^2_{xy}Y^{x,\mu,y}_t\cdot(h,k))\big]dt,
\end{eqnarray*}
where $\partial^2_{xy}Y^{x,\mu,y}_t\cdot(h,k)$ fulfills the following equation
\begin{equation*}
\left\{ \begin{aligned}
d\partial^2_{xy}Y^{x,\mu,y}_t\cdot(h,k)=&~\big[\partial^2_{xy}f(x,\mu,Y^{x,\mu,y}_t)\cdot (\partial_yY^{x,\mu,y}_t\cdot h, k)\\
&+\partial^2_{yy}f(x,\mu,Y^{x,\mu,y}_t)\cdot (\partial_yY^{x,\mu,y}_t\cdot h, \partial_xY^{x,\mu,y}_t\cdot k)\\
&+ \partial_yf(x,\mu,Y^{x,\mu,y}_t)\cdot\left(\partial^2_{xy}Y^{x,\mu,y}_t \cdot(h,k)\right) \big]dt\\
&+\left[\partial^2_{xy}g(x,\mu,Y^{x,\mu,y}_t)\cdot\big(\partial_yY^{x,\mu,y}_t\cdot h, \partial_yY^{x,\mu,y}_t\cdot k\big)\right.\\
&+\partial^2_{yy}g(x,\mu,Y^{x,\mu,y}_t)\cdot\big(\partial_yY^{x,\mu,y}_t\cdot h, \partial_xY^{x,\mu,y}_t\cdot k\big)\\
&+\left.\partial_yg(x,\mu,Y^{x,\mu,y}_t)\cdot\left(\partial^2_{xy}Y^{x,\mu,y}_t \cdot(h,k)\right) \right]d\tilde{W}^2_t,\\
\partial^2_{xy}Y^{x,\mu,y}_0\cdot(h,k)=0&.
\end{aligned} \right.
\end{equation*}
Under the conditions ${\mathbf{A\ref{A1}}}$ and ${\mathbf{A\ref{A2}}}$, by a straightforward computation (see \cite[Proposition 4.1]{RSX1}), we can see that
\begin{eqnarray*}
&&\sup_{x\in\RR^n,\mu\in\mathscr{P}_2,y\in\RR^m}\tilde\EE|\partial_xY^{x,\mu,y}_t\cdot h|^2\leq C|h|^2,\label{11}\\
&&\sup_{x\in\RR^n,\mu\in\mathscr{P}_2,y\in\RR^m}\tilde\EE|\partial^2_{xy}Y^{x,\mu,y}_t\cdot(h,k)|^2\leq Ce^{-\beta t}|h|^2|k|^2,\label{12}
\end{eqnarray*}
then by the boundedness of  $\|\partial_yb\|$, $\|\partial^2_{xy}b\|$ and $\|\partial^2_{yy}b\|$, we deduce that
$$\sup_{x\in\RR^n,\mu\in\mathscr{P}_2,y\in\RR^m}|\partial^2_{xy}\Phi(x,\mu,y)\cdot(h,k)|\leq C|h||k|.$$

(iii) For the term $\partial_{\mu}\partial_y\Phi(x,\mu,y)$, under the conditions ${\mathbf{A\ref{A1}}}$ and ${\mathbf{A\ref{A2}}}$, one can easily obtain that for any $\mu_1,\mu_2\in \mathscr{P}_2$,
\begin{eqnarray}
&&\sup_{t\geq 0, x\in\RR^n,y\in\RR^m}\tilde\EE|Y^{x,\mu_1,y}_t-Y^{x,\mu_2,y}_t|^2\leq C\mathbb{W}_2(\mu_1,\mu_2)^2,\label{13}\\
&&\sup_{x\in\RR^n,y\in\RR^m}\tilde\EE\|\partial_{y}Y^{x,\mu_1,y}_t-\partial_{y}Y^{x,\mu_2,y}_t\|^2\leq Ce^{-\beta t}\mathbb{W}_2(\mu_1,\mu_2)^2,\nonumber
\end{eqnarray}
which combine with the boundedness of  $\|\partial_{y}b\|$, $\|\partial_{\mu}\partial_y b\|_{L^2(\mu)}$ and $\|\partial^2_{yy}b\|$, it follows that
\begin{eqnarray*}
\|\partial_y\Phi(x,\mu_1,y)-\partial_y\Phi(x,\mu_2,y)\|\leq \!\!\!\!\!\!\!\!&&\int_0^{\infty}\tilde \EE\big[\|\partial_yb(x,\mu_1,Y^{x,\mu_1,y}_t)-\partial_yb(x,\mu_2,Y^{x,\mu_2,y}_t)\|\|\partial_yY^{x,\mu_1,y}_t\|\\
\!\!\!\!\!\!\!\!&&+\|\partial_yb(x,\mu_2,Y^{x,\mu_2,y}_t)\|\|\partial_yY^{x,\mu_1,y}_t-\partial_yY^{x,\mu_2,y}_t\|\big]dt\\
\leq \!\!\!\!\!\!\!\!&& C\mathbb{W}_2(\mu_1,\mu_2).
\end{eqnarray*}
As a consequence, we have
$$\sup_{x\in\RR^n,\mu\in\mathscr{P}_2,y\in\RR^m}\|\partial_{\mu}\partial_y \Phi(x,\mu,y)\|_{L^2(\mu)}\leq C.$$
The proof is complete.    \hspace{\fill}$\Box$
\end{proof}

\vspace{0.1cm}
\subsection{Construction of the auxiliary process} \label{S4.1}
We recall that, by \eref{E2} and \eref{1.3},  the deviation process $\{Z_t^\varepsilon=\frac{X^{\varepsilon}_t-\bar{X}_t}{\sqrt{\varepsilon}}\}_{t\geq 0}$ satisfies that
\begin{equation*}
\left\{ \begin{aligned}
dZ_t^\varepsilon=&\frac{1}{\sqrt{\varepsilon}}\Big[b(X^{\varepsilon}_t,\mathscr{L}_{X^{\varepsilon}_t},Y^{\varepsilon}_t)-\bar{b}(X^{\varepsilon}_t,\mathscr{L}_{X^{\varepsilon}_t})\Big]dt
+\frac{1}{\sqrt{\varepsilon}}\Big[\bar{b}(X^{\varepsilon}_t,\mathscr{L}_{X^{\varepsilon}_t})-\bar{b}(\bar{X}_t,\mathscr{L}_{X^{\varepsilon}_t})\Big]dt\\
&+\frac{1}{\sqrt{\varepsilon}}\Big[\bar{b}(\bar{X}_t,\mathscr{L}_{X^{\varepsilon}_t})-\bar{b}(\bar{X}_t,\mathscr{L}_{\bar{X}_t})\Big]dt
+\frac{1}{\sqrt{\varepsilon}}\Big[\sigma(X^{\varepsilon}_t,\mathscr{L}_{X^{\varepsilon}_t})-\sigma(\bar{X}_t,\mathscr{L}_{X^{\varepsilon}_t})\Big]dW_t^1\\
&+\frac{1}{\sqrt{\varepsilon}}\Big[\sigma(\bar{X}_t,\mathscr{L}_{X^{\varepsilon}_t})-\sigma(\bar{X}_t,\mathscr{L}_{\bar{X}_t})\Big]dW_t^1,\\
Z_0^\varepsilon=&~0.
\end{aligned} \right.
\end{equation*}
Due to (\ref{33}), there exists a constant $C_T>0$ independent of $\varepsilon$ such that
\begin{eqnarray}
\sup_{t\in [0,T]}\EE|Z_t^\varepsilon|^2\leq C_T\big(1+\EE|\xi|^4+\EE|\zeta|^4\big).\label{Zvare}
\end{eqnarray}
Furthermore, under the condition \eref{32}, we can easily obtain a stronger result
\begin{eqnarray}
\EE\left[\sup_{t\in [0,T]}|Z_t^\varepsilon|^2\right]\leq C_T\big(1+\EE|\xi|^4+\EE|\zeta|^4\big),\label{supZvare}
\end{eqnarray}
whose proof will be given in Section \ref{app2} below.

Now for any $\varepsilon>0$ we introduce an auxiliary process $\eta_t^\varepsilon$,
\begin{equation}\label{e6}
\left\{ \begin{aligned}
d\eta_t^\varepsilon=&~\frac{1}{\sqrt{\varepsilon}}\Big[b(X^{\varepsilon}_t,\mathscr{L}_{X^{\varepsilon}_t},Y^{\varepsilon}_t)-\bar{b}(X^{\varepsilon}_t,\mathscr{L}_{X^{\varepsilon}_t})\Big]dt\\
&+\partial_x\bar{b}(\bar{X}_t,\mathscr{L}_{\bar{X}_t})\cdot\eta_t^\varepsilon dt+\EE\Big[\partial_{\mu} \bar{b}(u,\mathscr{L}_{\bar{X}_t})(\bar{X}_t)\cdot \eta_t^\varepsilon\Big]\Big|_{u=\bar{X}_t}dt\\
&+\left[\partial_x\sigma(\bar{X}_t,\mathscr{L}_{\bar{X}_t})\cdot\eta_t^\varepsilon\right] dW_t^1
 + \EE\Big[\partial_{\mu} \sigma(u,\mathscr{L}_{\bar{X}_t})(\bar{X}_t)\cdot \eta_t^\varepsilon\Big]\Big|_{u=\bar{X}_t}dW_t^1,\\
\eta_0^\varepsilon=&~0.
\end{aligned} \right.
\end{equation}
Then the difference $Z_t^\varepsilon-\eta_t^\varepsilon$ converges in probability to zero in $C([0,T];\RR^n)$, as $\varepsilon\to0 $, which is a consequence of the following Proposition.

\begin{proposition}\label{th3}
Suppose that conditions ${\mathbf{A\ref{A1}}}$ and ${\mathbf{A\ref{A2}}}$ hold, then  we have
\begin{eqnarray}
\lim_{\varepsilon\rightarrow 0}\EE\left[\sup_{t\in [0,T]}|Z_t^\varepsilon-\eta_t^\varepsilon|\right]=0.
\end{eqnarray}
\end{proposition}

\begin{proof}
Let $\rho^{\vare}_t=Z_t^\varepsilon-\eta_t^\varepsilon$, then we have
\begin{eqnarray*}
\rho^{\vare}_t=\!\!\!\!\!\!\!\!&&\frac{1}{\sqrt{\varepsilon}}\int_0^t\Big[\bar{b}(X^{\varepsilon}_s,\mathscr{L}_{X^{\varepsilon}_s})-\bar{b}(\bar{X}_s,\mathscr{L}_{X^{\varepsilon}_s})-\partial_x \bar{b}(\bar{X}_s,\mathscr{L}_{X^{\varepsilon}_s})\cdot\sqrt{\varepsilon}Z_s^\varepsilon\Big]ds\\
&&+\int_0^t\left[\partial_x\bar{b}(\bar{X}_s,\mathscr{L}_{X^{\varepsilon}_s})-\partial_x\bar{b}(\bar{X}_s,\mathscr{L}_{\bar{X}_s})\right]\cdot Z_s^\varepsilon ds\\
&&+\frac{1}{\sqrt{\varepsilon}}\int_0^t\Big[\bar{b}(\bar{X}_s,\mathscr{L}_{X^{\varepsilon}_s})-\bar{b}(\bar{X}_s,\mathscr{L}_{\bar{X}_s})-\EE\left[\partial_{\mu} \bar{b}(u,\mathscr{L}_{\bar{X}_s})(\bar{X}_s)\cdot\sqrt{\varepsilon}Z_s^\varepsilon\right]\mid_{u=\bar{X}_s}\Big]ds\\
&&+\int_0^t\partial_x\bar{b}(\bar{X}_s,\mathscr{L}_{\bar{X}_s})\cdot \rho^{\vare}_s ds+\int_0^t\EE\left[\partial_{\mu}\bar{b}(u,\mathscr{L}_{\bar{X}_s})(\bar{X}_s)\cdot \rho^{\vare}_s\right]|_{u=\bar{X}_s} ds\\
&&+\frac{1}{\sqrt{\varepsilon}}\int_0^t\Big[\sigma(X^{\varepsilon}_s,\mathscr{L}_{X^{\varepsilon}_s})-\sigma(\bar{X}_s,\mathscr{L}_{X^{\varepsilon}_s})-\partial_x \sigma(\bar{X}_s,\mathscr{L}_{X^{\varepsilon}_s})\cdot\sqrt{\varepsilon}Z_s^\varepsilon\Big]dW_s^1\\
&&+\int_0^t\left[\left(\partial_x\sigma(\bar{X}_s,\mathscr{L}_{X^{\varepsilon}_s})-\partial_x\sigma(\bar{X}_s,\mathscr{L}_{\bar{X}_s})\right)\cdot Z_s^\varepsilon\right] dW_s^1\\
&&+\frac{1}{\sqrt{\varepsilon}}\int_0^t\Big[\sigma(\bar{X}_s,\mathscr{L}_{X^{\varepsilon}_s})-\sigma(\bar{X}_s,\mathscr{L}_{\bar{X}_s})-\EE\left[\partial_{\mu} \sigma(u,\mathscr{L}_{\bar{X}_s})(\bar{X}_s)\cdot\sqrt{\varepsilon}Z_s^\varepsilon\right]\mid_{u=\bar{X}_s}\Big]dW_s^1\\
&&+\int_0^t\left[\partial_x\sigma(\bar{X}_s,\mathscr{L}_{\bar{X}_s})\cdot \rho^{\vare}_s\right] dW_s^1+\int_0^t\EE\left[\partial_{\mu}\sigma(u,\mathscr{L}_{\bar{X}_s})(\bar{X}_s)\cdot \rho^{\vare}_s\right]\mid_{u=\bar{X}_s} dW_s^1.
\end{eqnarray*}
Note that by \cite[(A.2)]{RSX1}, the averaged coefficients $\bar{b}$ and $\sigma$ are globally Lipschitz, i.e.,
\begin{eqnarray*}
|\bar{b}(x_1, \mu_1)-\bar{b}(x_2, \mu_2)|+\|\sigma(x_1,\mu_1)-\sigma(x_2,\mu_2)\|\leq C\big[|x_1-x_2|+\mathbb{W}_2(\mu_1, \mu_2)\big],
\end{eqnarray*}
which implies that
$$\sup_{x\in\RR^n, \mu\in\mathscr{P}_2}\max\{\|\partial_x\bar{b}(x,\mu)\|,\|\partial_{\mu}\bar{b}(x,\mu)\|_{L^2(\mu)},\|\partial_x\sigma(x,\mu)\|,\|\partial_{\mu}\sigma(x,\mu)\|_{L^2(\mu)}\}<\infty.$$
Then by Burkholder-Davis-Gundy's inequality and Young's inequality, we have
\begin{eqnarray*}
\!\!\!\!\!\!\!\!&&\EE\left[\sup_{t\in[0,T]}|\rho^{\vare}_t|\right]\\
\leq\!\!\!\!\!\!\!\!&&\frac{1}{\sqrt{\varepsilon}}\int_0^T\EE\Big|\bar{b}(X^{\varepsilon}_s,\mathscr{L}_{X^{\varepsilon}_s})-\bar{b}(\bar{X}_s,\mathscr{L}_{X^{\varepsilon}_s})-\partial_x \bar{b}(\bar{X}_s,\mathscr{L}_{X^{\varepsilon}_s})\cdot\sqrt{\varepsilon}Z_s^\varepsilon\Big|ds\\
&&+\int_0^T\EE\left[\left\|\partial_x\bar{b}(\bar{X}_s,\mathscr{L}_{X^{\varepsilon}_s})-\partial_x\bar{b}(\bar{X}_s,\mathscr{L}_{\bar{X}_s})\right\||Z_s^\varepsilon|\right]ds\\
&&+\frac{1}{\sqrt{\varepsilon}}\int_0^T\EE\Big|\bar{b}(\bar{X}_s,\mathscr{L}_{X^{\varepsilon}_s})-\bar{b}(\bar{X}_s,\mathscr{L}_{\bar{X}_s})-\EE\left[\partial_{\mu} \bar{b}(u,\mathscr{L}_{\bar{X}_s})(\bar{X}_s)\cdot\sqrt{\varepsilon}Z_s^\varepsilon\right]\mid_{u=\bar{X}_s}\Big|ds\\
&&+\frac{C}{\sqrt{\varepsilon}}\EE\Big[\int_0^T\Big\|\sigma(X^{\varepsilon}_s,\mathscr{L}_{X^{\varepsilon}_s})-\sigma(\bar{X}_s,\mathscr{L}_{X^{\varepsilon}_s})-\partial_x \sigma(\bar{X}_s,\mathscr{L}_{X^{\varepsilon}_s})\cdot\sqrt{\varepsilon}Z_s^\varepsilon\Big\|^2ds\Big]^{\frac{1}{2}}\\
&&+C\EE\Big[\int_0^T\left\|\partial_x\sigma(\bar{X}_s,\mathscr{L}_{X^{\varepsilon}_s})-\partial_x\sigma(\bar{X}_s,\mathscr{L}_{\bar{X}_s})\right\|^2|Z_s^\varepsilon|^2ds\Big]^{\frac{1}{2}}\\
&&+\frac{C}{\sqrt{\varepsilon}}\EE\Big[\int_0^T\left\|\sigma(\bar{X}_s,\mathscr{L}_{X^{\varepsilon}_s})-\sigma(\bar{X}_s,\mathscr{L}_{\bar{X}_s})-\EE\left[\partial_{\mu} \sigma(u,\mathscr{L}_{\bar{X}_s})(\bar{X}_s)\cdot\sqrt{\varepsilon}Z_s^\varepsilon\right]\mid_{u=\bar{X}_s}\right\|^2ds\Big]^{\frac{1}{2}}\\
&&+C\EE\left[\int_0^T|\rho^{\vare}_s|^2ds\right]^{\frac{1}{2}}+C\left[\int_0^T\big(\EE|\rho^{\vare}_s|\big)^2ds\right]^{\frac{1}{2}}\\
=:\!\!\!\!\!\!\!\!&&\sum^6_{k=1}J^{\varepsilon}_k(T)+C\EE\left[\int_0^T|\rho^{\vare}_s|^2ds\right]^{\frac{1}{2}}+C\left[\int_0^T\big(\EE|\rho^{\vare}_s|\big)^2ds\right]^{\frac{1}{2}}\\
\leq\!\!\!\!\!\!\!\!&&\sum^6_{k=1}J^{\varepsilon}_k(T)+C\EE\left[\left(\sup_{s\in[0,T]}|\rho^{\vare}_s|\right)\int_0^T|\rho^{\vare}_s|ds\right]^{\frac{1}{2}}
+C\left[\left(\sup_{s\in[0,T]}\EE|\rho^{\vare}_s|\right)\int_0^T\EE|\rho^{\vare}_s|ds\right]^{\frac{1}{2}}\\
\leq\!\!\!\!\!\!\!\!&&\sum^6_{k=1}J^{\varepsilon}_k(T)+\frac{1}{2}\EE\left[\sup_{s\in[0,T]}|\rho^{\vare}_s|\right]+C\int_0^T\EE|\rho^{\vare}_s|ds.
\end{eqnarray*}
Then by Gronwall's inequality, we obtain
\begin{eqnarray}
\EE\left[\sup_{t\in[0,T]}|\rho^{\vare}_t|\right]\leq\!\!\!\!\!\!\!\!&&C_T\sum^6_{k=1}J^{\varepsilon}_k(T).\label{J0}
\end{eqnarray}

For the term $J^{\varepsilon}_1(T)$, note that
\begin{eqnarray*}
&&\EE\Big|\bar{b}(X^{\varepsilon}_s,\mathscr{L}_{X^{\varepsilon}_s})-\bar{b}(\bar{X}_s,\mathscr{L}_{X^{\varepsilon}_s})-\partial_x \bar{b}(\bar{X}_s,\mathscr{L}_{X^{\varepsilon}_s})\cdot\sqrt{\varepsilon}Z_s^\varepsilon\Big|\\
=\!\!\!\!\!\!\!\!&&\EE\Big|\bar{b}(\bar{X}_s+\sqrt{\varepsilon}Z_s^\varepsilon,\mathscr{L}_{X^{\varepsilon}_s})-\bar{b}(\bar{X}_s,\mathscr{L}_{X^{\varepsilon}_s})-\partial_x \bar{b}(\bar{X}_s,\mathscr{L}_{X^{\varepsilon}_s})\cdot\sqrt{\varepsilon}Z_s^{\varepsilon}\Big|\\
\leq\!\!\!\!\!\!\!\!&&\sqrt{\varepsilon}\EE\Big[\int^1_0\big|\partial_{x} \bar{b}(\bar{X}_s+r\sqrt{\varepsilon}Z_s^\varepsilon,\mathscr{L}_{X^{\varepsilon}_s})-\partial_x \bar{b}(\bar{X}_s,\mathscr{L}_{X^{\varepsilon}_s})\big||Z_s^\varepsilon|dr\Big]\\
\leq\!\!\!\!\!\!\!\!&&\sqrt{\varepsilon}\int_0^1\left[\EE\big|\partial_{x} \bar{b}(\bar{X}_s+r\sqrt{\varepsilon}Z_s^\varepsilon,\mathscr{L}_{X^{\varepsilon}_s})-\partial_x \bar{b}(\bar{X}_s,\mathscr{L}_{X^{\varepsilon}_s})\big|^2\right]^{\frac{1}{2}}dr\left(\EE|Z_s^\varepsilon|^2\right)^{\frac{1}{2}}.
\end{eqnarray*}
Note that $\partial_x \bar{b}(\cdot,\cdot)$ is jointly continuous (see Lemma \ref{C1} in Appendix for the detailed proof).  Then by combining the boundedness of $\|\partial_{x} \bar{b}\|$ and \eref{Zvare}, it follows that
\begin{eqnarray}
\lim_{\ep\rightarrow 0}J^{\varepsilon}_1(T)=0 \label{J1}
\end{eqnarray}
and by \eref{Averaging}, it follows
\begin{eqnarray}
\lim_{\ep\rightarrow 0}J^{\varepsilon}_2(T)=0. \label{J2}
\end{eqnarray}

For the term $J^{\varepsilon}_3(T)$,  for any fixed $u\in\RR^n$, we have
\begin{eqnarray*}
&&|\bar{b}(u,\mathscr{L}_{X^{\varepsilon}_s})-\bar{b}(u,\mathscr{L}_{\bar{X}_s})-\EE\left[\partial_{\mu} \bar{b}(u,\mathscr{L}_{\bar{X}_s})(\bar{X}_s)\cdot\sqrt{\varepsilon}Z_s^\varepsilon\right]|\\
=\!\!\!\!\!\!\!\!&&\left|\int^1_0 \EE\left\{\left[\partial_{\mu}\bar{b}(u,\mathscr{L}_{\bar{X}_s+r\sqrt{\ep}Z^{\varepsilon}_s})(\bar{X}_s+r\sqrt{\ep}Z^{\varepsilon}_s)
-\partial_{\mu}\bar{b}(u,\mathscr{L}_{\bar{X}_s})(\bar{X}_s)\right]\cdot\sqrt{\varepsilon}Z_s^\varepsilon\right\}dr\right|\\
\leq\!\!\!\!\!\!\!\!&&\sqrt{\varepsilon}\int_0^1\left[\EE\big|\partial_{\mu} \bar{b}(u,\mathscr{L}_{\bar{X}_s+r\sqrt{\ep}Z^{\varepsilon}_s})(\bar{X}_s+r\sqrt{\ep}Z^{\varepsilon}_s)-\partial_{\mu} \bar{b}(u,\mathscr{L}_{\bar{X}_s})(\bar{X}_s)\big|^2\right]^{\frac{1}{2}}dr\left(\EE|Z_s^\varepsilon|^2\right)^{\frac{1}{2}}.
\end{eqnarray*}
Then according to  the joint continuity of $\partial_\mu \bar{b}(x,\cdot)(\cdot)$ (see also Lemma \ref{C1} in Appendix), the boundedness of $\|\partial_{\mu} \bar{b}(x,\mu)(z)\|$ and \eref{Zvare}, it immediately leads to
\begin{eqnarray}
\lim_{\ep\rightarrow 0}J^{\varepsilon}_3(T)=0 . \label{J3}
\end{eqnarray}

For the term $J^{\varepsilon}_4(T)$, using H\"{o}lder inequality, we have
\begin{eqnarray*}
J^{\varepsilon}_4(T)\leq\!\!\!\!\!\!\!\!&&C\EE\Big[\int_0^T\int_0^1\big\|\partial_x \sigma(\bar{X}_s+r\sqrt{\varepsilon}Z^{\varepsilon}_s,\mathscr{L}_{X^{\varepsilon}_s})-\partial_x \sigma(\bar{X}_s,\mathscr{L}_{X^{\varepsilon}_s})\big\|^2dr |Z_s^\varepsilon|^2ds\Big]^{\frac{1}{2}}
\nonumber\\
\leq\!\!\!\!\!\!\!\!&&C\Big[\EE\sup_{s\in[0,T]}|Z_s^\varepsilon|^2\Big]^{\frac{1}{2}}\Big[\EE\int_0^T\int_0^1\big\|\partial_x \sigma(\bar{X}_s+r\sqrt{\varepsilon}Z^{\varepsilon}_s,\mathscr{L}_{X^{\varepsilon}_s})-\partial_x \sigma(\bar{X}_s,\mathscr{L}_{X^{\varepsilon}_s})\big\|^2drds\Big]^{\frac{1}{2}}.
\end{eqnarray*}
Due to the jointly continuous property of  $\partial_x \sigma(\cdot,\cdot)$, the boundedness of $\|\partial_{x}\sigma(x,\mu)\|$ and \eref{supZvare}, we have
\begin{eqnarray}
\lim_{\ep\rightarrow 0}J^{\varepsilon}_4(T)=0. \label{J7}
\end{eqnarray}
By similar arguments above, we can deal with the terms $J^{\varepsilon}_5(T)$ and $J^{\varepsilon}_6(T)$, i.e.,
\begin{eqnarray}
\lim_{\ep\rightarrow 0}\left[J^{\varepsilon}_5(T)+J^{\varepsilon}_6(T)\right]=0.\label{J4-J6}
\end{eqnarray}

Combining \eref{J0}-\eref{J4-J6}, we have
$$
\lim_{\varepsilon\rightarrow 0}\EE\left[\sup_{t\in[0,T]}|\rho^{\vare}_t|\right]=0.
$$

The proof is complete. \hspace{\fill}$\Box$
\end{proof}

\subsection{Tightness}\label{S4.2}

In this subsection, we intend to prove the solution $\eta^\varepsilon$ of equation (\ref{e6}) is tight in $C([0,T],\RR^n)$. We introduce the following criterion of tightness (cf.~\cite[Theorem 7.3]{B0}).
\begin{lemma}
For any $T>0$, the family $\{\Pi^\varepsilon\}_{\varepsilon>0}$ is tight in $C([0,T];\RR^n)$ if and only if the following two conditions hold.

(i) For any positive $\epsilon$, there exist $K$ and $\varepsilon_0$ such that
\begin{equation}\label{t1}
\sup_{0<\varepsilon<\varepsilon_0}\mathbb{P}\Big(|\Pi^\varepsilon_0|\geq K\Big)\leq \epsilon.
\end{equation}

(ii) For any positive $\theta,\epsilon$, there exist constants $\varepsilon_0,\delta$ such that
\begin{equation}\label{t2}
\sup_{0<\varepsilon<\varepsilon_0}\mathbb{P}\Big(\sup_{t_1,t_2\in[0,T], |t_1-t_2|<\delta}|\Pi^\varepsilon_{t_1}-\Pi^\varepsilon_{t_2}|\geq \theta\Big)\leq \epsilon.
\end{equation}

\end{lemma}

\vspace{0.3cm}
Recall (\ref{e6}),  for any $0\leq t\leq T$, we denote
$$\eta^\varepsilon(t)=I_1^\varepsilon(t)+I_2^\varepsilon(t)+I_3^\varepsilon(t)+I_4^\varepsilon(t)+I_5^\varepsilon(t),$$
where
\begin{eqnarray*}
I_1^\varepsilon(t):=\!\!\!\!\!\!\!\!&&\frac{1}{\sqrt{\varepsilon}}\int_{0}^t\Big[b(X^{\varepsilon}_s,\mathscr{L}_{X^{\varepsilon}_s},Y^{\varepsilon}_s)-\bar{b}(X^{\varepsilon}_s,\mathscr{L}_{X^{\varepsilon}_s})\Big]ds,
\nonumber\\
I_2^\varepsilon(t):=\!\!\!\!\!\!\!\!&&\int_{0}^t\partial_x\bar{b}(\bar{X}_s,\mathscr{L}_{\bar{X}_s})\cdot \eta_s^\varepsilon ds,
\nonumber\\
I_3^\varepsilon(t):=\!\!\!\!\!\!\!\!&&\int_{0}^t\EE\Big[\partial_{\mu} \bar{b}(u,\mathscr{L}_{\bar{X}_s})(\bar{X}_s)\cdot \eta_s^\varepsilon\Big]\Big|_{u=\bar{X}_s}ds,
\nonumber\\
I_4^\varepsilon(t):=\!\!\!\!\!\!\!\!&&\int_{0}^t\partial_x\sigma(\bar{X}_s,\mathscr{L}_{\bar{X}_s})\cdot\eta_s^\varepsilon dW_s^1,
\nonumber\\
I_5^\varepsilon(t):=\!\!\!\!\!\!\!\!&&\int_{0}^t\EE\Big[\partial_{\mu} \sigma(u,\mathscr{L}_{\bar{X}_s})(\bar{X}_s)\cdot \eta_s^\varepsilon\Big]\Big|_{u=\bar{X}_s}dW_s^1.
\end{eqnarray*}

\begin{lemma}\label{th1}
 Suppose that the assumptions in Theorem \ref{main result 2} hold. The distribution of process $\{\Pi^\varepsilon\}$, where $\Pi^\varepsilon:=(\eta^\varepsilon,I_1^\varepsilon,I_2^\varepsilon,I_3^\varepsilon,I_4^\varepsilon,I_5^\varepsilon)$, is tight in $C([0,T];\RR^{6n})$.
\end{lemma}
\begin{proof}
It is sufficient to prove that  $\{I_1^\varepsilon\},\{I_2^\varepsilon\},\{I_3^\varepsilon\},\{I_4^\varepsilon\}$ and $\{I_5^\varepsilon\}$ satisfy (\ref{t1}) and (\ref{t2}), respectively. We separate the proof into the following two steps.

\textbf{Step 1:} In this step, we shall prove the following uniform estimates of $\{I_i^\varepsilon\}$, $i=1,2,\ldots,5$,
\begin{equation}\label{es1}
\sup_{\varepsilon\in(0,1)}\EE\big[\sup_{t\in[0,T]}|I_i^{\varepsilon}|^2\big]<\infty.
\end{equation}
 Recall the Poisson equation (\ref{PE}), by It\^{o}'s formula (cf.~\cite[Theorem 7.1]{BLPR}) for the function $\Phi(x,\mu,y)$, we have
\begin{eqnarray*}
&&\Phi(X_{t}^{\varepsilon},\mathscr{L}_{X^{\varepsilon}_{t}},Y^{\varepsilon}_{t})\\
=\!\!\!\!\!\!\!\!&&\Phi(\xi,\mathscr{L}_{\xi},\zeta)
+\int^t_0 \EE\left[b(X^{\varepsilon}_s,\mathscr{L}_{ X^{\varepsilon}_{s}}, Y^{\varepsilon}_s)\partial_{\mu}\Phi(x,\mathscr{L}_{X^{\varepsilon}_{s}},y)(X^{\varepsilon}_s)\right]\mid_{x=X_{s}^{\varepsilon},y=Y^{\varepsilon}_{s}}ds\\
&&+\int^t_0 \frac{1}{2}\EE \text{Tr}\left[\sigma\sigma^{*}(X^{\varepsilon}_s,\mathscr{L}_{ X^{\varepsilon}_{s}})\partial_z\partial_{\mu}\Phi(x,\mathscr{L}_{X^{\varepsilon}_{s}},y)(X^{\varepsilon}_s)\right]\mid_{x=X_{s}^{\varepsilon},y=Y^{\varepsilon}_{s}}ds\\
&&+\int^t_0 \mathcal{L}_{1}(\mathscr{L}_{X^{\varepsilon}_{s}},Y^{\varepsilon}_{s})\Phi(X_{s}^{\varepsilon},\mathscr{L}_{X^{\varepsilon}_{s}},Y^{\varepsilon}_{s})ds\\
&&+\frac{1}{\varepsilon}\int^t_0 \mathcal{L}_{2}(X_{s}^{\varepsilon},\mathscr{L}_{X^{\varepsilon}_{s}})\Phi(X_{s}^{\varepsilon},\mathscr{L}_{X^{\varepsilon}_{s}},Y^{\varepsilon}_{s})ds+M^{1,\varepsilon}_t+\frac{1}{\sqrt{\varepsilon}}M^{2,\varepsilon}_t,
\end{eqnarray*}
where $\mathcal{L}_{1}(\mu,y)\Phi(x,\mu,y):=(\mathfrak{L}_{1}(\mu,y)\Phi_1(x,\mu,y),\ldots, \mathfrak{L}_{1}(\mu,y)\Phi_n(x,\mu,y))$ with
\begin{eqnarray}\label{inf2}
 \mathfrak{L}_{1}(\mu,y)\Phi_k(x,\mu,y):=\!\!\!\!\!\!\!\!&&\left\langle b(x,\mu,y), \partial_x \Phi_k(x,\mu,y)\right \rangle\nonumber\\
 &&+\frac{1}{2}\text{Tr}\left[\sigma\sigma^{*}(x,\mu)\partial^2_{xx} \Phi_k(x,\mu,y)\right ],\quad k=1,\ldots, n,
\end{eqnarray}
and $M^{1,\varepsilon}_t, M^{2,\varepsilon}_t$ are two local martingales, which are defined by
\begin{eqnarray*}
&&M^{1,\varepsilon}_t=\int^t_0 \partial_x \Phi(X_{s}^{\varepsilon},\mathscr{L}_{X^{\varepsilon}_{s}},Y_{s}^{\varepsilon})\cdot \sigma(X^{\varepsilon}_s,\mathscr{L}_{ X^{\varepsilon}_{s}}) dW^1_s,\\
&&M^{2,\varepsilon}_t=\int^t_0 \partial_y \Phi(X_{s}^{\varepsilon},\mathscr{L}_{X^{\varepsilon}_{s}},Y_{s}^{\varepsilon})\cdot g(X^{\varepsilon}_s,\mathscr{L}_{ X^{\varepsilon}_{s}}, Y^{\varepsilon}_s) dW^2_s.
\end{eqnarray*}
Then it follows that
\begin{eqnarray}\label{4}
I_1^\varepsilon(t)=\!\!\!\!\!\!\!\!&&\sqrt{\varepsilon}\Big\{-\Phi(X_{t}^{\varepsilon},\mathscr{L}_{X^{\varepsilon}_{t}},Y^{\varepsilon}_{t})+\Phi(\xi,\mathscr{L}_{\xi},\zeta)
\nonumber\\
\!\!\!\!\!\!\!\!&&+\int^t_0 \EE\left[b(X^{\varepsilon}_s,\mathscr{L}_{ X^{\varepsilon}_{s}}, Y^{\varepsilon}_s)\partial_{\mu}\Phi(x,\mathscr{L}_{X^{\varepsilon}_{s}},y)(X^{\varepsilon}_s)\right]\mid_{x=X_{s}^{\varepsilon},y=Y^{\varepsilon}_{s}}ds\nonumber\\
&&+\int^t_0 \frac{1}{2}\EE \text{Tr}\left[\sigma\sigma^{*}(X^{\varepsilon}_s,\mathscr{L}_{ X^{\varepsilon}_{s}})\partial_z\partial_{\mu}\Phi(x,\mathscr{L}_{X^{\varepsilon}_{s}},y)(X^{\varepsilon}_s)\right]\mid_{x=X_{s}^{\varepsilon},y=Y^{\varepsilon}_{s}}ds\nonumber\\
&&+\int^t_0 \mathcal{L}_{1}(\mathscr{L}_{X^{\varepsilon}_{s}},Y^{\varepsilon}_{s})\Phi(X_{s}^{\varepsilon},\mathscr{L}_{X^{\varepsilon}_{s}},Y^{\varepsilon}_{s})ds+M^{1,\varepsilon}_t\Big\}+M^{2,\varepsilon}_t\nonumber\\
=:\!\!\!\!\!\!\!\!&&I_{11}^\varepsilon(t)+M^{2,\varepsilon}_t.
\end{eqnarray}
By Burkholder-Davis-Gundy's inequality, (\ref{E1})  and (\ref{X})-(\ref{Y1}), we have
\begin{eqnarray}
\EE\left[\sup_{t\in[0,T]}|I_{11}^\varepsilon(t)|^2\right]\leq\!\!\!\!\!\!\!\!&&C_T\ep\left\{1+\EE\left[\sup_{t\in[0,T]}|X^{\varepsilon}_t|^4\right]+\left[\EE\Big(\sup_{t\in[0,T]}|Y^{\varepsilon}_t|^4\Big)\right]^{1/2}+\sup_{t\in[0,T]}\EE|Y^{\varepsilon}_t|^4\right\}\nonumber\\
\leq\!\!\!\!\!\!\!\!&&C_T\sqrt{\ep}(1+\EE|\xi|^4+\EE|\zeta|^4)\label{Vanish}
\end{eqnarray}
and
\begin{eqnarray}\label{35}
\EE\left[\sup_{t\in[0,T]}|M^{2,\varepsilon}_t|^2\right]\leq\!\!\!\!\!\!\!\!&&C_T\left(1+\sup_{t\in[0,T]}\EE|Y^{\varepsilon}_t|^2\right)\leq C_T(1+\EE|\xi|^2+\EE|\zeta|^2).
\end{eqnarray}
Thus we can get
\begin{eqnarray}\label{6}
\sup_{\varepsilon\in(0,1)}\EE\left[\sup_{t\in[0,T]}|I_1^\varepsilon(t)|^2\right]\leq\!\!\!\!\!\!\!\!&&C_T\left(1+\EE|\xi|^4+\EE|\zeta|^4\right),
\end{eqnarray}
which easily implies that $I_1^\varepsilon$ satisfies (\ref{es1}).

Secondly, in order to show the uniform estimates of $I_i^\varepsilon, i=2,3,4,5$, we need to prove that
\begin{eqnarray}\label{5}
\sup_{\varepsilon\in(0,1)}\EE\big[\sup_{t\in[0,T]}|\eta^\varepsilon_t|^2\big]\leq C_T\big(1+\EE|\xi|^4+\EE|\zeta|^4\big).
\end{eqnarray}
 In fact, note that the uniform boundedness of $\|\partial_x\bar{b}(x,\mu)\|$, $\|\partial_{\mu}\bar{b}(x,\mu)(z)\|$, $\|\partial_x\sigma(x,\mu)\|$ and $\|\partial_{\mu}\sigma(x,\mu)(z)\|$, we infer that
\begin{eqnarray*}
\EE\big[\sup_{t\in[0,T]}|\eta^\varepsilon_t|^2\big]\leq\!\!\!\!\!\!\!\!&&\EE\Big[\sup_{t\in[0,T]}|I_1^\varepsilon(t)|^2\Big]+C\EE\int_0^T|\eta^\varepsilon_t|dt+
\EE\Big[\sup_{t\in[0,T]}\Big|\int_{0}^t\left[\partial_x\sigma(\bar{X}_s,\mathscr{L}_{\bar{X}_s})\cdot\eta_s^\varepsilon \right] dW_s^1\Big|^2\Big]
\nonumber\\
\!\!\!\!\!\!\!\!&&+\EE\Big[\sup_{t\in[0,T]}\Big|\int_{0}^t\EE\Big[\partial_{\mu} \sigma(u,\mathscr{L}_{\bar{X}_s})(\bar{X}_s)\cdot \eta_t^\varepsilon\Big]\Big|_{u=\bar{X}_s}dW_s^1\Big|^2\Big]
\nonumber\\
\leq\!\!\!\!\!\!\!\!&&C_T\big(1+\EE|\xi|^4+\EE|\zeta|^4\big)+C_T\EE\int_0^T|\eta^\varepsilon_t|^2dt.
\end{eqnarray*}
Then the Gronwall's lemma yields \eref{5} holds. As a consequence, we can see that $I_i^\varepsilon, i=2,3,4,5$ also satisfy  (\ref{es1}).

\vspace{0.1cm}
\textbf{Step 2:} In this step, we shall prove that $\{I_i^\varepsilon\}$, $i=1,2,\ldots,5$, satisfy (\ref{t2}). Owing to the uniform boundedness of $\|\partial_x\bar{b}(x,\mu)\|$, $\|\partial_{\mu}\bar{b}(x,\mu)(z)\|$, $\|\partial_x\sigma(x,\mu)\|$ and $\|\partial_{\mu}\sigma(x,\mu)(z)\|$, we can prove that the terms $I_i^\varepsilon$,$i=2,3,4$, fulfill (\ref{t2}) by making use of (\ref{5}).

By Burkholder-Davis-Gundy's inequality, we have
 \begin{eqnarray}
\!\!\!\!\!\!\!\!&&\EE\Big|\int_{t_1}^{t_2}\EE\big[\partial_{\mu} \sigma(u,\mathscr{L}_{\bar{X}_s})(\bar{X}_s)\cdot \eta_s^\varepsilon\big]\big|_{u=\bar{X}_s}dW_s^1\Big|^4
\nonumber\\
\leq\!\!\!\!\!\!\!\!&&\EE\Big[\int_{t_1}^{t_2}\Big|\EE\big[\partial_{\mu} \sigma(u,\mathscr{L}_{\bar{X}_s})(\bar{X}_s)\cdot \eta_s^\varepsilon\big]\big|_{u=\bar{X}_s}\Big|^2ds\Big]^{2}
\nonumber\\
\leq\!\!\!\!\!\!\!\!&&\Big(\int_{t_1}^{t_2}\EE|\eta_s^\varepsilon|^2ds\Big)^{2}
\nonumber\\
\leq\!\!\!\!\!\!\!\!&&C_T\big(1+\EE|\xi|^4+\EE|\zeta|^4\big)|t_2-t_1|^2,\label{F4.14}
\end{eqnarray}
which implies (\ref{t2}) holds for $I_5^\varepsilon$ by the Kolmogorov's criterion (cf.~\cite{Daz1}). 

Next, we shall prove $I_{1}^\varepsilon$ satisfies (\ref{t2}). On the one hand, following from \eref{Vanish}, $I_{11}^\varepsilon$ vanishes in probability in $C([0,T];\RR^n)$. On the other hand, by a similar argument in \eref{F4.14}, we can see that $M^{2,\varepsilon}$ also satisfies (\ref{t2}). Hence, $I_{1}^\varepsilon$ satisfies (\ref{t2}). The proof is complete.

\hspace{\fill}$\Box$
\end{proof}

\subsection{Proof of Theorem \ref{main result 2}} \label{S4.3}

In order to prove the first main result, the key point is to identify the weak limit of $\eta_t^\varepsilon$ in $C([0,T];\RR^n)$.
\begin{proposition}\label{th2}
Under the assumptions in Theorem \ref{main result 2}, the solution $\eta^\varepsilon$ of equation (\ref{e6}) weakly converges  to the solution $Z$ of equation (\ref{e5}) in $C([0,T];\RR^n)$, as $\varepsilon\to0 $.
\end{proposition}
\begin{proof}
It is sufficient to prove that there is a subsequence $\{\varepsilon_{n_k}\}_{k\geq1}$ of any sequence $\{\varepsilon_n\}_{n\geq1}$ tending to $0$ such that $\eta^{\varepsilon_{n_k}}$ converges weakly to $Z$ in $C([0,T];\RR^n)$.
 According to Lemma \ref{th1}, for a subsequence $\{\varepsilon_{n_k}\}_{k\geq1}$, one may assume that
  $\{(\eta^{\varepsilon_{n_k}},I_1^{\varepsilon_{n_k}}, I_2^{\varepsilon_{n_k}},I_3^{\varepsilon_{n_k}},I_4^{\varepsilon_{n_k}},I_5^{\varepsilon_{n_k}},W^1,W^2)\}_{k\geq1}$ converges weakly, where $W^1,W^2$ are two independent Brownian motions defined in (\ref{E2}).

Then using the Skorohod representation theorem (cf.~\cite[Theorem C.1]{BHR} or \cite[Corollary C.1]{M2}), there exists a new probability space denoted by $(\hat{\Omega},\hat{\mathscr{F}},\hat{\PP})$ and a sequence $$\{(\hat{\eta}^{\varepsilon_{n_k}},\hat{I_1}^{\varepsilon_{n_k}},\hat{I_2}^{\varepsilon_{n_k}},\hat{I_3}^{\varepsilon_{n_k}},\hat{I_4}^{\varepsilon_{n_k}},\hat{I_5}^{\varepsilon_{n_k}},\hat{W}^{1,\varepsilon_{n_k}},\hat{W}^{2,\varepsilon_{n_k}})\}_{k\geq1},$$
 which has the same law with
$\{(\eta^{\varepsilon_{n_k}},I_1^{\varepsilon_{n_k}},W^1, W^2,I_2^{\varepsilon_{n_k}},I_3^{\varepsilon_{n_k}},I_4^{\varepsilon_{n_k}},I_5^{\varepsilon_{n_k}})\}_{k\geq1}$ and the following properties hold

\vspace{0.1cm}
(i) $(\hat{\eta}^{\varepsilon_{n_k}},\hat{I_1}^{\varepsilon_{n_k}},\hat{I_2}^{\varepsilon_{n_k}},\hat{I_3}^{\varepsilon_{n_k}},\hat{I_4}^{\varepsilon_{n_k}},\hat{I_5}^{\varepsilon_{n_k}},\hat{W}^{1,\varepsilon_{n_k}},\hat{W}^{2,\varepsilon_{n_k}})$
converges to  $(\hat{\eta},\hat{I_1},\hat{I_2},\hat{I_3},\hat{I_4},\hat{I_5},\hat{W}^{1},\hat{W}^{2})$ in $C([0,T];\RR^{6n+d_1+d_2})$, $\hat{\PP}$-a.s., as $k\to \infty$;

\vspace{0.1cm}
(ii) $(\hat{W}^{1,\varepsilon_{n_k}}(\hat{\omega}), \hat{W}^{2,\varepsilon_{n_k}}(\hat{\omega}))=(\hat{W}^{1}(\hat{\omega}), \hat{W}^{2}(\hat{\omega}))$ for all $\hat{\omega}\in \hat{\Omega}$.

From the equation \eref{e6} satisfied by  $\{(\eta^{\varepsilon_{n_k}},I_1^{\varepsilon_{n_k}}, I_2^{\varepsilon_{n_k}},I_3^{\varepsilon_{n_k}},I_4^{\varepsilon_{n_k}},I_5^{\varepsilon_{n_k}},W^1,W^2)\}_{k\geq1}$, it follows that
$$\hat{\eta}^{\varepsilon_{n_k}}(t)=\hat{I}_1^{\varepsilon_{n_k}}(t)+\hat{I}_2^{\varepsilon_{n_k}}(t)+\hat{I}_3^{\varepsilon_{n_k}}(t)+\hat{I}_4^{\varepsilon_{n_k}}(t)+\hat{I}_5^{\varepsilon_{n_k}}(t),$$
where
\begin{eqnarray*}
\hat{I}_1^{\varepsilon_{n_k}}(t):=\!\!\!\!\!\!\!\!&&\frac{1}{\sqrt{\varepsilon_{n_k}}}\int_{0}^t\Big[b(\hat{X}^{\varepsilon_{n_k}}_s,\mathscr{L}_{\hat{X}^{\varepsilon_{n_k}}_s},\hat{Y}^{\varepsilon_{n_k}}_s)-\bar{b}(\hat{X}^{\varepsilon_{n_k}}_s,\mathscr{L}_{\hat{X}^{\varepsilon_{n_k}}_s})\Big]ds,
\nonumber\\
\hat{I}_2^{\varepsilon_{n_k}}(t):=\!\!\!\!\!\!\!\!&&\int_{0}^t\partial_x\bar{b}(\hat{\bar{X}}_s,\mathscr{L}_{\hat{\bar{X}}_s})\cdot \hat{\eta}_s^{\varepsilon_{n_k}} ds,
\nonumber\\
\hat{I}_3^{\varepsilon_{n_k}}(t):=\!\!\!\!\!\!\!\!&&\int_{0}^t\hat{\EE}\Big[\partial_{\mu} \bar{b}(u,\mathscr{L}_{\hat{\bar{X}}_s})(\hat{\bar{X}}_s)\cdot \hat{\eta}_s^{\varepsilon_{n_k}}\Big]\Big|_{u=\hat{\bar{X}}_s}ds,
\nonumber\\
\hat{I}_4^{\varepsilon_{n_k}}(t):=\!\!\!\!\!\!\!\!&&\int_{0}^t\partial_x\sigma(\hat{\bar{X}}_s,\mathscr{L}_{\hat{\bar{X}}_s})\cdot \hat{\eta}_s^{\varepsilon_{n_k}} d \hat{W}_s^1,
\nonumber\\
\hat{I}_5^{\varepsilon_{n_k}}(t):=\!\!\!\!\!\!\!\!&&\int_{0}^t\hat{\EE}\Big[\partial_{\mu} \sigma(u,\mathscr{L}_{\hat{\bar{X}}_s})(\hat{\bar{X}}_s)\cdot \hat{\eta}_s^{\varepsilon_{n_k}}\Big]\Big|_{u=\hat{\bar{X}}_s}d \hat{W}_s^1.
\end{eqnarray*}
Here $(\hat{X}^{\varepsilon_{n_k}},\hat{Y}^{\varepsilon_{n_k}})$ is the solutions of (\ref{E2}) with $\delta=1$, $\varepsilon_{n_k}$ replacing $\varepsilon$, $(\hat{W}^{1},\hat{W}^{2})$ replacing $(W^1,W^2)$ and $(\hat{\xi},\hat{\zeta})$ replacing $(\xi,\zeta)$,
and $\hat{\bar{X}}$ is the solutions of (\ref{1.3}) with $\hat{W}^{1}$ replacing $W^1$ and $\hat{\xi}$ replacing $\xi$, where $\hat{\xi}$ and $\hat{\zeta}$ are two random variables on $(\hat{\Omega},\hat{\mathscr{F}}, \hat{\PP})$, which coincide in laws with $\xi$ and $\zeta$ respectively, and $\hat{\EE}$ is the expectation on $(\hat{\Omega},\hat{\mathscr{F}},\{\hat{\mathscr{F}_t}\}_{t\geq0},\hat{\PP})$. Here $\hat{\mathscr{F}}_t$ denote
the natural filtration generated by $\{\hat{\xi},\hat{\zeta},\hat{W}^1_s,\hat{W}^2_s, s\leq t\}$.

We first identify the limiting processes $\hat{I}_i$ of $\hat{I}^{\varepsilon_{n_k}}_i$, $i=2,3,4,5$, have the following form,  i.e. $\hat{\PP}$-a.s.,
\begin{eqnarray*}
\hat{I}_2(t)=\!\!\!\!\!\!\!\!&&\int_{0}^t\partial_x\bar{b}(\hat{\bar{X}}_s,\mathscr{L}_{\hat{\bar{X}}_s})\cdot\hat{\eta}_s ds,
\nonumber\\
\hat{I}_3(t)=\!\!\!\!\!\!\!\!&&\int_{0}^t\hat{\EE}\Big[\partial_{\mu} \bar{b}(u,\mathscr{L}_{\hat{\bar{X}}_s})(\hat{\bar{X}}_s)\cdot \hat{\eta}_s\Big]\Big|_{u=\hat{\bar{X}}_s}ds,
\nonumber\\
\hat{I}_4(t)=\!\!\!\!\!\!\!\!&&\int_{0}^t\left[\partial_x\sigma(\hat{\bar{X}}_s,\mathscr{L}_{\hat{\bar{X}}_s})\cdot\hat{\eta}_s\right] d\hat{W}_s^1,
\nonumber\\
\hat{I}_5(t)=\!\!\!\!\!\!\!\!&&\int_{0}^t\hat{\EE}\Big[\partial_{\mu} \sigma(u,\mathscr{L}_{\hat{\bar{X}}_s})(\hat{\bar{X}}_s)\cdot \hat{\eta}_s\Big]\Big|_{u=\hat{\bar{X}}_s}d\hat{W}_s^1.
\end{eqnarray*}
In fact, by \eref{5} and $\eta^{\varepsilon_{n_k}}$ coincides in law with $\hat{\eta}^{\varepsilon_{n_k}}$, we have
$$
\sup_{k\geq 1}\hat{\EE}\big[\sup_{t\in[0,T]}|\hat{\eta}^{\varepsilon_{n_k}}_t|^2\big]\leq C_T\big(1+\hat{\EE}|\hat{\xi}|^4+\EE|\hat{\zeta}|^4\big),
$$
then by $\hat{\eta}^{\varepsilon_{n_k}}\rightarrow \hat{\eta} $ in $C([0,T];\RR^n)$, $\hat{\PP}$-a.s. and applying the Vitali's convergence theorem,
we deduce that
\begin{eqnarray}
\hat{\EE}\left[\sup_{t\in[0,T]}|\hat{\eta}^{\varepsilon_{n_k}}_t-\hat{\eta}_t|\right]=0.\label{F4.20}
\end{eqnarray}
By the uniform boundedness of $\partial_{\mu} \bar{b}(u,\mu)(z)$ (see \eref{Boundbarb} in the Appendix) and \eref{F4.20}, it follows
\begin{eqnarray*}
\lim_{k\rightarrow \infty}\sup_{t\in [0,T]}|\hat{I}_3^{\varepsilon_{n_k}}(t)-\hat{I}_3(t)|\leq \lim_{k\rightarrow \infty}C_T\hat{\EE}\left[\sup_{t\in[0,T]}|\hat{\eta}^{\varepsilon_{n_k}}_t-\hat{\eta}_t|\right]=0.
\end{eqnarray*}
Hence  the limiting process of $\hat{I}_3^{\varepsilon_{n_k}}$ is $\hat{I}_3$.

Similarly, by the uniform boundedness of $\partial_{\mu} \sigma(u,\mu)(z)$ and Burkholder-Davis-Gundy's inequality, we have
\begin{eqnarray*}
\hat{\EE}\left[\sup_{t\in [0,T]}|\hat{I}_5^{\varepsilon_{n_k}}(t)-\hat{I}_5(t)|^2\right]\leq C\int^T_0 \left(\hat{\EE}|\hat{\eta}^{\varepsilon_{n_k}}_t-\hat{\eta}_t|\right)^2dt.
\end{eqnarray*}
Then by \eref{F4.20} again we get
\begin{eqnarray*}
\lim_{k\rightarrow \infty}\hat{\EE}\left[\sup_{t\in [0,T]}|\hat{I}_5^{\varepsilon_{n_k}}(t)-\hat{I}_5(t)|^2\right]=0,
\end{eqnarray*}
which implies that there exists a subsequence (still denote by ${n_k}$) such that $\hat{I}_5^{\varepsilon_{n_k}}$ converges to $\hat{I}_5$ in $C([0,T];\RR^n)$, $\hat{\PP}$-a.s., thus the limiting process of $\hat{I}_5^{\varepsilon_{n_k}}$ should be $\hat{I}_5$. By the same argument, we can check that the limiting processes of $\hat{I}_2^{\varepsilon_{n_k}}$ and $\hat{I}_4^{\varepsilon_{n_k}}$ are $\hat{I}_2$ and $\hat{I}_4$ respectively.

Next, we intend to identify the limiting process of $\hat{I}_1^{\varepsilon_{n_k}}$ in $C([0,T];\RR^n)$. Note that (\ref{Vanish}) implies that $I_{11}^\varepsilon$ converges to $0$ in probability in $C([0,T];\RR^{n})$, so $\hat{I_1}$ should be the same as the limiting process $\hat{M}^2$ of $\hat{M}^{2,\varepsilon_{n_k}}$ in $C([0,T];\RR^{n})$, where
$$\hat{M}^{2,\varepsilon_{n_k}}_t:=\int_0^t \partial_y \Phi_g(\hat{X}_{s}^{\varepsilon_{n_k}},\mathscr{L}_{\hat{X}^{\varepsilon_{n_k}}_{s}},\hat{Y}_{s}^{\varepsilon_{n_k}}) d \hat{W}_s^{2,\varepsilon_{n_k}},$$
is a continuous local martingale on $(\hat{\Omega},\hat{\mathscr{F}},\{\hat{\mathscr{F}_t}\}_{t\geq0},\hat{\PP})$ with quadratic variational process
$$\langle \hat{M}^{2,\varepsilon_{n_k}}\rangle_t=\int_0^{t}(\partial_y\Phi_g)(\partial_y\Phi_g)^*(\hat{X}_{s}^{\varepsilon_{n_k}},\mathscr{L}_{\hat{X}^{\varepsilon_{n_k}}_{s}},\hat{Y}_{s}^{\varepsilon_{n_k}})  ds ,\quad t\in[0,T].$$

Note that $\hat{M}^{2,\varepsilon_{n_k}}_t$ converges to $\hat{M}^2_t$, $\hat{\PP}$-a.s., we can see that $\hat{M}^2_t$ is a continuous local martingale on $(\hat{\Omega},\hat{\mathscr{F}},\{\hat{\mathscr{F}_t}\}_{t\geq0},\hat{\PP})$. If the quadratic variational process of $\hat{M}^2_t$  is
\begin{eqnarray}
\langle \hat{M}^{2}\rangle_t=\int_0^{t}\overline{(\partial_y\Phi_g)(\partial_y\Phi_g)^*}(\hat{\bar{X}}_s,\mathscr{L}_{\hat{\bar{X}}_s})  ds,\quad t\in[0,T], \label{quadratic1}
\end{eqnarray}
(the proof is presented in Lemma \ref{L4.5} below)
 then according to the  martingale representation theorem (cf.~\cite[Theorem 4.5.1]{SV} or \cite[Theorem 8.2]{Daz1}), there exists a probability
space $(\tilde{\Omega},\tilde{\mathscr{F}},\tilde{\PP})$, a filtration $\{\tilde{\mathscr{F}_t}\}_{t\geq0}$ and a standard Brownian motion $\hat{W}$ defined on $(\hat{\Omega}\times \tilde{\Omega},\hat{\mathscr{F}}\times \tilde{\mathscr{F}},\hat{\PP}\times \tilde{\PP})$ adapted to $\hat{\mathscr{F}_t}\times \tilde{\mathscr{F}_t}$ such that
\begin{eqnarray*}
\hat M^{2}_t=\int_0^{t}\big(\overline{(\partial_y\Phi_g)(\partial_y\Phi_g)^*}\big)^{\frac{1}{2}}(\hat{\bar{X}}_s,\mathscr{L}_{\hat{\bar{X}}_s})  d\hat{W}_s ,~~t\in[0,T],
\end{eqnarray*}
where
$$
\hat M^{2}_t(\hat{\omega},\tilde{\omega})=\hat M^{2}_t(\hat{\omega}),\quad \hat{\bar{X}}_s(\hat{\omega},\tilde{\omega})=\hat{\bar{X}}_s(\hat{\omega}),\quad (\hat{\omega},\tilde{\omega})\in (\hat{\Omega}\times \tilde{\Omega}).
$$
Note that $\langle \hat{W}^1,\hat{M}^{2,\varepsilon_{n_k}}\rangle_t=0$ and $\hat{M}^{2,\varepsilon_{n_k}}_t$ converges to $\hat{M}^2_t$, $\hat{\PP}\times \tilde{\PP}$-a.s., thus $\langle \hat{W}^1,\hat{M}^{2}\rangle_t=0$, which implies that $\hat{W}$ is independent of $\hat{W}^1$.

Finally, by the discussion above,  one can immediately conclude that $\hat{\eta}$ satisfies
\begin{eqnarray*}
d\hat{\eta}_t= \!\!\!\!\!\!\!\!&& \partial_x\bar{b}(\hat{\bar{X}}_t,\mathscr{L}_{\hat{\bar{X}}_t})\cdot \hat{\eta}_tdt
+\bar{\EE}\Big[\partial_{\mu} \bar{b}(u,\mathscr{L}_{\hat{\bar{X}}_t})(\hat{\bar{X}}_t)\cdot \hat{\eta}_t \Big]\Big|_{u=\hat{\bar{X}}_t}dt+\partial_x\sigma(\hat{\bar{X}}_t,\mathscr{L}_{\hat{\bar{X}}_t})\cdot \hat{\eta}_t d\hat{W}_t^1
 \nonumber\\
 \!\!\!\!\!\!\!\!&& + \bar{\EE}\Big[\partial_{\mu} \sigma(u,\mathscr{L}_{\hat{\bar{X}}_t})(\hat{\bar{X}}_t)\cdot \hat{\eta}_t\Big]\Big|_{u=\hat{\bar{X}}_t}d\hat{W}_t^1+\big(\overline{(\partial_y\Phi_g)(\partial_y\Phi_g)^*}\big)^{\frac{1}{2}}(\hat{\bar{X}}_t,\mathscr{L}_{\hat{\bar{X}}_t}) d\hat{W}_t,~\hat{\eta}_0=0,
\end{eqnarray*}
 where $\bar{\EE}$ is the expectation on $(\hat{\Omega}\times \tilde{\Omega},\hat{\mathscr{F}}\times \tilde{\mathscr{F}},\hat{\PP}\times \tilde{\PP})$ and
 $$
\hat{\eta}_t(\hat{\omega},\tilde{\omega})=\hat{\eta}_t(\hat{\omega}),\quad \hat{W}_t^1(\hat{\omega},\tilde{\omega})=\hat{W}_t^1(\hat{\omega}) .
$$
Then by the weak uniqueness of equations \eref{1.3} and (\ref{e5}) (see subsection \ref{subsection 6.3} in the Appendix), $\hat{\eta}$ equals to the solution $Z$ of equation (\ref{e5}) in the sense of distribution. The proof is complete. \hspace{\fill}$\Box$
\end{proof}

\vspace{0.2cm}
Now, we are in the position to finish the  proof of our first main result.

\vspace{2mm}
\textbf{Proof of Theorem \ref{main result 2}}: Combining Propositions  \ref{th3} and \ref{th2}, we immediate  obtain that Theorem \ref{main result 2} holds. The proof is complete. \hspace{\fill}$\Box$

\vspace{0.2cm}
In order to prove \eref{quadratic1}, we will use the time discretization scheme, which is an effective method in studying the averaging principle for multi-scale stochastic system.
\begin{lemma}\label{L4.5}
Under the assumptions in Theorem \ref{main result 2}, the quadratic variational process of $\hat{M}^2_t$ is given by
\begin{eqnarray}
\langle \hat{M}^{2}\rangle_t=\int_0^{t}\overline{(\partial_y\Phi_g)(\partial_y\Phi_g)^*}(\hat{\bar{X}}_s,\mathscr{L}_{\hat{\bar{X}}_s})  ds,\quad t\in[0,T]. \label{quadratic}
\end{eqnarray}
\end{lemma}
\begin{proof}
Without loss of generality, we still use $(\Omega,\mathscr{F},\{\mathscr{F}_t\}_{t\geq0},\PP)$ instead of $(\hat{\Omega},\hat{\mathscr{F}},\{\hat{\mathscr{F}}_t\}_{t\geq0},\hat{\PP})$, and $(M^{2,\varepsilon},M^{2},X^{\varepsilon},Y^{\varepsilon},\bar{X})$ instead of $(\hat{M}^{2,\varepsilon_{n_k}},\hat{M}^{2},\hat{X}^{\varepsilon_{n_k}},\hat{Y}^{\varepsilon_{n_k}},\hat{\bar{X}})$ below.

We denote
$$F(x,\mu,y)=\big[(\partial_y\Phi_g)(\partial_y\Phi_g)^*\big](x,\mu,y),$$
and the corresponding averaged coefficient is denoted by
$$\bar{F}(x,\mu)=\int_{\RR^m}F(x,\mu,y)\nu^{x,\mu}(dy)=\overline{(\partial_y\Phi_g)(\partial_y\Phi_g)^*}(x,\mu).$$
where $\nu^{x,\mu}$ is the unique invariant measure of the transition semigroup of the frozen equation (\ref{FEQ2}) for any $x\in\RR^n,\mu\in\mathscr{P}_2$. Note that $M^{2,\varepsilon}_t$ converges to $M^2_t$, $\PP$-a.s., we shall prove
\begin{eqnarray}
\langle M^{2}\rangle_t=\int_0^{t}\overline{(\partial_y\Phi_g)(\partial_y\Phi_g)^*}(\bar{X}_s,\mathscr{L}_{\bar{X}_s})  ds,\quad t\in[0,T]. \label{quadratic2}
\end{eqnarray}
It is sufficient to prove that for any $t\in[0,T]$,
\begin{eqnarray}
\lim_{\varepsilon\to 0}\EE\left\|\int_0^t F(X_{s}^{\varepsilon},\mathscr{L}_{X^{\varepsilon}_{s}},Y_{s}^{\varepsilon})ds-\int_0^{t}\bar{F}(\bar{X}_s,\mathscr{L}_{\bar{X}_s})  ds\right\|^2=0.\label{Key}
\end{eqnarray}
In fact, if \eref{Key} holds, then there exists a sequence $\{\varepsilon_k\}_{k\geq 1}\downarrow 0$ such that
$$
\int_0^t F(X_{s}^{\varepsilon_k},\mathscr{L}_{X^{\varepsilon_k}_{s}},Y_{s}^{\varepsilon_k})ds\rightarrow\int_0^{t}\bar{F}(\bar{X}_s,\mathscr{L}_{\bar{X}_s})  ds,~\PP\text{-a.s.},~\text{as}~k\to\infty.
$$
Note that $M^{2,\varepsilon_k}_t\otimes M^{2,\varepsilon_k}_t-\int_0^{t}F(X_{r}^{\varepsilon_k},\mathscr{L}_{X^{\varepsilon_k}_{r}},Y_{r}^{\varepsilon_k})dr$ is a matrix-valued martingale, and by the Vitali's convergence theorem, we have for any $0\leq s\leq t\leq T$,
\begin{eqnarray*}
\!\!\!\!\!\!\!\!&&\EE\Big[\big(M^{2}_t\otimes M^{2}_t-\int_0^{t}\bar{F}(\bar{X}_r,\mathscr{L}_{\bar{X}_r})dr\big)-\big(M^{2}_s\otimes M^{2}_s-\int_0^{s}\bar{F}(\bar{X}_r,\mathscr{L}_{\bar{X}_r})dr\big) \big|\mathscr{F}_s\Big]
\nonumber \\
=\!\!\!\!\!\!\!\!&&\lim_{k\to\infty}\EE\Big[\big(M^{2,\varepsilon_k}_t\otimes M^{2,\varepsilon_k}_t-\int_0^{t}F(X_{r}^{\varepsilon_k},\mathscr{L}_{X^{\varepsilon_k}_{r}},Y_{r}^{\varepsilon_k})dr\big)\\
&&\quad\quad\quad-\big(M^{2,\varepsilon_k}_s\otimes  M^{2,\varepsilon_k}_s-\int_0^{s}F(X_{r}^{\varepsilon_k},\mathscr{L}_{X^{\varepsilon_k}_{r}},Y_{r}^{\varepsilon_k})dr\big)\big|\mathscr{F}_s\Big]
\nonumber \\
=\!\!\!\!\!\!\!\!&&0,~\PP\text{-a.s.}.
\end{eqnarray*}
Therefore, we conclude
\begin{eqnarray*}
\langle M^{2}\rangle_t=\int_0^{t}\overline{(\partial_y\Phi_g)(\partial_y\Phi_g)^*}(\bar{X}_s,\mathscr{L}_{\bar{X}_s})  ds ,\quad t\in[0,T].
\end{eqnarray*}

The remaining proofs are divided into three steps.

\vspace{3mm}
\textbf{Step 1:} In this step, we claim that $F$ is locally Lipschitz  with respect to $y$ and $\bar{F}$ is globally Lipschitz.

Indeed, (\ref{E3}) it implies that $\partial_y\Phi(x,\mu,y)$ is bounded and Lipschitz continuous w.r.t.  $(x,\mu,y)$, thus by condition ${\mathbf{A\ref{A1}}}$ and  (\ref{32}), we have that $\partial_y\Phi_g$ is locally Lipschitz continuous. Thus by a straightforward computation, we can prove $(\partial_y\Phi_g)(\partial_y\Phi_g)^*$ is also locally Lipschitz continuous. More precisely, there is a constant $C>0$ such that for any $x_1,x_2\in\RR^n,\mu_1,\mu_2\in\mathscr{P}_2,y_1,y_2\in\RR^m$,
\begin{eqnarray*}
\|F(x_1,\mu_1,y_1)-F(x_2,\mu_2,y_2)\|\leq C(1+|y_1|^2+|y_2|^2)\big(|x_1-x_2|+\mathbb{W}_2(\mu_1,\mu_2)+|y_1-y_2|\big).
\end{eqnarray*}

Under the conditions \eref{sm} and \eref{32}, we have for any $y\in \RR^m$,
\begin{eqnarray}
\sup_{x\in \RR^n, \mu\in\mathscr{P}_2}\tilde \EE|Y_{t}^{x,\mu,y}|^6\leq e^{-\beta t }|y|^6+C,\quad \sup_{x\in \RR^n, \mu\in\mathscr{P}_2}\int_{\RR^m}|y|^6\nu^{x,\mu}(dy)<\infty,\label{F4.18}
\end{eqnarray}
where $\beta\in (0,\gamma)$ and for any $x_1,x_2\in\RR^n,\mu_1,\mu_2\in\mathscr{P}_2$,
\begin{eqnarray}
\tilde{\mathbb{E}}\left |Y^{ x_1,\mu_1,y_1}_t-Y^{x_2,\mu_2,y_2}_t\right |^2\leq e^{-\beta t}|y_1-y_2|^2+C_T\left[|x_1-x_2|^2+\mathbb{W}_2(\mu_1,\mu_2)^2\right].\label{F4.19}
\end{eqnarray}
Then by \eref{F4.18} and \eref{F4.19} and the definition of an invariant measure,  we have for any $t\geq 0$,
\begin{eqnarray}
\!\!\!\!\!\!\!\!&&\left\|\tilde{\EE}F(x,\mu,Y_t^{x,\mu,y})-\bar{F}(x,\mu)\right\|\nonumber\\
=\!\!\!\!\!\!\!\!&&\left\|\tilde \EE F( x, \mu, Y^{x,\mu,y}_t)-\int_{\RR^m}F(x,\mu, z)\nu^{x,\mu}(dz)\right\|\nonumber\\
=\!\!\!\!\!\!\!\!&& \left\|\int_{\RR^m}\left[\tilde \EE F(x, \mu,Y^{x,\mu,y}_t)-\tilde \EE F(x, \mu,Y^{x,\mu,z}_t)\right]\nu^{x,\mu}(dz)\right\|\nonumber\\
\leq\!\!\!\!\!\!\!\!&& C\int_{\RR^m}\!\!\!\!\tilde \EE\left[(1+| Y^{x,\mu,y}_t|^2+|Y^{x,\mu,z}_t|^2)\left| Y^{x,\mu,y}_t-Y^{x,\mu,z}_t\right|\right]\nu^{x,\mu}(dz)\nonumber\\
\leq\!\!\!\!\!\!\!\!&& C e^{-\frac{\beta t}{2}}\int_{\RR^m}(1+|y|^2+|z|^2)|y-z|\nu^{x,\mu}(dz)\nonumber\\
\leq\!\!\!\!\!\!\!\!&&C e^{-\frac{\beta t}{2}}\left\{1+|y|^3\right\}.\label{Ergodicity}
\end{eqnarray}
By \eref{Ergodicity}, it follows that for any $t>0$,
\begin{eqnarray*}
\|\bar{F}(x_1,\mu_1)-\bar{F}(x_2,\mu_2)\|=\!\!\!\!\!\!\!\!&&\left\|\bar{F}(x_1,\mu_1)-\tilde{\EE}F(x_1,\mu_1,Y_t^{x_1,\mu_1,y})\right\|\\
\!\!\!\!\!\!\!\!&&+\left\|\tilde{\EE}F(x_2,\mu_2,Y_t^{x_2,\mu_2,y})-\bar{F}(x_2,\mu_2)\right\|\\
\!\!\!\!\!\!\!\!&&+\left\|\tilde{\EE}F(x_1,\mu_1,Y_t^{x_1,\mu_1,y})-\tilde{\EE}F(x_2,\mu_2,Y_t^{x_2,\mu_2,y})\right\|\\
\leq\!\!\!\!\!\!\!\!&& Ce^{-\frac{\beta t}{2} }\big(1+|y|^3\big)+C\left[|x_1-x_2|+\mathbb{W}_2(\mu_1,\mu_2)\right].
\end{eqnarray*}
Letting $t\rightarrow \infty$, we obtain
\begin{eqnarray}
\|\bar{F}(x_1,\mu_1)-\bar{F}(x_2,\mu_2)\|\leq C\Big[|x_1-x_2|+\mathbb{W}_2(\mu_1,\mu_2)\Big].\label{LF}
\end{eqnarray}

\textbf{Step 2:} In this step we aim to prove \eref{Key}.
To this end, we construct an auxiliary process $\bar{Y}^{\varepsilon}_t$, which is defined by
\begin{eqnarray}\label{30}
d\bar{Y}_{t}^{\varepsilon}=\frac{1}{\varepsilon} f( X_{t(\Delta)}^{\varepsilon},\mathscr{L}_{X_{t(\Delta)}^{\varepsilon}},\bar{Y}_{t}^{\varepsilon})dt+\frac{1}{\sqrt{\varepsilon}}g( X_{t(\Delta)}^{\varepsilon},\mathscr{L}_{X_{t(\Delta)}^{\varepsilon}},\bar{Y}_{t}^{\varepsilon})dW^{2}_t,\quad \bar{Y}_{0}^{\varepsilon}=\zeta,
\end{eqnarray}
where $t(\Delta):=[\frac{t}{\Delta}]\Delta$ and $[s]$ denotes the integer part of $s$.

Following from \cite[Lemmas 3.3, 3.4]{RSX1} with slight modification, we easily obtain that for any $T>0$, there exists  a constant $C_{T}>0$ such that
\begin{eqnarray}
&&\sup_{\varepsilon\in(0,1)}\sup_{t\in [0, T]}\mathbb{E}|\bar Y_{t}^{\varepsilon}|^{6}\leq C_{T}(1+\EE|\xi|^{6}+\EE|\zeta|^{6}),\label{F4.21}\\
&&\sup_{\varepsilon\in(0,1)}\sup_{t\in[0,T]}\mathbb{E}|Y_{t}^{\varepsilon}-\bar{Y}_{t}^{\varepsilon}|^{6}\leq C_{T}(1+\EE|\xi|^6+\EE|\zeta|^6)\Delta^{3}.\label{F4.22}
\end{eqnarray}
Note that
\begin{eqnarray}\label{15}
\!\!\!\!\!\!\!\!&&\int_0^t F\left(X_{s}^{\varepsilon},\mathscr{L}_{X^{\varepsilon}_{s}},Y_{s}^{\varepsilon}\right)ds-\int_0^{t}\bar{F}\left(\bar{X}_s,\mathscr{L}_{\bar{X}_s}\right)  ds
\nonumber \\
=\!\!\!\!\!\!\!\!&&\int_{0} ^{t}F\left(X^{\varepsilon}_s,\mathscr{L}_{X^{\varepsilon}_s},Y^{\varepsilon}_s\right)
-F\left(X^{\varepsilon}_{s(\Delta)},\mathscr{L}_{X^{\varepsilon}_{s(\Delta)}},\bar{Y}^{\varepsilon}_s\right) ds
\nonumber \\
 \!\!\!\!\!\!\!\!&&+ \int_{0} ^{t} F\left(X^{\varepsilon}_{s(\Delta)},\mathscr{L}_{X^{\varepsilon}_{s(\Delta)}},\bar{Y}^{\varepsilon}_s\right)-\bar{F}\left(X^{\varepsilon}_{s(\Delta)},\mathscr{L}_{X^{\varepsilon}_{s(\Delta)}}\right) ds\nonumber \\
 \!\!\!\!\!\!\!\!&& + \int_{0} ^{t} \bar{F}\left(X^{\varepsilon}_{s(\Delta)},\mathscr{L}_{X^{\varepsilon}_{s(\Delta)}}\right)-
 \bar{F}\left(X^{\varepsilon}_{s},\mathscr{L}_{X^{\varepsilon}_s}\right) ds\nonumber \\
 \!\!\!\!\!\!\!\!&& + \int_{0} ^{t}{}
 \bar{F}\left(X^{\varepsilon}_{s},\mathscr{L}_{X_s^{\varepsilon}}\right)
 -\bar{F}\left(\bar{X}_{s},\mathscr{L}_{\bar{X}_{s}}\right) ds\nonumber \\
=:\!\!\!\!\!\!\!\!&&\mathscr{I}_1(t)+\mathscr{I}_2(t)+\mathscr{I}_3(t)+\mathscr{I}_4(t).
\end{eqnarray}
Next we shall estimate the terms $\mathscr{I}_i(t)$, $i=1,2,3,4$, respectively.

Using Lemma \ref{PMY}, \eref{F4.21} and \eref{F4.22}, we get
\begin{eqnarray}
\!\!\!\!\!\!\!\!&&\EE\|\mathscr{I}_{1}(t)+\mathscr{I}_{3}(t)\|^2\nonumber \\
\leq \!\!\!\!\!\!\!\!&&
C_T\EE\big|\int_0^T\!\!\!\!\left(1+|Y_{s}^{\varepsilon}|^2+|\bar{Y}_{s}^{\varepsilon}|^2\right)\left(|X_{s}^{\varepsilon}-X_{s(\Delta)} ^{\varepsilon}|+\mathbb{W}_{2}(\mathscr{L}_{X^{\vare}_{s}},\mathscr{L}_{X_{s(\Delta)} ^{\varepsilon}})+|Y_{s}^{\varepsilon}-\bar{Y}_{s} ^{\varepsilon}|\right)ds\big|^2
\nonumber\\\leq\!\!\!\!\!\!\!\!&&
C_T\left[\EE\int_0^T\!\!\!\!\left(1+|Y_{s}^{\varepsilon}|^6+|\bar{Y}_{s}^{\varepsilon}|^6\right)ds\right]^{\frac{2}{3}}\left[\EE\int_0^T|X_{s}^{\varepsilon}-X^{\varepsilon}_{s(\Delta)}|^6ds
+C_T\EE\int_0^T|Y_{s}^{\varepsilon}-\bar{Y}_{s} ^{\varepsilon}|^6 ds\right]^{\frac{1}{3}}
\nonumber\\\leq\!\!\!\!\!\!\!\!&&
C_T\Delta(1+{\EE|\xi|^{6}}+{\EE|\zeta|^{6}}).
\label{p6}
\end{eqnarray}
By \eref{LF} and \eref{Averaging}, we obtain
\begin{eqnarray}
\EE\|\mathscr{I}_{4}(t)\|^2\leq \!\!\!\!\!\!\!\!&&C_T\EE\int_0^T|X_{s}^{\varepsilon}-\bar{X}_s|^2ds\nonumber \\
\leq\!\!\!\!\!\!\!\!&&C_T\varepsilon(1+{\EE|\xi|^{2}}+{\EE|\zeta|^{2}}).\label{p7}
\end{eqnarray}
We now focus on the term $\mathscr{I}_2(t)$. Notice that
\begin{eqnarray}
 \|\mathscr{I}_2(t)\|^2=\!\!\!\!\!\!\!\!&&\left\|\sum_{k=0}^{[t/\Delta]-1}\int_{k\Delta} ^{(k+1)\Delta} F\left(X^{\varepsilon}_{s(\Delta)},\mathscr{L}_{X^{\varepsilon}_{s(\Delta)}},\bar{Y}^{\varepsilon}_s\right)-\bar{F}\left(X^{\varepsilon}_{s(\Delta)},\mathscr{L}_{X^{\varepsilon}_{s(\Delta)}}\right) ds\right.\nonumber \\
 \!\!\!\!\!\!\!\!&& +\left.\int_{t(\Delta)} ^{t} F\left(X^{\varepsilon}_{s(\Delta)},\mathscr{L}_{X^{\varepsilon}_{s(\Delta)}},\bar{Y}^{\varepsilon}_s\right)-\bar{F}\left(X^{\varepsilon}_{s(\Delta)},\mathscr{L}_{X^{\varepsilon}_{s(\Delta)}}\right) ds\right\|^2\nonumber \\
 \leq\!\!\!\!\!\!\!\!&&\frac{C_T}{\Delta}\sum_{k=0}^{[t/\Delta]-1}\left\|\int_{k\Delta} ^{(k+1)\Delta} F\left(X^{\varepsilon}_{s(\Delta)},\mathscr{L}_{X^{\varepsilon}_{s(\Delta)}},\bar{Y}^{\varepsilon}_s\right)-\bar{F}\left(X^{\varepsilon}_{s(\Delta)},\mathscr{L}_{X^{\varepsilon}_{s(\Delta)}}\right)ds\right\|^2\nonumber \\
 \!\!\!\!\!\!\!\!&& +2\left\|\int_{t(\Delta)} ^{t} F\left(X^{\varepsilon}_{s(\Delta)},\mathscr{L}_{X^{\varepsilon}_{s(\Delta)}},\bar{Y}^{\varepsilon}_s\right)-\bar{F}\left(X^{\varepsilon}_{s(\Delta)},\mathscr{L}_{X^{\varepsilon}_{s(\Delta)}}\right) ds\right\|^2\nonumber \\
=:\!\!\!\!\!\!\!\!&&\mathscr{O}_{1}(t)+\mathscr{O}_{2}(t).  \label{p12}
\end{eqnarray}
In view of term $\mathscr{O}_{2}(t)$, it follows that
\begin{eqnarray}
\EE\mathscr{O}_{2}(t)\leq \!\!\!\!\!\!\!\!&&
C_T\EE\int_{t(\Delta)} ^{t} \big[1+|X^{\varepsilon}_{s(\Delta)}|^6+\mathscr{L}_{X^{\varepsilon}_{s(\Delta)}}(|\cdot|^6)+|\bar{Y}^{\varepsilon}_s|^6\big]ds.
\nonumber \\
 \leq\!\!\!\!\!\!\!\!&& C_T\Delta(1+{\EE|\xi|^{6}}+{\EE|\zeta|^{6}}).\label{p8}
\end{eqnarray}

If the following estimate holds
\begin{eqnarray}
\EE\mathscr{O}_{1}(t)
\leq C_T(1+{\EE|\xi|^{6}}+{\EE|\zeta|^{6}})\frac{\varepsilon}{\Delta},\label{w3}
\end{eqnarray}
then collecting estimates (\ref{15})-(\ref{w3}),  recalling (\ref{Averaging}), yields that
\begin{eqnarray*}
&&\EE\Big\|\int_0^{t}F\left(X_{s}^{\varepsilon},\mathscr{L}_{X^{\varepsilon}_{s}},Y_{s}^{\varepsilon}\right)ds-\int_0^{t}\bar{F}\left(\bar{X}_s,\mathscr{L}_{\bar{X}_s}\right)  ds\Big\|^2
\nonumber \\
\leq\!\!\!\!\!\!\!\!&&
 C_T(1+{\EE|\xi|^{6}}+{\EE|\zeta|^{6}})(\frac{\varepsilon}{\Delta}+\Delta).
\end{eqnarray*}
Taking $\Delta=\varepsilon^{1/2}$, we infer that \eref{Key} holds.

\vspace{3mm}
\textbf{Step 3:} In this step, we intend to prove \eref{w3}. Note that
\begin{eqnarray}
\EE\mathscr{O}_{1}(t)\leq\!\!\!\!\!\!\!\!&&\frac{C_T}{\Delta}\EE\sum_{k=0}^{[T/\Delta]-1}\left\|\int_{k\Delta} ^{(k+1)\Delta} F\left(X^{\varepsilon}_{k\Delta},\mathscr{L}_{X^{\varepsilon}_{k\Delta}},\bar{Y}^{\varepsilon}_s\right)-\bar{F}\left(X^{\varepsilon}_{k\Delta},\mathscr{L}_{X^{\varepsilon}_{k\Delta}}\right) ds\right\|^2\nonumber \\
\leq\!\!\!\!\!\!\!\!&&\frac{C_T}{\Delta^2}\max_{0\leq k\leq[T/\Delta]-1}\EE\left\|\int_{k\Delta} ^{(k+1)\Delta} F\left(X^{\varepsilon}_{k\Delta},\mathscr{L}_{X^{\varepsilon}_{k\Delta}},\bar{Y}^{\varepsilon}_s\right)-\bar{F}\left(X^{\varepsilon}_{k\Delta},\mathscr{L}_{X^{\varepsilon}_{k\Delta}}\right)ds\right\|^2\nonumber \\
\leq\!\!\!\!\!\!\!\!&&\frac{C_T\varepsilon^2}{\Delta^2}\max_{0\leq k\leq[T/\Delta]-1}\EE\left\|\int_{0} ^{\frac{\Delta}{\varepsilon}} F\left(X^{\varepsilon}_{k\Delta},\mathscr{L}_{X^{\varepsilon}_{k\Delta}},\bar{Y}^{\varepsilon}_{s\varepsilon+k\Delta}\right)-\bar{F}\left(X^{\varepsilon}_{k\Delta},\mathscr{L}_{X^{\varepsilon}_{k\Delta}}\right)ds\right\|^2
\nonumber \\
\leq\!\!\!\!\!\!\!\!&&
\frac{C_T\varepsilon^2}{\Delta^2}\max_{0\leq k\leq[T/\Delta]-1}\left[\int_{0} ^{\frac{\Delta}{\varepsilon}} \int_{r} ^{\frac{\Delta}{\varepsilon}}\Psi_k(s,r)dsdr \right],\label{w1}
\end{eqnarray}
where for any $0\leq r\leq s\leq \frac{\Delta}{\varepsilon}$,
\begin{eqnarray*}
\Psi_k(s,r):=\!\!\!\!\!\!\!\!&&\EE\left[\left\langle F\left(X^{\varepsilon}_{k\Delta},\mathscr{L}_{X^{\varepsilon}_{k\Delta}},\bar{Y}^{\varepsilon}_{s\varepsilon+k\Delta}\right)-\bar{F}\left(X^{\varepsilon}_{k\Delta},\mathscr{L}_{X^{\varepsilon}_{k\Delta}}\right),
\right.\right.\nonumber \\
\!\!\!\!\!\!\!\!&&\left.\left.~~~~~~F\left(X^{\varepsilon}_{k\Delta},\mathscr{L}_{X^{\varepsilon}_{k\Delta}},\bar{Y}^{\varepsilon}_{r\varepsilon+k\Delta}\right)-\bar{F}\left(X^{\varepsilon}_{k\Delta},\mathscr{L}_{X^{\varepsilon}_{k\Delta}}\right)\right\rangle\right].
\end{eqnarray*}

For any $s\geq 0$, $\mu\in\mathscr{P}_2$, and any $\mathscr{F}_s$-measurable $\RR^n$-valued random variable $X$ and $\RR^m$-valued random variable $Y$, we consider the following equation
\begin{eqnarray}\label{p14}
\left\{ \begin{aligned}
&d\tilde{Y}_{t}=\frac{1}{\varepsilon}f(X,\mu,\tilde{Y}_{t})dt+\frac{1}{\sqrt{\varepsilon}}g(X,\mu,\tilde{Y}_t)d{{W}}_{t}^{2},\quad t\geq s,\\
&\tilde{Y}_s=Y.
\end{aligned} \right.
\end{eqnarray}
We can get that \eref{p14} has a unique solution denoted by $\tilde{Y}_t^{\varepsilon,s,X,\mu,Y}$.
By the construction of $\bar{Y}_t^\varepsilon$, for any $k\in \mathbb{N}$, we have
 $$\bar{Y}_t^\varepsilon=\tilde{Y}_t^{\varepsilon,k\Delta,X^\varepsilon_{k\Delta},
 \mathscr{L}_{X^{\varepsilon}_{k\Delta}},\bar{Y}^\varepsilon_{k\Delta}},~~
 t\in[k\Delta,(k+1)\Delta].$$
Then one can obtain that
\begin{eqnarray*}
\Psi_k(s,r)=\!\!\!\!\!\!\!\!&&\EE\left[\left\langle F\left(X^{\varepsilon}_{k\Delta},\mathscr{L}_{X^{\varepsilon}_{k\Delta}},\tilde{Y}^{\varepsilon,k\Delta,X_{k\Delta}^\varepsilon,\mathscr{L}_{X^{\varepsilon}_{k\Delta}},\bar{Y}^\varepsilon_{k\Delta}}_{s\varepsilon+k\Delta}\right)
-\bar{F}\left(X^{\varepsilon}_{k\Delta},\mathscr{L}_{X^{\varepsilon}_{k\Delta}}\right),
\right.\right.\nonumber \\
\!\!\!\!\!\!\!\!&&\left.\left.~~~~~~F\left(X^{\varepsilon}_{k\Delta},\mathscr{L}_{X^{\varepsilon}_{k\Delta}},\tilde{Y}^{\varepsilon,k\Delta,X_{k\Delta}^\varepsilon,\mathscr{L}_{X^{\varepsilon}_{k\Delta}},\bar{Y}^\varepsilon_{k\Delta}}_{r\varepsilon+k\Delta}\right)-\bar{F}\left(X^{\varepsilon}_{k\Delta},\mathscr{L}_{X^{\varepsilon}_{k\Delta}}\right)\right\rangle\right].
\end{eqnarray*}
By the fact that $X_{k\Delta}^\varepsilon$, $\bar{Y}_{k\Delta}^\varepsilon$ are $\mathscr{F}_{r\varepsilon +k\Delta}$-measurable and the Markov property, we have
\begin{eqnarray*}
\Psi_k(s,r)=\!\!\!\!\!\!\!\!&&\EE\Bigg\{\EE\Big[\big\langle F\big(X^{\varepsilon}_{k\Delta},\mathscr{L}_{X^{\varepsilon}_{k\Delta}},\tilde{Y}^{\varepsilon,k\Delta,X_{k\Delta}^\varepsilon,\mathscr{L}_{X^{\varepsilon}_{k\Delta}},\bar{Y}^\varepsilon_{k\Delta}}_{s\varepsilon+k\Delta}\big)-\bar{F}\big(X^{\varepsilon}_{k\Delta},\mathscr{L}_{X^{\varepsilon}_{k\Delta}}\big),
\nonumber\\
\!\!\!\!\!\!\!\!&&~~~~~~F\big(X^{\varepsilon}_{k\Delta},\mathscr{L}_{X^{\varepsilon}_{k\Delta}},\tilde{Y}^{\varepsilon,k\Delta,X_{k\Delta}^\varepsilon,\mathscr{L}_{X^{\varepsilon}_{k\Delta}},\bar{Y}^\varepsilon_{k\Delta}}_{r\varepsilon+k\Delta}\big)-\bar{F}\big(X^{\varepsilon}_{k\Delta},
\mathscr{L}_{X^{\varepsilon}_{k\Delta}}\big)\big\rangle\big|\mathscr{F}_{r\varepsilon+k\Delta}\Big] \Bigg\}
\nonumber\\
\leq\!\!\!\!\!\!\!\!&&
\EE\Bigg\{\left|F\big(X^{\varepsilon}_{k\Delta},\mathscr{L}_{X^{\varepsilon}_{k\Delta}},\tilde{Y}^{\varepsilon,k\Delta,X^{\varepsilon}_{k\Delta},\mathscr{L}_{X^{\varepsilon}_{k\Delta}},\bar{Y}^\varepsilon_{k\Delta}}_{r\varepsilon+k\Delta}\big)-\bar{F}\big(X^{\varepsilon}_{k\Delta},
\mathscr{L}_{X^{\varepsilon}_{k\Delta}}\big)\right|\cdot\nonumber\\
\!\!\!\!\!\!\!\!&&~~~~~~\Big|\EE F\big(x,\mathscr{L}_{X^{\varepsilon}_{k\Delta}},\tilde{Y}^{\varepsilon,r\varepsilon+k\Delta,x,\mathscr{L}_{X^{\varepsilon}_{k\Delta}},y}_{s\varepsilon+k\Delta}\big)-\bar{F}\big(x,\mathscr{L}_{X^{\varepsilon}_{k\Delta}}\big)
\Big|\Big|_{(x,y)=(X^{\varepsilon}_{k\Delta},\bar{Y}^\varepsilon_{r\varepsilon+k\Delta})}\Bigg\}.
\end{eqnarray*}
Recall the definition of $\{\tilde{Y}_{r\varepsilon+k\Delta+u\varepsilon}^{\varepsilon,r\varepsilon+k\Delta,x,\mu,y}\}_{u\geq0}$, by a time shift transformation, it follows that
\begin{eqnarray}
\tilde{Y}_{r\varepsilon+k\Delta+u\varepsilon}^{\varepsilon,r\varepsilon+k\Delta,x,\mu,y}
=\!\!\!\!\!\!\!\!&&y+\frac{1}{\varepsilon}\int_{r\varepsilon+k\Delta}^{r\varepsilon+k\Delta+u\varepsilon}f(x,\mu,\tilde{Y}_{\tau}^{\varepsilon,r\varepsilon+k\Delta,x,\mu,y})d\tau\nonumber\\
\!\!\!\!\!\!\!\!&&~~~~~~+\frac{1}{\sqrt{\varepsilon}}\int_{r\varepsilon+k\Delta}^{r\varepsilon+k\Delta+u\varepsilon}g(x,\mu,\tilde{Y}_{\tau}^{\varepsilon,r\varepsilon+k\Delta,x,\mu,y})d{{W}}_{\tau}^{2}
\nonumber \\
=\!\!\!\!\!\!\!\!&&
y+\frac{1}{\varepsilon}\int_{0}^{u\varepsilon}f(x,\mu,\tilde{Y}_{r\varepsilon+k\Delta+\tau}^{\varepsilon,r\varepsilon+k\Delta,x,\mu,y})d\tau\nonumber\\
\!\!\!\!\!\!\!\!&&~~~~~~+\frac{1}{\sqrt{\varepsilon}}\int_{0}^{u\varepsilon}g(x,\mu,\tilde{Y}_{r\varepsilon+k\Delta+\tau}^{\varepsilon,r\varepsilon+k\Delta,x,\mu,y})d{{W}}_{\tau}^{2,k\Delta}
\nonumber \\
=\!\!\!\!\!\!\!\!&&
y+\int_{0}^{u}f(x,\mu,\tilde{Y}_{r\varepsilon+k\Delta+\tau\varepsilon}^{\varepsilon,r\varepsilon+k\Delta,x,\mu,y})d\tau\nonumber\\
\!\!\!\!\!\!\!\!&&~~~~~~+\int_{0}^{u}g(x,\mu,\tilde{Y}_{r\varepsilon+k\Delta+\tau\varepsilon}^{\varepsilon,r\varepsilon+k\Delta,x,\mu,y})d{\hat{W}}_{\tau}^{2,k\Delta},\label{p20}
\end{eqnarray}
where $$\Big\{W_{\tau}^{2,k\Delta}:=W_{\tau+r\varepsilon+k\Delta}^{2}-W_{r\varepsilon+k\Delta}^{2}\Big\}_{\tau\geq0}~~\text{and}~~\Big\{\hat{W}_{\tau}^{2,k\Delta}:=\frac{1}{\sqrt{\varepsilon}}W_{\tau\varepsilon}^{2,k\Delta}\Big\}_{\tau\geq0}.$$

We remark that the solution of the frozen equation satisfies
\begin{eqnarray}
{Y}_{u}^{x,\mu,y}
=
y+\int_{0}^{u}f(x,\mu,{Y}_{\tau}^{x,\mu,y})d\tau
+\int_{0}^{u}g(x,\mu,{Y}_{\tau}^{x,\mu,y})d\tilde{W}_{\tau}^{2}.\label{p21}
\end{eqnarray}
Therefore the weak uniqueness of the solution of \eref{p20} and \eref{p21} implies that the distribution of $\{\tilde{Y}_{r\varepsilon+k\Delta+u\varepsilon}^{\varepsilon,r\varepsilon+k\Delta,x,\mu,y}\}_{u\geq 0} $ coincides in law with  $\{{Y}_{u}^{x,\mu,y}\}_{u\geq 0}$.

By \eref{F4.18} and \eref{Ergodicity}, we get that
\begin{eqnarray}
\Psi_k(s,r)\leq\!\!\!\!\!\!\!\!&&
\EE\Bigg\{\left|F\big(X^{\varepsilon}_{k\Delta},\mathscr{L}_{X^{\varepsilon}_{k\Delta}},\tilde{Y}^{\varepsilon,k\Delta,X^{\varepsilon}_{k\Delta},\mathscr{L}_{X^{\varepsilon}_{k\Delta}},\bar{Y}^\varepsilon_{k\Delta}}_{r\varepsilon+k\Delta}\big)-\bar{F}\big(X^{\varepsilon}_{k\Delta},
\mathscr{L}_{X^{\varepsilon}_{k\Delta}}\big)\right|\cdot\nonumber\\
\!\!\!\!\!\!\!\!&&~~~~~~\Big|\tilde{\EE} F\Big(x,\mathscr{L}_{X^{\varepsilon}_{k\Delta}},{Y}^{x,\mathscr{L}_{X^{\varepsilon}_{k\Delta}},y}_{s-r}\Big)-\bar{F}\Big(x,\mathscr{L}_{X^{\varepsilon}_{k\Delta}}\Big)\Big|\Big|_{(x,y)=(X^{\varepsilon}_{k\Delta},\bar{Y}^\varepsilon_{r\varepsilon+k\Delta})}\Bigg\}
\nonumber \\
\leq\!\!\!\!\!\!\!\!&&
C_T\EE\left\{\left[1+|X^{\varepsilon}_{k\Delta}|^6+\EE|X^{\varepsilon}_{k\Delta}|^6+|\bar{Y}^\varepsilon_{r\varepsilon+k\Delta}|^6\right]e^{-\frac{(s-r)\beta}{2}}\right\}
\nonumber \\
\leq\!\!\!\!\!\!\!\!&& C_T(1+{\EE|\xi|^{6}}+{\EE|\zeta|^{6}})e^{-\frac{(s-r)\beta}{2}}.\label{w2}
\end{eqnarray}
By \eref{w1} and \eref{w2}, we deduce that \eref{w3} holds. The proof is complete.   \hspace{\fill}$\Box$
\end{proof}

\section{Proof of large deviation principle}\label{sec5}
\setcounter{equation}{0}
 \setcounter{definition}{0}
In this section, we establish the LDP for the multi-scale McKean-Vlasov stochastic system (\ref{E2}). To this end, we first introduce a general criterion for the LDP, and then apply it to the case of distribution dependence.

 \subsection{Weak convergence approach}\label{sec5.1}
%

Let us define
$$\mathcal{A}=\left\lbrace  h:  h\  \text{is  $\RR^{d_1+ d_2}$-valued
 $\mathscr{F}_t$-predictable process and}\
  \int_0^T|\dot{ h}_s|^2 d s<\infty, \  \mathbb{P}\text{-a.s.}\right\rbrace, $$
  and
$$S_M=\left\lbrace  h\in \mathcal{H}_0:
\int_0^T|\dot{ h}_s|^2  ds\leq M
 \right\rbrace,$$
where $\mathcal{H}_0$ is defined by (\ref{CM1}). The set $S_M$ endowed with the weak topology is a compact Polish space, see \cite{BD,BDM} (throughout the paper we
 always consider the weak topology on $S_M$). We now define
$$\mathcal{A}_M=\Big\{ h\in\mathcal{A}:  h_{\cdot}(\omega)\in S_M, ~\mathbb{P}\text{-}a.s.\Big\}.$$

The following sufficient condition for the Laplace principle is formulated recently in \cite{MSZ}, which is a generalized version of powerful weak convergence approach established by Budhiraja et al. (cf.~\cite{BD,BDM}) and is more convenient to use in the current setting.

Let $\mathscr{E}$ be a Polish space, for any $\delta>0$, suppose that $\mathcal{G}^\delta: C([0,T]; \RR^{d_1+d_2})\rightarrow
\mathscr{E}$ is a measurable map.

\begin{hypothesis}\label{h2}
 There exists a measurable map $\mathcal{G}^0: C([0,T];
\mathbb{R}^{d_1+d_2})\rightarrow \mathscr{E}$ for which the following two conditions hold.

(i) Let $\{h^\delta\}_{\delta>0}\subset S_M$ for some $M<\infty$ such that $h^\delta$ converges weakly to element $h$ in $S_M$ as $\delta\to0$, then
$\mathcal{G}^0\big(\int_0^{\cdot}\dot{h}^\delta_s ds\big)$ converges to $\mathcal{G}^0\big(\int_0^{\cdot}\dot{h}_s ds\big)$ in $\mathscr{E}$.

(ii) Let $\{h^\delta\}_{\delta>0}\subset \mathcal{A}_M$ for
some $M<\infty$. For any $\varepsilon_0>0$, we have
$$\lim_{\delta\to 0}\mathbb{P}\Big(d\Big(\mathcal{G}^\delta\big(W_{\cdot}
+\frac{1}{\sqrt{\delta}}\int_0^{\cdot}\dot{h}^\delta_sd s\big),\mathcal{G}^0\big(\int_0^{\cdot}\dot{h}^\delta_s ds\big)\Big)>\varepsilon_0\Big)=0, $$
where $d(\cdot,\cdot)$ denotes the metric in $\mathscr{E}$.
\end{hypothesis}

Let us recall the following general result for the LDP.
\begin{lemma}\label{app1}$($\cite[Theorem 3.2]{MSZ}$)$  If
$X^\delta=\mathcal{G}^\delta(W_{\cdot})$ and Hypothesis \ref{h2}
holds, then $\{X^\delta\}_{\delta>0}$ satisfies the Large deviation
principle in $\mathscr{E}$ with the good
rate function $I$ given by
\begin{equation}\label{rf}
I(f)=\inf_{\left\{ h \in \mathcal{H}_0:\  f=\mathcal{G}^0(\int_0^\cdot
\dot{h}_t dt)\right\}}\left\lbrace\frac{1}{2}
\int_0^T|\dot{h}_t|^2 d t \right\rbrace,
\end{equation}
where infimum over an empty set is taken as $+\infty$.
\end{lemma}

Now we would like to sketch the idea of the proof of Theorem \ref{t3}.  Intuitively, as the parameter $\delta$ tends to $0$ in stochastic system (\ref{E2}), the noise term in slow equation of (\ref{E2}) vanishes, and in view of  the theory of averaging principle we can get the following differential equation
\begin{equation}\label{e7}
\frac{d\bar{X}_t^0}{dt}=\bar{b}(\bar{X}_t^0,\mathscr{L}_{\bar{X}_t^0}),~\bar{X}_0^0=x\in\RR^n,
\end{equation}
where the coefficient $\bar{b}$ is defined by (\ref{1.3}).

We can show that (\ref{e7}) admits a unique solution by the condition ${\mathbf{A\ref{A1}}}$, always denoted by $\bar{X}^0$ in this section, which satisfies $\bar{X}^0\in C([0,T];\RR^n)$. In addition, we mention that the solution $\bar{X}^0$ of (\ref{e7}) is a deterministic path and its distribution $\mathscr{L}_{\bar{X}^0_t}=\delta_{\bar{X}^0_t}$, where $\delta_{\bar{X}^0_t}$ is the Dirac measure of $\bar{X}^0_t$.

Recall (\ref{E2}) and for any $\mu\in C([0,T];\mathscr{P}_2)$ (which could be viewed as a deterministic measure flow), we consider the following reference SDE system
\begin{equation}\left\{\begin{array}{l}\label{refSDE}
\displaystyle
d \tilde{X}^{\delta,\varepsilon}_t=b^{\mu_t}(\tilde{X}^{\delta,\varepsilon}_t, \tilde{Y}^{\delta,\varepsilon}_t)dt+\sqrt{\delta}\sigma^{\mu_t}(\tilde{X}^{\delta,\varepsilon}_t)d W_t^1, \\
\displaystyle d \tilde{Y}^{\delta,\varepsilon}_t=\frac{1}{\varepsilon}f^{\mu_t}(\tilde{X}^{\delta,\varepsilon}_t, \tilde{Y}^{\delta,\varepsilon}_t)dt+\frac{1}{\sqrt{\varepsilon}}g^{\mu_t}( \tilde{X}^{\delta,\varepsilon}_t, \tilde{Y}^{\delta,\varepsilon}_t)d W_t^2,\\
\displaystyle \tilde{X}^{\delta,\varepsilon}_0=x,~\tilde{Y}^{\delta,\varepsilon}_0=y,
\end{array}\right.
\end{equation}
where $W_t^1,W_t^2$ are defined by (\ref{bm}), and  we denote $b^{\mu}(x,y)=b(x,\mu,y)$, the other terms are similar. Note that (\ref{refSDE}) is a  classical SDE, one can apply the classical Yamada-Watanabe theorem so that there exists a measurable map
$\mathcal{G}_{\mu}: C([0,T]; \RR^{d_1+d_2})\rightarrow C([0,T]; \RR^n)$ such that we have the representation
\begin{equation}\label{Yawa}
\tilde{X}^{\delta,\varepsilon}=\mathcal{G}_{\mu}\big(W_{\cdot}\big).
\end{equation}
Now we fix $\mu=\mu^{\delta,\varepsilon}:=\mathscr{L}_{X^{\delta,\varepsilon}}$, then (\ref{E2}) reduces to
\begin{equation}\left\{\begin{array}{l}\label{e8}
\displaystyle
d X^{\delta,\varepsilon}_t=b^{\mu_t}(X^{\delta,\varepsilon}_t, Y^{\delta,\varepsilon}_t)dt+\sqrt{\delta}\sigma^{\mu_t}(X^{\delta,\varepsilon}_t)d W_t^1, \\
\displaystyle d Y^{\delta,\varepsilon}_t=\frac{1}{\varepsilon}f^{\mu_t}(X^{\delta,\varepsilon}_t, Y^{\delta,\varepsilon}_t)dt+\frac{1}{\sqrt{\varepsilon}}g^{\mu_t}( X^{\delta,\varepsilon}_t, Y^{\delta,\varepsilon}_t)d W_t^2,\\
\displaystyle X^{\delta,\varepsilon}_0=x,~Y^{\delta,\varepsilon}_0=y.
\end{array}\right.
\end{equation}
We observe that $X^{\delta,\varepsilon}_t$ is also a solution of system (\ref{refSDE}) with $\mu=\mu^{\delta,\varepsilon}$, then
by the strong uniqueness of system (\ref{refSDE}), in this case
$$X^{\delta,\varepsilon}_t=\tilde{X}^{\delta,\varepsilon}_t,~t\in[0,T].$$
Therefore, by the representation (\ref{Yawa}), we obtain
$$X^{\delta,\varepsilon}=\mathcal{G}_{\mu}\big(W_{\cdot}\big)=\mathcal{G}_{\mu^{\delta,\varepsilon}}\big(W_{\cdot}\big).$$

\begin{remark}
Note that the classical Yamada-Watanabe theorem  is not directly applicable to McKean-Vlasov SDEs. However, under the transformation above,  one can freeze the distribution and apply the known results of
classical SDEs  to the distribution dependent setting. This is often called the ``decoupled method'' in the investigation of the McKean-Vlasov SDEs.
\end{remark}

For simplicity of notation, we denote $\mathcal{G}^{\delta}:=\mathcal{G}_{\mu^{\delta,\varepsilon}}$. Then for any $ h^\delta\in \mathcal{A}_M$, let us define
$$X^{\delta,\varepsilon, h^\delta}:=\mathcal{G}^\delta\Big(W_{\cdot}
+\frac{1}{\sqrt{\delta}}\int_0^{\cdot}\dot{ h}^\delta_sd s\Big),$$
then process $X^{\delta,\varepsilon, h^\delta}$ is the the first component of the following stochastic control problem
\begin{equation}\label{e9}
\left\{ \begin{aligned}
&dX^{\delta,\varepsilon,h^\delta}_t=b(X^{\delta,\varepsilon,h^\delta}_t,\mathscr{L}_{X^{\delta,\varepsilon}_t},Y^{\delta,\varepsilon,h^\delta}_t)dt
+\sigma(X^{\delta,\varepsilon,h^\delta}_t,\mathscr{L}_{X^{\delta,\varepsilon}_t})P_1\dot{h}^\delta_t dt+\sqrt{\delta}\sigma(X^{\delta,\varepsilon,h^\delta}_t,\mathscr{L}_{X^{\delta,\varepsilon}_t})dW_t^1,\\
&dY^{\delta,\varepsilon,h^\delta}_t=\frac{1}{\varepsilon}f(X^{\delta,\varepsilon,h^\delta}_t,\mathscr{L}_{X^{\delta,\varepsilon}_t},Y^{\delta,\varepsilon,h^\delta}_t)dt
+\frac{1}{\sqrt{\delta\varepsilon}}g(X^{\delta,\varepsilon,h^\delta}_t,\mathscr{L}_{X^{\delta,\varepsilon}_t},Y^{\delta,\varepsilon,h^\delta}_t)P_2\dot{h}^\delta_t dt\\
&~~~~~~~~~~~~~+\frac{1}{\sqrt{\varepsilon}}g(X^{\delta,\varepsilon,h^\delta}_t,\mathscr{L}_{X^{\delta,\varepsilon}_t},Y^{\delta,\varepsilon,h^\delta}_t)d W_t^2,\\
&X^{\delta,\varepsilon,h^\delta}_0=x, Y^{\delta,\varepsilon,h^\delta}_0=y.
\end{aligned} \right.
\end{equation}

We are now in the position to define the following skeleton equation w.r.t.~the slow component of stochastic system (\ref{E2}),
\begin{equation}\label{skeleton}
\frac{d \bar{X}^{h}_t}{dt}=\bar{b}(\bar{X}^{h}_t,\mathscr{L}_{\bar{X}^{0}_t})+\sigma(\bar{X}^{h}_t,\mathscr{L}_{\bar{X}^{0}_t})P_1\dot{h}_t,~\bar{X}^{h}_0=x,
\end{equation}
where $h\in\mathcal{H}_0$, $\bar{X}^{0}$ is the solution of (\ref{e7}) and $\bar{b}$ is define by (\ref{1.3}).
The existence and uniqueness of solutions to (\ref{skeleton}) for any $h\in \mathcal{H}_0$ will be proved later (see Lemma \ref{existence of skeleton} below), which allows us to define
a map $\mathcal{G}^0: C([0,T]; \RR^{d_1+d_2})\rightarrow C([0,T]; \RR^n)$ by
\begin{equation}\label{g1}
\mathcal{G}^0(h):=\left\{ \begin{aligned}
&\bar{X}^{h},~~h\in\mathcal{H}_0,\\
&0,~~~~\text{otherwise}.
\end{aligned} \right.
\end{equation}

In what follows, we intend to verify the weak convergence criterion (i.e.~Hypothesis \ref{h2}) for the above mentioned maps $\mathcal{G}^\delta$ and $\mathcal{G}^0$ with $\mathscr{E}:=C([0,T];\RR^n)$.

\subsection{A priori estimates}

In this subsection, we show the existence and uniqueness of solutions with some uniform estimates to the skeleton equation (\ref{skeleton}).
\begin{lemma}\label{existence of skeleton}
Under the assumptions ${\mathbf{A\ref{A1}}}$ and (\ref{h5}), for any $x\in\RR^n$ and $h\in\mathcal{H}_0$, there exists a unique solution to (\ref{skeleton}) satisfying
\begin{equation}\label{19}
\sup_{h\in S_M}\left\{\sup_{t\in[0,T]}|\bar{X}^h_t|^2\right\}\leq C_{T,M}(1+|x|^2+\sup_{t\in[0,T]}|\bar{X}^0_t|^2).
\end{equation}
\end{lemma}
\begin{proof}
Set $\bar{b}^{\mu}(x)=\bar{b}(x,\mu)$ and $\sigma^{\mu}_{h}(x)=\sigma(x,\mu)P_1\dot{h}$. Note that $P_1$ is a projection operator, thus $\|P_1\|\leq 1$. Then for any $x,y\in\RR^n$,
$$|\bar{b}^{\mu}(x)-\bar{b}^{\mu}(y)|\leq C|x-y|,$$
$$\|\sigma^{\mu}_h(x)-\sigma^{\mu}_h(y)\|\leq C|\dot{h}||x-y|.$$
Then equation (\ref{skeleton}) is equivalent to
\begin{equation}\label{20}
\frac{d \bar{X}^{h}_t}{dt}=\bar{b}^{\mu^0_t}(\bar{X}^{h}_t)+\sigma^{\mu^0_t}_h(\bar{X}^{h}_t),~\bar{X}^{h}_0=x,
\end{equation}
where we denote $\mu^0_t=\mathscr{L}_{\bar{X}^{0}_t}$ for simplicity.

The existence and uniqueness of solutions to (\ref{20}) follows directly from the classical fixed point arguments due to $h\in\mathcal{H}_0$, which implies the existence and uniqueness of solution to (\ref{19}). By the definition of $S_M$ and the fact that $\bar{b}(x,\mu),\sigma(x,\mu)$ are Lipschitz continuous w.r.t.  $(x,\mu)$ and $\bar{X}^0\in C([0,T];\RR^n)$, the uniform estimate (\ref{19}) is straightforward. \hspace{\fill}$\Box$
\end{proof}


Furthermore, the estimates of solution $(X^{\delta,\varepsilon,h^\delta}_t,Y^{\delta,\varepsilon,h^\delta}_t)$ to the stochastic control problem (\ref{e9}) could be derived as follows.
\begin{lemma}
Under the assumptions in Theorem \ref{t3}, for any $\{h^{\delta}\}_{\delta>0}\subset\mathcal{A}_M$, there exists a constant $C_{T,M}>0$ such that
\begin{equation}\label{23}
\mathbb{E}\Big[\sup_{t\in[0,T]}|X^{\delta,\varepsilon,h^\delta}_t|^2\Big]\leq C_{T,M}(1+|x|^2+|y|^2)
\end{equation}
and
\begin{equation}\label{24}
\mathbb{E}\int_0^T|Y^{\delta,\varepsilon,h^\delta}_t|^2dt\leq C_{T,M}(1+|x|^2+|y|^2).
\end{equation}
\end{lemma}
\begin{proof}
Using It\^{o}'s formula for $|Y^{\delta,\varepsilon,h^\delta}_t|^2$ and then taking expectation, we derive that
\begin{eqnarray*}
\frac{d}{dt}\mathbb{E}|Y^{\delta,\varepsilon,h^\delta}_t|^2
=\!\!\!\!\!\!\!\!&&\frac{2}{\varepsilon}\mathbb{E}\Big[\langle f(X^{\delta,\varepsilon,h^\delta}_t,\mathscr{L}_{X^{\delta,\varepsilon}_t},Y^{\delta,\varepsilon,h^\delta}_t),Y^{\delta,\varepsilon,h^\delta}_t\rangle\Big]+\frac{1}{\varepsilon}\mathbb{E}\|g(X^{\delta,\varepsilon,h^\delta}_t,\mathscr{L}_{X^{\delta,\varepsilon}_t},Y^{\delta,\varepsilon,h^\delta}_t)\|^2
\nonumber\\
\!\!\!\!\!\!\!\!&&+\frac{2}{\sqrt{\delta\varepsilon}}\mathbb{E}\Big[\langle g(X^{\delta,\varepsilon,h^\delta}_t,\mathscr{L}_{X^{\delta,\varepsilon}_t},Y^{\delta,\varepsilon,h^\delta}_t)P_2\dot{h}^\delta_t,Y^{\delta,\varepsilon,h^\delta}_t\rangle\Big].
\end{eqnarray*}
By the assumptions of Theorem \ref{t3}, it follows that
\begin{eqnarray*}
\!\!\!\!\!\!\!\!&&\frac{2}{\sqrt{\delta\varepsilon}}\mathbb{E}\Big[\langle g(X^{\delta,\varepsilon,h^\delta}_t,\mathscr{L}_{X^{\delta,\varepsilon}_t},Y^{\delta,\varepsilon,h^\delta}_t)P_2\dot{h}^\delta_t,Y^{\delta,\varepsilon,h^\delta}_t\rangle\Big]
\nonumber\\
~\leq\!\!\!\!\!\!\!\!&&\frac{C}{\sqrt{\delta\varepsilon}}\mathbb{E}\Big[\big(1+|X^{\delta,\varepsilon,h^\delta}_t|+\sqrt{\mathscr{L}_{X^{\delta,\varepsilon}_t}(|\cdot|^2)}\big)\|P_2\||\dot{h}^\delta_t||Y^{\delta,\varepsilon,h^\delta}_t|\Big]
\nonumber\\
~\leq\!\!\!\!\!\!\!\!&&\frac{C}{\delta}\mathbb{E}\Big[\big(1+|X^{\delta,\varepsilon,h^\delta}_t|^2+\mathscr{L}_{X^{\delta,\varepsilon}_t}(|\cdot|^2)\big)|\dot{h}^\delta_t|^2\Big]+\frac{\tilde{\beta}}{\varepsilon}\mathbb{E}|Y^{\delta,\varepsilon,h^\delta}_t|^2,
\end{eqnarray*}
where we used Young's inequality in the last step for a small constant $\tilde{\beta}\in(0,\beta)$, $\beta$ is defined in (\ref{RE3}).

Then we have
\begin{eqnarray*}
\frac{d}{dt}\mathbb{E}|Y^{\delta,\varepsilon,h^\delta}_t|^2
\leq\!\!\!\!\!\!\!\!&&-\frac{\kappa}{\varepsilon}\mathbb{E}|Y^{\delta,\varepsilon,h^\delta}_t|^2+\frac{C}{\varepsilon}\left(1+\mathbb{E}|X^{\delta,\varepsilon,h^\delta}_t|^2+\mathscr{L}_{X^{\delta,\varepsilon}_t}(|\cdot|^2)\right)
\nonumber\\
\!\!\!\!\!\!\!\!&&+\frac{C}{\delta}\mathbb{E}\Big[\big(1+|X^{\delta,\varepsilon,h^\delta}_t|^2+\mathscr{L}_{X^{\delta,\varepsilon}_t}(|\cdot|^2)\big)|\dot{h}^\delta_t|^2\Big],
\end{eqnarray*}
where $\kappa:=\beta-\tilde{\beta}>0$. The Gronwall's lemma leads to
\begin{eqnarray}\label{25}
\mathbb{E}|Y^{\delta,\varepsilon,h^\delta}_t|^2\leq\!\!\!\!\!\!\!\!&&e^{-\frac{\kappa}{\varepsilon}t}|y|^2
+\frac{C}{\varepsilon}\int_0^te^{-\frac{\kappa}{\varepsilon}(t-s)}\left(1+\mathbb{E}|X^{\delta,\varepsilon,h^\delta}_s|^2+\mathscr{L}_{X^{\delta,\varepsilon}_s}(|\cdot|^2)\right)ds
\nonumber\\
\!\!\!\!\!\!\!\!&&+\frac{C}{\delta}\int_0^te^{-\frac{\kappa}{\varepsilon}(t-s)}\mathbb{E}\Big[\big(1+|X^{\delta,\varepsilon,h^\delta}_s|^2+\mathscr{L}_{X^{\delta,\varepsilon}_s}(|\cdot|^2)\big)|\dot{h}^\delta_s|^2\Big]ds.
\end{eqnarray}
Integrating (\ref{25}) w.r.t.~$t$ from $0$ to $T$ and by Fubini's theorem,
\begin{eqnarray}\label{26}
\!\!\!\!\!\!\!\!&&\mathbb{E}\Big[\int_0^T|Y^{\delta,\varepsilon,h^\delta}_t|^2dt\Big]
\nonumber\\
~\leq\!\!\!\!\!\!\!\!&&|y|^2\int_0^Te^{-\frac{\kappa}{\varepsilon}t}dt+\frac{C}{\varepsilon}\int_0^T\int_0^te^{-\frac{\kappa}{\varepsilon}(t-s)}(1+\mathbb{E}|X^{\delta,\varepsilon,h^\delta}_s|^2+\mathbb{E}|X^{\delta,\varepsilon}_s|^2)dsdt
\nonumber\\
\!\!\!\!\!\!\!\!&&+\frac{C}{\delta}\mathbb{E}\Big[\sup_{t\in[0,T]}(1+|X^{\delta,\varepsilon,h^\delta}_t|^2+\mathbb{E}|X^{\delta,\varepsilon}_t|^2)\int_0^T\int_0^te^{-\frac{\kappa}{\varepsilon}(t-s)}|\dot{h}^\delta_s|^2dsdt\Big]
\nonumber\\
~\leq\!\!\!\!\!\!\!\!&&\frac{\varepsilon}{\kappa}|y|^2+C\mathbb{E}\Big[\int_0^T(1+|X^{\delta,\varepsilon,h^\delta}_t|^2+|X^{\delta,\varepsilon}_t|^2)dt\Big]
\nonumber\\
\!\!\!\!\!\!\!\!&&+C\Big(\frac{\varepsilon}{\delta}\Big)\mathbb{E}\Big[\sup_{t\in[0,T]}(1+|X^{\delta,\varepsilon,h^\delta}_t|^2+\mathbb{E}|X^{\delta,\varepsilon}_t|^2)\int_0^T|\dot{h}^\delta_t|^2dt\Big]
\nonumber\\
~\leq\!\!\!\!\!\!\!\!&&C_{T}(1+|y|^2)+C\mathbb{E}\int_0^T|X^{\delta,\varepsilon,h^\delta}_t|^2dt+C\mathbb{E}\int_0^T|X^{\delta,\varepsilon}_t|^2dt
\nonumber\\
\!\!\!\!\!\!\!\!&&+C_{M}\Big(\frac{\varepsilon}{\delta}\Big)\Big[1+\mathbb{E}\Big(\sup_{t\in[0,T]}|X^{\delta,\varepsilon,h^\delta}_t|^2\Big)+\sup_{t\in[0,T]}\mathbb{E}|X^{\delta,\varepsilon}_t|^2\Big],
\end{eqnarray}
where we used the fact that $h^{\delta}\in\mathcal{A}_M$.

Applying It\^{o}'s formula to $|X^{\delta,\varepsilon,h^\delta}_t|^2$, we have
\begin{eqnarray*}
\!\!\!\!\!\!\!\!&&|X^{\delta,\varepsilon,h^\delta}_t|^2
\nonumber\\
=\!\!\!\!\!\!\!\!&&|x|^2+2\int_0^t\langle b(X^{\delta,\varepsilon,h^\delta}_s,\mathscr{L}_{X^{\delta,\varepsilon}_s},Y^{\delta,\varepsilon,h^\delta}_s),X^{\delta,\varepsilon,h^\delta}_s\rangle ds
+2\int_0^t\langle \sigma(X^{\delta,\varepsilon,h^\delta}_s,\mathscr{L}_{X^{\delta,\varepsilon}_s})P_1\dot{h}^\delta_s,X^{\delta,\varepsilon,h^\delta}_s\rangle ds
\nonumber\\
\!\!\!\!\!\!\!\!&&
+\delta\int_0^t\|\sigma(X^{\delta,\varepsilon,h^\delta}_s,\mathscr{L}_{X^{\delta,\varepsilon}_s})\|^2ds
+2\sqrt{\delta}\int_0^t\langle \sigma(X^{\delta,\varepsilon,h^\delta}_s,\mathscr{L}_{X^{\delta,\varepsilon}_s})dW_s^1,X^{\delta,\varepsilon,h^\delta}_s\rangle.
\end{eqnarray*}
Due to the  condition ${\mathbf{A\ref{A1}}}$,
\begin{eqnarray}\label{27}
\mathbb{E}\Big[\sup_{t\in[0,T]}|X^{\delta,\varepsilon,h^\delta}_t|^2\Big]\leq \!\!\!\!\!\!\!\!&& |x|^2+C\mathbb{E}\int_0^T\Big(1+|X^{\delta,\varepsilon,h^\delta}_t|^2+\mathscr{L}_{X^{\delta,\varepsilon}_t}(|\cdot|^2)\Big)dt+C\mathbb{E}\int_0^T|Y^{\delta,\varepsilon,h^\delta}_t|^2dt
\nonumber\\
\!\!\!\!\!\!\!\!&&+2\mathbb{E}\int_0^T\Big|\langle \sigma(X^{\delta,\varepsilon,h^\delta}_t,\mathscr{L}_{X^{\delta,\varepsilon}_t})P_1\dot{h}^\delta_t,X^{\delta,\varepsilon,h^\delta}_t\rangle\Big|dt
\nonumber\\
\!\!\!\!\!\!\!\!&&
+2\sqrt{\delta}\mathbb{E}\Big[\sup_{t\in[0,T]}\Big|\int_0^t\langle \sigma(X^{\delta,\varepsilon,h^\delta}_s,\mathscr{L}_{X^{\delta,\varepsilon}_s})dW_s^1,X^{\delta,\varepsilon,h^\delta}_s\rangle\Big|\Big].
\end{eqnarray}
The fourth term on the right hand side of (\ref{27}) can be estimated as follows,
\begin{eqnarray}\label{28}
\!\!\!\!\!\!\!\!&&2\mathbb{E}\int_0^T\Big|\langle \sigma(X^{\delta,\varepsilon,h^\delta}_t,\mathscr{L}_{X^{\delta,\varepsilon}_t})P_1\dot{h}^\delta_t,X^{\delta,\varepsilon,h^\delta}_t\rangle\Big|dt
\nonumber\\
~\leq\!\!\!\!\!\!\!\!&&\frac{1}{4}\mathbb{E}\Big[\sup_{t\in[0,T]}|X^{\delta,\varepsilon,h^\delta}_t|^2\Big]+C\mathbb{E}\Big(\int_0^T\|\sigma(X^{\delta,\varepsilon,h^\delta}_t,\mathscr{L}_{X^{\delta,\varepsilon}_t})\|\|P_1\||\dot{h}^{\delta}_t|dt\Big)^2
\nonumber\\
~\leq\!\!\!\!\!\!\!\!&&\frac{1}{4}\mathbb{E}\Big[\sup_{t\in[0,T]}|X^{\delta,\varepsilon,h^\delta}_t|^2\Big]+C\mathbb{E}\Big[\Big(\int_0^T\|\sigma(X^{\delta,\varepsilon,h^\delta}_t,\mathscr{L}_{X^{\delta,\varepsilon}_t})\|^2dt\Big)\Big(\int_0^T|\dot{h}^{\delta}_t|^2dt\Big)\Big]
\nonumber\\
~\leq\!\!\!\!\!\!\!\!&&\frac{1}{4}\mathbb{E}\Big[\sup_{t\in[0,T]}|X^{\delta,\varepsilon,h^\delta}_t|^2\Big]+C_{M,T}+C_M\mathbb{E}\int_0^T|X^{\delta,\varepsilon,h^\delta}_t|^2dt+C_M\mathbb{E}\int_0^T|X^{\delta,\varepsilon}_t|^2dt.
\end{eqnarray}
Moreover, by Burkholder-Davis-Gundy's inequality, we obtain
\begin{eqnarray}\label{29}
\!\!\!\!\!\!\!\!&&2\sqrt{\delta}\mathbb{E}\Big[\sup_{t\in[0,T]}\Big|\int_0^t\langle \sigma(X^{\delta,\varepsilon,h^\delta}_s,\mathscr{L}_{X^{\delta,\varepsilon}_s})dW_s^1,X^{\delta,\varepsilon,h^\delta}_s\rangle\Big|\Big]
\nonumber\\
~\leq\!\!\!\!\!\!\!\!&&8\sqrt{\delta}\mathbb{E}\Big[\int_0^T\|\sigma(X^{\delta,\varepsilon,h^\delta}_t,\mathscr{L}_{X^{\delta,\varepsilon}_t})\|^2|X^{\delta,\varepsilon,h^\delta}_t|^2dt\Big]^{\frac{1}{2}}
\nonumber\\
~\leq\!\!\!\!\!\!\!\!&&\frac{1}{4}\mathbb{E}\Big[\sup_{t\in[0,T]}|X^{\delta,\varepsilon,h^\delta}_t|^2\Big]+C\mathbb{E}\int_0^T|X^{\delta,\varepsilon,h^\delta}_t|^2dt+C\mathbb{E}\int_0^T|X^{\delta,\varepsilon}_t|^2dt+C_T,
\end{eqnarray}
where we used Young's inequality in the last step.

Combining (\ref{26})-(\ref{29}) yields that
\begin{eqnarray*}
\!\!\!\!\!\!\!\!&&\mathbb{E}\Big[\sup_{t\in[0,T]}|X^{\delta,\varepsilon,h^\delta}_t|^2\Big]
\nonumber\\
~\leq \!\!\!\!\!\!\!\!&& 2|x|^2+C_{M,T}+C_{M,T}\mathbb{E}\int_0^T|X^{\delta,\varepsilon,h^\delta}_t|^2dt+C_{M,T}\mathbb{E}\int_0^T|X^{\delta,\varepsilon}_t|^2dt+C_{M,T}\mathbb{E}\int_0^T|Y^{\delta,\varepsilon,h^\delta}_t|^2dt
\nonumber\\
~\leq \!\!\!\!\!\!\!\!&& C_{M,T}(1+|y|^2+|x|^2)+C_{M,T}\mathbb{E}\int_0^T|X^{\delta,\varepsilon,h^\delta}_t|^2dt+C_{M,T}\mathbb{E}\int_0^T|X^{\delta,\varepsilon}_t|^2dt
\nonumber\\
\!\!\!\!\!\!\!\!&&+C_{M,T}\Big(\frac{\varepsilon}{\delta}\Big)\Big[1+\mathbb{E}\Big(\sup_{t\in[0,T]}|X^{\delta,\varepsilon,h^\delta}_t|^2\Big)+\sup_{t\in[0,T]}\mathbb{E}|X^{\delta,\varepsilon}_t|^2\Big].
\end{eqnarray*}
Owing to the condition $\lim_{\delta\to0}\frac{\varepsilon}{\delta}=0$, we can choose $\frac{\varepsilon}{\delta}\in(0,\frac{1}{2C_{M,T}})$ such that
\begin{eqnarray*}
\!\!\!\!\!\!\!\!&&\mathbb{E}\Big[\sup_{t\in[0,T]}|X^{\delta,\varepsilon,h^\delta}_t|^2\Big]
\nonumber\\
~\leq \!\!\!\!\!\!\!\!&&C_{M,T}(1+|y|^2+|x|^2)+C_{M,T}\mathbb{E}\int_0^T|X^{\delta,\varepsilon,h^\delta}_t|^2dt+C_{M,T}\Big[\sup_{t\in[0,T]}\mathbb{E}|X^{\delta,\varepsilon}_t|^2\Big]
\nonumber\\
~\leq \!\!\!\!\!\!\!\!&&C_{M,T}(1+|y|^2+|x|^2)+C_{M,T}\mathbb{E}\int_0^T|X^{\delta,\varepsilon,h^\delta}_t|^2dt,
\end{eqnarray*}
where in the last step we used the uniform estimate of $X^{\delta,\varepsilon}_t$ which follows directly from (\ref{X}). Then  Gronwall's lemma implies the estimate (\ref{23}).

Substituting (\ref{X}) and (\ref{23}) into (\ref{26}), one can immediately get (\ref{24}). \hspace{\fill}$\Box$
\end{proof}


The following integral estimate of time increment for $X^{\delta,\varepsilon,h^\delta}_t$ is essential in proving the criterion (ii) in Hypothesis \ref{h2}.
\begin{lemma} \label{COX1}
Suppose that the assumptions in Theorem \ref{t3} hold. For all $M > 0$, there exists
$C_{M,T} > 0$ such that for any $\{h^{
\delta}\}_{\delta>0}\subset \mathcal{A}_M$,
\begin{align}
\mathbb{E}\left[\int^{T}_0|X^{\delta,\varepsilon,h^\delta}_t-X_{t(\Delta)}^{\delta,\varepsilon,h^\delta}|^2 dt\right]\leq C_{M,T}\Delta(1+|y|^2+|x|^2).\label{F3.10}
\end{align}
\end{lemma}

\begin{proof}
By (\ref{23}), it follows that
\begin{eqnarray}\label{F3.11}
&&\mathbb{E}\left[\int^{T}_0|X^{\delta,\varepsilon,h^\delta}_t-X_{t(\Delta)}^{\delta,\varepsilon,h^\delta}|^2dt\right]\nonumber\\
=\!\!\!\!\!\!\!\!&& \mathbb{E}\left(\int^{\Delta}_0|X^{\delta,\varepsilon,h^\delta}_t-x|^2dt\right)+\mathbb{E}\left[\int^{T}_{\Delta}|X^{\delta,\varepsilon,h^\delta}_t-X_{t(\Delta)}^{\delta,\varepsilon,h^\delta}|^2dt\right]\nonumber\\
\leq\!\!\!\!\!\!\!\!&& C\left(1+|y|^2+|x|^2\right)\Delta \nonumber \\
 \!\!\!\!\!\!\!\!&& +2\mathbb{E}\left(\int^{T}_{\Delta}|X^{\delta,\varepsilon,h^\delta}_t-X_{t-\Delta}^{\delta,\varepsilon,h^\delta}|^2dt\right)+2\mathbb{E}\left(\int^{T}_{\Delta}|X_{t(\Delta)}^{\delta,\varepsilon,h^\delta}-X_{t-\Delta}^{\delta,\varepsilon,h^\delta}|^2dt\right).\label{F3.8}
\end{eqnarray}
We now focus on the second term on the right-hand side of \eref{F3.11}. Using It\^{o}'s formula, we have
\begin{eqnarray}
\!\!\!\!\!\!\!\!&&|X^{\delta,\varepsilon,h^\delta}_t-X_{t-\Delta}^{\delta,\varepsilon,h^\delta}|^{2}
\nonumber \\
~=\!\!\!\!\!\!\!\!&&2\int_{t-\Delta} ^{t}\langle b(X^{\delta,\varepsilon,h^\delta}_s,\mathscr{L}_{X^{\delta,\varepsilon}_s},Y^{\delta,\varepsilon,h^\delta}_s), X^{\delta,\varepsilon,h^\delta}_s-X_{t-\Delta}^{\delta,\varepsilon,h^\delta}\rangle ds
\nonumber \\
 \!\!\!\!\!\!\!\!&&+ 2\int_{t-\Delta} ^{t}\langle \sigma(X^{\delta,\varepsilon,h^\delta}_s,\mathscr{L}_{X^{\delta,\varepsilon}_s})P_1\dot{h}^\delta_s,X^{\delta,\varepsilon,h^\delta}_s-X_{t-\Delta}^{\delta,\varepsilon,h^\delta}\rangle ds+\delta\int_{t-\Delta} ^{t}\|\sigma(X^{\delta,\varepsilon,h^\delta}_s,\mathscr{L}_{X^{\delta,\varepsilon}_s})\|^2ds
 \nonumber \\
 \!\!\!\!\!\!\!\!&&
 +2\sqrt{\delta}\int_{t-\Delta} ^{t}\langle \sigma(X^{\delta,\varepsilon,h^\delta}_s,\mathscr{L}_{X^{\delta,\varepsilon}_s})dW_s^1, X^{\delta,\varepsilon,h^\delta}_s-X_{t-\Delta}^{\delta,\varepsilon,h^\delta}\rangle \nonumber\\
=:\!\!\!\!\!\!\!\!&&\mathscr{I}_{1}(t)+\mathscr{I}_{2}(t)+\mathscr{I}_{3}(t)+\mathscr{I}_{4}(t).  \label{F3.9}
\end{eqnarray}
In what follows, we will estimate the terms $\mathscr{I}_i(t)$, $i=1,...,4$, one by one.

For terms $\mathscr{I}_{1}(t)$ and $\mathscr{I}_3(t)$, by condition (${\mathbf{A\ref{A1}}}$), (\ref{23}) and (\ref{24}), we infer that
\begin{eqnarray}\label{REGX2}
&&\mathbb{E}\Big(\int^{T}_{\Delta}\mathscr{I}_{1}(t)dt\Big)
\nonumber\\
\leq\!\!\!\!\!\!\!\!&&C\mathbb{E}\Big[\int^{T}_{\Delta}\int_{t-\Delta} ^{t}\Big(1+|X^{\delta,\varepsilon,h^\delta}_s|+|Y_{s}^{\delta,\varepsilon,h^\delta}|+\big(\mathscr{L}_{X^{\delta,\varepsilon}_s}(|\cdot|^2)\big)^{\frac{1}{2}}\Big)
\nonumber\\
\!\!\!\!\!\!\!\!&&\cdot\Big(|X^{\delta,\varepsilon,h^\delta}_s|+|X_{t-\Delta}^{\delta,\varepsilon,h^\delta}|\Big)dsdt\Big]\nonumber\\
\leq\!\!\!\!\!\!\!\!&&C_T\Delta\mathbb{E}\Big[\sup_{s\in[0,T]}(1+|X^{\delta,\varepsilon,h^\delta}_s|^2)\Big]+C_T\Delta\Big(\int_0^T\mathbb{E}|X^{\delta,\varepsilon}_s|^2ds\Big)^{\frac{1}{2}}\Big[\mathbb{E}\Big(\sup_{s\in[0,T]}|X^{\delta,\varepsilon,h^\delta}_s|^2\Big)\Big]^{\frac{1}{2}}
\nonumber\\\!\!\!\!\!\!\!\!&&+C\mathbb{E}\Big[\sup_{s\in[0,T]}|X^{\delta,\varepsilon,h^\delta}_s|\int^T_{\Delta}\int^t_{t-\Delta}|Y^{\delta,\varepsilon,h^\delta}_s|dsdt\Big]\nonumber\\
\leq\!\!\!\!\!\!\!\!&&C_T\Delta\mathbb{E}\Big[\sup_{s\in[0,T]}(1+|X^{\delta,\varepsilon,h^\delta}_s|^2)\Big]+C_T\Delta\Big[\sup_{s\in[0,T]}\mathbb{E}|X^{\delta,\varepsilon}_s|^2\Big]
\nonumber\\\!\!\!\!\!\!\!\!&&
+C_T\Delta^{\frac{1}{2}}\Big[\mathbb{E}\Big(\sup_{s\in[0,T]}|X^{\delta,\varepsilon,h^\delta}_s|^2\Big)\Big]^{\frac{1}{2}} \Big[\mathbb{E}\Big(\int^T_{\Delta}\int^t_{t-\Delta}|Y_{s}^{\delta,\varepsilon,h^\delta}|^2dsdt\Big)\Big]^{\frac{1}{2}}\nonumber\\
\leq\!\!\!\!\!\!\!\!&&C_{T}\Delta\Big(1+|y|^2+|x|^2\Big)
\end{eqnarray}
and
\begin{eqnarray}\label{REGX2a}
\mathbb{E}\left(\int^{T}_{\Delta}\mathscr{I}_{3}(t)dt\right)\leq\!\!\!\!\!\!\!\!&&C\delta\mathbb{E}\left[\int^{T}_{\Delta}\int_{t-\Delta} ^{t}\left(1+|X^{\delta,\varepsilon,h^\delta}_s|^2+\mathscr{L}_{X^{\delta,\varepsilon}_s}(|\cdot|^2)\right)ds dt\right]\nonumber\\
\leq\!\!\!\!\!\!\!\!&&C_T\Delta\mathbb{E}\left[\sup_{s\in[0,T]}\left(1+|X^{\delta,\varepsilon,h^\delta}_s|^2+\EE|X^{\delta,\varepsilon}_s|^2\right)\right]\nonumber\\
\leq\!\!\!\!\!\!\!\!&&C_{T}\Delta\left(1+|y|^2+|x|^2\right).
\end{eqnarray}
For the term $\mathscr{I}_{2}(t)$, we deduce
\begin{eqnarray}\label{131}
\!\!\!\!\!\!\!\!&&\mathbb{E}\left(\int^{T}_{\Delta}\mathscr{I}_{2}(t)dt\right)
\nonumber\\
\leq\!\!\!\!\!\!\!\!&&C\Big[\mathbb{E}\Big(\int^{T}_{\Delta}\int_{t-\Delta} ^{t}\|\sigma(X^{\delta,\varepsilon,h^\delta}_s,\mathscr{L}_{X^{\delta,\varepsilon}_s})\|^2\|P_1\|^2|\dot{h}^\delta_s|^2dsdt\Big)\Big]^{\frac{1}{2}}
\nonumber\\
\!\!\!\!\!\!\!\!&&\cdot\Big[\mathbb{E}\Big(\int^{T}_{\Delta}\int_{t-\Delta}^{t}|X^{\delta,\varepsilon,h^\delta}_s-X_{t-\Delta}^{\delta,\varepsilon,h^\delta}|^2\Big)\Big]^{\frac{1}{2}}
\nonumber\\
\leq\!\!\!\!\!\!\!\!&&C\Big[\Delta\mathbb{E}\Big(\sup_{s\in[0,T]}\big(1+|X^{\delta,\varepsilon,h^\delta}_s|^2+\mathscr{L}_{X^{\delta,\varepsilon}_s}(|\cdot|^2)\big)\int_0^T|\dot{h}^\delta_s|^2ds\Big)\Big]^{\frac{1}{2}}
\nonumber\\
\!\!\!\!\!\!\!\!&&\cdot\Big[\Delta\mathbb{E}\Big(\int_0^T|X^{\delta,\varepsilon,h^\delta}_s|^2ds\Big)\Big]^{\frac{1}{2}}
\nonumber\\
\leq\!\!\!\!\!\!\!\!&&C_{M,T}\Delta(1+|y|^2+|x|^2).
\end{eqnarray}
For the term $\mathscr{I}_{4}(t)$, by martingale property, it follows that
\begin{eqnarray}  \label{REGX3}
\mathbb{E}\left(\int^{T}_{\Delta}\mathscr{I}_{4}(t)dt\right)=\!\!\!\!\!\!\!\!&&2\sqrt{\delta}\int^{T}_{\Delta}\mathbb{E}\left[\int_{t-\Delta} ^{t}\langle \sigma(X^{\delta,\varepsilon,h^\delta}_s,\mathscr{L}_{X^{\delta,\varepsilon}_s})dW_s^1,X^{\delta,\varepsilon,h^\delta}_s-X_{t-\Delta}^{\delta,\varepsilon,h^\delta}\rangle\right]dt \nonumber\\
=\!\!\!\!\!\!\!\!&&0.
\end{eqnarray}
Combining the estimates \eref{F3.9}-\eref{REGX3}, it follows that
\begin{eqnarray}
\mathbb{E}\left(\int^{T}_{\Delta}|X^{\delta,\varepsilon,h^\delta}_t-X_{t-\Delta}^{\delta,\varepsilon,h^\delta}|^2dt\right)\leq\!\!\!\!\!\!\!\!&&C_{M,T}\Delta(1+|y|^2+|x|^2). \label{F3.13}
\end{eqnarray}

 Next, we intend to estimate the last term on the right hand side of (\ref{F3.11}).
Using It\^{o}'s formula, we know that
\begin{eqnarray}\label{erro1}
\!\!\!\!\!\!\!\!&&|X^{\delta,\varepsilon,h^\delta}_{t(\Delta)}-X_{t-\Delta}^{\delta,\varepsilon,h^\delta}|^{2}
\nonumber \\
~=\!\!\!\!\!\!\!\!&&2\int_{t-\Delta} ^{t(\Delta)}\langle b(X^{\delta,\varepsilon,h^\delta}_s,\mathscr{L}_{X^{\delta,\varepsilon}_s},Y^{\delta,\varepsilon,h^\delta}_s), X^{\delta,\varepsilon,h^\delta}_s-X_{t-\Delta}^{\delta,\varepsilon,h^\delta}\rangle ds
\nonumber \\
 \!\!\!\!\!\!\!\!&&+ 2\int_{t-\Delta} ^{t(\Delta)}\langle \sigma(X^{\delta,\varepsilon,h^\delta}_s,\mathscr{L}_{X^{\delta,\varepsilon}_s})P_1\dot{h}^\delta_s,X^{\delta,\varepsilon,h^\delta}_s-X_{t-\Delta}^{\delta,\varepsilon,h^\delta}\rangle ds+\delta\int_{t-\Delta} ^{t(\Delta)}\|\sigma(X^{\delta,\varepsilon,h^\delta}_s,\mathscr{L}_{X^{\delta,\varepsilon}_s})\|^2ds
 \nonumber \\
 \!\!\!\!\!\!\!\!&&
 +2\sqrt{\delta}\int_{t-\Delta} ^{t(\Delta)}\langle \sigma(X^{\delta,\varepsilon,h^\delta}_s,\mathscr{L}_{X^{\delta,\varepsilon}_s})dW_s^1, X^{\delta,\varepsilon,h^\delta}_s-X_{t-\Delta}^{\delta,\varepsilon,h^\delta}\rangle \nonumber\\
=:\!\!\!\!\!\!\!\!&&\mathscr{J}_{1}(t)+\mathscr{J}_{2}(t)+\mathscr{J}_{3}(t)+\mathscr{J}_{4}(t).  \label{F3.9}
\end{eqnarray}

For terms $\mathscr{J}_{1}(t)$ and $\mathscr{J}_3(t)$, as the proof of (\ref{REGX2}) and (\ref{REGX2a}), we can deduce that
\begin{eqnarray}\label{erro2}
&&\mathbb{E}\Big(\int^{T}_{\Delta}\mathscr{J}_{1}(t)dt\Big)
\nonumber\\
\leq\!\!\!\!\!\!\!\!&&C\mathbb{E}\Big[\int^{T}_{\Delta}\int_{t-\Delta} ^{t(\Delta)}\Big(1+|X^{\delta,\varepsilon,h^\delta}_s|+|Y_{s}^{\delta,\varepsilon,h^\delta}|+\big(\mathscr{L}_{X^{\delta,\varepsilon}_s}(|\cdot|^2)\big)^{\frac{1}{2}}\Big)
\nonumber\\
\!\!\!\!\!\!\!\!&&\cdot\Big(|X^{\delta,\varepsilon,h^\delta}_s|+|X_{t-\Delta}^{\delta,\varepsilon,h^\delta}|\Big)dsdt\Big]\nonumber\\
\leq\!\!\!\!\!\!\!\!&&C\mathbb{E}\Big[\int^{T}_{\Delta}\int_{t-\Delta} ^{t}\Big(1+|X^{\delta,\varepsilon,h^\delta}_s|+|Y_{s}^{\delta,\varepsilon,h^\delta}|+\big(\mathscr{L}_{X^{\delta,\varepsilon}_s}(|\cdot|^2)\big)^{\frac{1}{2}}\Big)
\nonumber\\
\!\!\!\!\!\!\!\!&&\cdot\Big(|X^{\delta,\varepsilon,h^\delta}_s|+|X_{t-\Delta}^{\delta,\varepsilon,h^\delta}|\Big)dsdt\Big]\nonumber\\
\leq\!\!\!\!\!\!\!\!&&C_{T}\Delta\Big(1+|y|^2+|x|^2\Big)
\end{eqnarray}
and
\begin{eqnarray}\label{erro3}
\mathbb{E}\left(\int^{T}_{\Delta}\mathscr{J}_{3}(t)dt\right)\leq\!\!\!\!\!\!\!\!&&C\delta\mathbb{E}\left[\int^{T}_{\Delta}\int_{t-\Delta} ^{t(\Delta)}\left(1+|X^{\delta,\varepsilon,h^\delta}_s|^2+\mathscr{L}_{X^{\delta,\varepsilon}_s}(|\cdot|^2)\right)ds dt\right]\nonumber\\
\leq\!\!\!\!\!\!\!\!&&C\delta\mathbb{E}\left[\int^{T}_{\Delta}\int_{t-\Delta} ^{t}\left(1+|X^{\delta,\varepsilon,h^\delta}_s|^2+\mathscr{L}_{X^{\delta,\varepsilon}_s}(|\cdot|^2)\right)ds dt\right]\nonumber\\
\leq\!\!\!\!\!\!\!\!&&C_{T}\Delta\left(1+|y|^2+|x|^2\right).
\end{eqnarray}
Similarly, for the term $\mathscr{J}_{2}(t)$, we easily get
\begin{eqnarray}\label{erro4}
\mathbb{E}\left(\int^{T}_{\Delta}\mathscr{J}_{2}(t)dt\right)
\leq C_{M,T}\Delta(1+|y|^2+|x|^2).
\end{eqnarray}
For the term $\mathscr{J}_{4}(t)$, by martingale property, it follows that
\begin{eqnarray}  \label{erro5}
\mathbb{E}\left(\int^{T}_{\Delta}\mathscr{J}_{4}(t)dt\right)=\!\!\!\!\!\!\!\!&&2\sqrt{\delta}\int^{T}_{\Delta}\mathbb{E}\left[\int_{t-\Delta} ^{t(\Delta)}\langle \sigma(X^{\delta,\varepsilon,h^\delta}_s,\mathscr{L}_{X^{\delta,\varepsilon}_s})dW_s^1,X^{\delta,\varepsilon,h^\delta}_s-X_{t-\Delta}^{\delta,\varepsilon,h^\delta}\rangle\right]dt \nonumber\\
=\!\!\!\!\!\!\!\!&&0.
\end{eqnarray}
Combining the estimates (\ref{erro1})-(\ref{erro5}) implies
\begin{eqnarray}
\mathbb{E}\left(\int^{T}_{\Delta}|X_{t(\Delta)}^{\delta,\varepsilon,h^\delta}-X_{t-\Delta}^{\delta,\varepsilon,h^\delta}|^2dt\right)\leq\!\!\!\!\!\!\!\!&&C_{M,T}\Delta(1+|y|^2+|x|^2). \label{F3.14}
\end{eqnarray}

Consequently, the estimate \eref{F3.10} holds by combining \eref{F3.11}, \eref{F3.13} and \eref{F3.14}. The proof is completed. \hspace{\fill}$\Box$
\end{proof}
Similarly, it is straightforward to deduce the integral estimate of time increment for $\bar{X}^{h}$.
\begin{lemma} \label{l5}
Under the assumptions ${\mathbf{A\ref{A1}}}$ and (\ref{h5}), for any $h\in S_M$, there exists a constant $C_{M,T}>0$ such that for any  $\Delta\in(0,1)$,
\begin{align*}
\mathbb{E}\left[\int^{T}_0|\bar{X}^{h}_t-\bar{X}_{t(\Delta)}^{h}|^2 dt\right]\leq C_{M,T}\Delta(1+{|x|^{2}}).
\end{align*}
\end{lemma}
\subsection{Construction of the auxiliary process}
In this subsection, we introduce an auxiliary process $\bar{Y}_{t}^{\delta,\varepsilon}\in\RR^m$ similar to (\ref{30}), with $\bar{Y}_{0}^{\delta,\varepsilon}=Y^{\delta,\varepsilon}_{0}=Y^{\delta,\varepsilon,h^\delta}_{0}=y$, and for any $k\in \mathbb{N}$ and $t\in[k\Delta,\min\{(k+1)\Delta,T\}]$,
\begin{eqnarray}
\bar{Y}_{t}^{\delta,\varepsilon}=\bar{Y}_{k\Delta}^{\delta,\varepsilon}+\frac{1}{\varepsilon}\int_{k\Delta}^{t}
f(X_{k\Delta}^{\delta,\varepsilon,h^\delta},\mathscr{L}_{X^{\delta,\varepsilon}_{k\Delta}},\bar{Y}_{s}^{\delta,\varepsilon})ds+\frac{1}{\sqrt{\varepsilon}}\int_{k\Delta}^{t}g(X_{k\Delta}^{\delta,\varepsilon,h^\delta},\mathscr{L}_{X^{\delta,\varepsilon}_{k\Delta}},\bar{Y}_{s}^{\delta,\varepsilon})dW_s^2,\label{4.6a}
\end{eqnarray}
which is equivalent to
$$
d\bar{Y}_{t}^{\delta,\varepsilon}=\frac{1}{\varepsilon}f\left(X^{\delta,\varepsilon,h^\delta}_{t(\Delta)},\mathscr{L}_{X^{\delta,\varepsilon}_{t(\Delta)}},\bar{Y}_{t}^{\delta,\varepsilon}\right)dt+\frac{1}{\sqrt{\varepsilon}}g\left(X^{\delta,\varepsilon,h^\delta}_{t(\Delta)},\mathscr{L}_{X^{\delta,\varepsilon}_{t(\Delta)}},\bar{Y}_{t}^{\delta,\varepsilon}\right)dW_t^2.
$$

We derive the following error estimate between the process $Y^{\delta,\varepsilon,h^\delta}$ and $\bar{Y}^{\delta,\varepsilon}$.
\begin{lemma} \label{MDYa}
Under the assumptions in Theorem \ref{t3}, for any $M,T>0$, there exist some constants
$C_{T},C_{M,T}>0$ such that
\begin{eqnarray}
\sup_{\varepsilon,\delta\in(0,1)}\sup_{t\in[0,T]}\mathbb{E}|\bar{Y}_{t}^{\delta,\varepsilon}|^2\leq
C_{T}\left(1+|y|^2+|x|^2\right) \label{3.13a}
\end{eqnarray}
and
\begin{eqnarray}
\mathbb{E}\left(\int_0^{T}|Y^{\delta,\varepsilon,h^\delta}_t-\bar{Y}_{t}^{\delta,\varepsilon}|^2dt\right)\leq C_{M,T}(1+|y|^2+|x|^2)\Big(\frac{\varepsilon}{\delta}+\Delta\Big). \label{3.14}
\end{eqnarray}
\end{lemma}

\begin{proof}
The proof of \eref{3.13a} is straightforward, we only prove \eref{3.14}.
For reader's convenience, we recall that the process $Y^{\delta,\varepsilon,h^\delta}_t-\bar{Y}_{t}^{\delta,\varepsilon}$ satisfies the following equation
\begin{eqnarray}\label{e77}
\left\{ \begin{aligned}
d(Y^{\delta,\varepsilon,h^\delta}_t-\bar{Y}_{t}^{\delta,\varepsilon})=&\frac{1}{\varepsilon}\left[f\left(X^{\delta,\varepsilon,h^\delta}_{t},\mathscr{L}_{X^{\delta,\varepsilon}_{t}},{Y}_{t}^{\delta,\varepsilon,h^\delta}\right)-f\left(X^{\delta,\varepsilon,h^\delta}_{t(\Delta)},\mathscr{L}_{X^{\delta,\varepsilon}_{t(\Delta)}},\bar{Y}_{t}^{\delta,\varepsilon}\right)\right]dt\\ &+\frac{1}{\sqrt{\varepsilon}}\left[g\left(X^{\delta,\varepsilon,h^\delta}_{t},\mathscr{L}_{X^{\delta,\varepsilon}_{t}},{Y}_{t}^{\delta,\varepsilon,h^\delta}\right)-g\left(X^{\delta,\varepsilon,h^\delta}_{t(\Delta)},\mathscr{L}_{X^{\delta,\varepsilon}_{t(\Delta)}},\bar{Y}_{t}^{\delta,\varepsilon}\right)\right]dW_t^2\\
&+\frac{1}{\sqrt{\delta\varepsilon}}g\left(X^{\delta,\varepsilon,h^\delta}_{t},\mathscr{L}_{X^{\delta,\varepsilon}_{t}},{Y}_{t}^{\delta,\varepsilon,h^\delta}\right)P_2\dot{h}^{\delta}_tdt,\\
Y_{0}^{\delta,\varepsilon,h^\delta}-\bar{Y}_{0}^{\delta,\varepsilon}=0.&
\end{aligned}\right.
\end{eqnarray}
Using It\^{o}'s formula and taking expectation, we have
\begin{eqnarray}\label{f1}
\!\!\!\!\!\!\!\!&&\frac{d}{dt}\EE|Y^{\delta,\varepsilon,h^\delta}_t-\bar{Y}_{t}^{\delta,\varepsilon}|^2
\nonumber\\=\!\!\!\!\!\!\!\!&&\frac{2}{\varepsilon}\EE\big\langle f(X^{\delta,\varepsilon,h^\delta}_{t},\mathscr{L}_{X^{\delta,\varepsilon}_{t}},{Y}_{t}^{\delta,\varepsilon,h^\delta})-f(X^{\delta,\varepsilon,h^\delta}_{t(\Delta)},\mathscr{L}_{X^{\delta,\varepsilon}_{t(\Delta)}},\bar{Y}_{t}^{\delta,\varepsilon}),Y^{\delta,\varepsilon,h^\delta}_t-\bar{Y}_{t}^{\delta,\varepsilon}\big\rangle \nonumber\\
\!\!\!\!\!\!\!\!&&+\frac{2}{\sqrt{\delta\varepsilon}}\EE\big\langle g(X^{\delta,\varepsilon,h^\delta}_{t},\mathscr{L}_{X^{\delta,\varepsilon}_{t}},{Y}_{t}^{\delta,\varepsilon,h^\delta})P_2\dot{h}^{\delta}_t,Y^{\delta,\varepsilon,h^\delta}_t-\bar{Y}_{t}^{\delta,\varepsilon}\big\rangle
\nonumber\\
\!\!\!\!\!\!\!\!&&+\frac{1}{\varepsilon}\EE\left\|g(X^{\delta,\varepsilon,h^\delta}_{t},\mathscr{L}_{X^{\delta,\varepsilon}_{t}},{Y}_{t}^{\delta,\varepsilon,h^\delta})-g(X^{\delta,\varepsilon,h^\delta}_{t(\Delta)},\mathscr{L}_{X^{\delta,\varepsilon}_{t(\Delta)}},\bar{Y}_{t}^{\delta,\varepsilon})\right\|^2\nonumber\\
\nonumber\\
=:\!\!\!\!\!\!\!\!&&\mathscr{U}_1(t)+\mathscr{U}_2(t)+\mathscr{U}_3(t).
\end{eqnarray}
For the terms $\mathscr{U}_1(t)$ and $\mathscr{U}_3(t)$,
\begin{eqnarray}
\mathscr{U}_1(t)+\mathscr{U}_3(t)\leq\!\!\!\!\!\!\!\!&&-\frac{\beta}{\varepsilon}\EE|Y^{\delta,\varepsilon,h^\delta}_t-\bar{Y}_{t}^{\delta,\varepsilon}|^2+\frac{C}{\varepsilon}\EE |X^{\delta,\varepsilon,h^\delta}_t-X_{t(\Delta)}^{\delta,\varepsilon,h^\delta}|^2
\nonumber\\
\!\!\!\!\!\!\!\!&&+\frac{C}{\varepsilon}\mathbb{W}_{2}(\mathscr{L}_{X^{\delta,\varepsilon}_{t}},\mathscr{L}_{X^{\delta,\varepsilon}_{t(\Delta)}})^2.
\end{eqnarray}
Note the fact that
\begin{eqnarray}\label{f2}
\mathbb{W}_{2}(\mathscr{L}_{X^{\delta,\varepsilon}_{t}},\mathscr{L}_{X^{\delta,\varepsilon}_{t(\Delta)}})^2\leq\EE|X_t^{\delta,\varepsilon}-X_{t(\Delta)}^{\delta,\varepsilon}|^2.
\end{eqnarray}
For the term $\mathscr{U}_2(t)$, by Young's inequality,
\begin{eqnarray}\label{f3}
\mathscr{U}_2(t)\leq\!\!\!\!\!\!\!\!&&\frac{\tilde{\beta}}{\varepsilon}\EE|Y^{\delta,\varepsilon,h^\delta}_t-\bar{Y}_{t}^{\delta,\varepsilon}|^2+\frac{C}{\delta}\EE\Big[\big(1+|X^{\delta,\varepsilon,h^\delta}_t|^2+\mathscr{L}_{X^{\delta,\varepsilon}_t}(|\cdot|^2)\big)\|P_2\||\dot{h}^{\delta}_t|^2\Big],
\end{eqnarray}
where $\tilde{\beta}\in(0,\beta)$.

Combining (\ref{f1})-(\ref{f3}), we have
\begin{eqnarray*}
\frac{d}{dt}\EE|Y^{\delta,\varepsilon,h^\delta}_t-\bar{Y}_{t}^{\delta,\varepsilon}|^2
\leq\!\!\!\!\!\!\!\!&&-\frac{\kappa}{\varepsilon}\EE|Y^{\delta,\varepsilon,h^\delta}_t-\bar{Y}_{t}^{\delta,\varepsilon}|^2
+\frac{C}{\varepsilon}\EE |X^{\delta,\varepsilon,h^\delta}_t-X_{t(\Delta)}^{\delta,\varepsilon,h^\delta}|^2
\nonumber\\
\!\!\!\!\!\!\!\!&&+\frac{C}{\varepsilon}\EE|X_t^{\delta,\varepsilon}-X_{t(\Delta)}^{\delta,\varepsilon}|^2
+\frac{C}{\delta}\EE\Big[\big(1+|X^{\delta,\varepsilon,h^\delta}_t|^2+\mathscr{L}_{X^{\delta,\varepsilon}_t}(|\cdot|^2)\big)|\dot{h}^{\delta}_t|^2\Big],
\end{eqnarray*}
where we denote $\kappa=\beta-\tilde{\beta}>0$. By the Gronwall's lemma we derive
\begin{eqnarray*}
\EE|Y^{\delta,\varepsilon,h^\delta}_t-\bar{Y}_{t}^{\delta,\varepsilon}|^2
\leq\!\!\!\!\!\!\!\!&&\frac{C}{\varepsilon}\int_0^te^{-\frac{\kappa(t-s)}{\varepsilon}}\EE|X_s^{\delta,\varepsilon,h^\delta}-X_{s(\Delta)}^{\delta,\varepsilon,h^\delta}|^2ds
+\frac{C}{\varepsilon}\int_0^te^{-\frac{\kappa(t-s)}{\varepsilon}}\EE|X_s^{\delta,\varepsilon}-X_{s(\Delta)}^{\delta,\varepsilon}|^2ds
\nonumber\\
\!\!\!\!\!\!\!\!&&+\frac{C}{\delta}\int_0^te^{-\frac{\kappa(t-s)}{\varepsilon}}\EE\Big[\big(1+|X^{\delta,\varepsilon,h^\delta}_s|^2+\mathscr{L}_{X^{\delta,\varepsilon}_s}(|\cdot|^2)\big)|\dot{h}^{\delta}_s|^2\Big]ds.
\end{eqnarray*}
Due to the Fubini's theorem we have
\begin{eqnarray*}
\EE\left(\int_0^T|Y^{\delta,\varepsilon,h^\delta}_t-\bar{Y}_{t}^{\delta,\varepsilon}|^2dt\right)
\leq\!\!\!\!\!\!\!\!&& \frac{C}{\varepsilon}\EE\Big[\int_0^T|X_s^{\delta,\varepsilon,h^\delta}-X_{s(\Delta)}^{\delta,\varepsilon,h^\delta}|^2\Big(\int^T_se^{-\frac{\kappa(t-s)}{\varepsilon}}dt\Big)ds\Big]
\nonumber\\
\!\!\!\!\!\!\!\!&&
+\frac{C}{\varepsilon}\EE\Big[\int_0^T|X_s^{\delta,\varepsilon}-X_{s(\Delta)}^{\delta,\varepsilon}|^2\Big(\int^T_se^{-\frac{\kappa(t-s)}{\varepsilon}}dt\Big)ds\Big]
\nonumber\\
\!\!\!\!\!\!\!\!&&
+\frac{C}{\delta}\EE\Big[\int_0^T\big(1+|X^{\delta,\varepsilon,h^\delta}_s|^2+|X^{\delta,\varepsilon}_s|^2\big)|\dot{h}^{\delta}_s|^2\Big(\int^T_se^{-\frac{\kappa(t-s)}{\varepsilon}}dt\Big)ds\Big]
\nonumber\\
\leq\!\!\!\!\!\!\!\!&&  \frac{C}{\kappa}\EE\left[\int_0^T|X_s^{\delta,\varepsilon,h^\delta}-X_{s(\Delta)}^{\delta,\varepsilon,h^\delta}|^2ds\right]+\frac{C}{\kappa}\EE\left[\int_0^T|X_s^{\delta,\varepsilon}-X_{s(\Delta)}^{\delta,\varepsilon}|^2ds\right]
\nonumber\\
\!\!\!\!\!\!\!\!&&
+\frac{C}{\kappa}\big(\frac{\varepsilon}{\delta}\big)\EE\left[\sup_{s\in[0,T]}\left(1+|X^{\delta,\varepsilon,h^\delta}_s|^2+|X^{\delta,\varepsilon}_s|^2\right)\int_0^T|\dot{h}^{\delta}_s|^2ds\right]
\nonumber\\
\leq\!\!\!\!\!\!\!\!&&C_{M,T}(1+|x|^{2}+|y|^{2})\Big(\frac{\varepsilon}{\delta}+\Delta\Big).
\end{eqnarray*}

The proof is complete. \hspace{\fill}$\Box$
\end{proof}
\subsection{Proof of Theorem \ref{t3}}
In order to prove our second main result, we also need to use the following averaging principle result. Since we here only consider the convergence of $X^{\delta,\varepsilon}_t$ to $\bar{X}^0_t$ in $L^2(\Omega)$-sense but without the convergence rate, the condition ${\mathbf{A\ref{A1}}}$ is enough to conclude the following result.
\begin{lemma}\label{p3}
Under the assumptions ${\mathbf{A\ref{A1}}}$ and (\ref{h5}), for any $T>0$, $t\in[0,T]$, we have
\begin{align}\label{31}
\lim_{\delta\to 0}\EE|X^{\delta,\varepsilon}_t-\bar{X}^0_t|^2=0.
\end{align}
\end{lemma}
\begin{proof}
 The result follows by \cite[Theorem 5.3]{HLL4} (see also \cite[Theorem 2.3]{RSX1}) with slight modification. \hspace{\fill}$\Box$
\end{proof}

The proof of the criterions (i) and (ii) in Hypothesis \ref{h2} will be presented in Propositions \ref{t4} and \ref{t5} respectively.
\begin{proposition}\label{t4}
Under the assumptions ${\mathbf{A\ref{A1}}}$ and (\ref{h5}), let $\{h^\delta\}_{\delta>0}\subset S_M$ for any $M<\infty$ such that $h^\delta$ converges to element $h$ in $S_M$ as $\delta\to0$, then
$\mathcal{G}^0\big(\int_0^{\cdot}\dot{h}^\delta_s ds\big)$ converges to $\mathcal{G}^0\big(\int_0^{\cdot}\dot{h}_s ds\big)$ in  $C([0,T];\RR^n)$.
\end{proposition}
\begin{proof}
Let $\bar{X}^{h^\delta}=\mathcal{G}^0\big(\int_0^{\cdot}\dot{ h}^\delta_sds\big)$, then it is clear that $\bar{X}^{h^\delta}$ solves the following equation
\begin{equation*}
\frac{d \bar{X}^{h^\delta}_t}{dt}=\bar{b}(\bar{X}^{h^\delta}_t,\mathscr{L}_{\bar{X}^{0}_t})+\sigma(\bar{X}^{h^\delta}_t,\mathscr{L}_{\bar{X}^{0}_t})P_1\dot{h}_t^\delta,~\bar{X}^{h^\delta}_0=x.
\end{equation*}
If $h^\delta$ converges to an element $h$ in $S_M$ as $\delta\to0$, it suffices to show that $\bar{X}^{h^\delta}$ strongly converges to $\bar{X}^{h}$ in $C([0,T];\RR^n)$, as $\delta\to0$, which solves
\begin{equation*}
\frac{d \bar{X}^{h}_t}{dt}=\bar{b}(\bar{X}^{h}_t,\mathscr{L}_{\bar{X}^{0}_t})+\sigma(\bar{X}^{h}_t,\mathscr{L}_{\bar{X}^{0}_t})P_1\dot{h}_t,~\bar{X}^{h}_0=x.
\end{equation*}
Note that  $\bar{b}$ and $\sigma$ are globally Lipschitz w.r.t.~both two variables. Thus the next proof is standard and we omit the details to keep the length of this paper (see \cite[Proposition 5.4]{LSZZ}, where the locally Lipschitz case has been considered). \hspace{\fill}$\Box$
\end{proof}

\begin{proposition}\label{t5}
Under the assumptions in Theorem \ref{t3}, let $\{ h^\delta\}_{\delta>0}\subset \mathcal{A}_M$ for
any $M<\infty$. Then for any $\varepsilon_0>0$, we have
$$\lim_{\delta\to 0}\mathbb{P}\Big(d\Big(\mathcal{G}^\delta\big(W_{\cdot}
+\frac{1}{\sqrt{\delta}}\int_0^{\cdot}\dot{ h}^\delta_sd s\big),\mathcal{G}^0\big(\int_0^{\cdot}\dot{ h}^\delta_sd s\big)\Big)>\varepsilon_0\Big)=0, $$
where $d(\cdot,\cdot)$ denotes the metric in the space $C([0,T];\RR^n)$, i.e.,
$$d(u_1,u_2):=\sup_{t\in [0,T]}|u_1(t)-u_2(t)|,\quad u_1,u_2\in C([0,T];\RR^n).$$
\end{proposition}
\begin{proof}
The proof of Proposition \ref{t5} will be separated into the following two steps.

\vspace{2mm}
\textbf{Step 1}.
We recall that the process $X^{\delta,\varepsilon,h^\delta}_t-\bar{X}^{h^\delta}_{t}$ satisfies the equation
\begin{eqnarray*}
\left\{ \begin{aligned}
d(X^{\delta,\varepsilon,h^\delta}_t-\bar{X}_{t}^{h^\delta})=&\Big[b\Big(X^{\delta,\varepsilon,h^\delta}_t,\mathscr{L}_{X^{\delta,\varepsilon}_t},Y^{\delta,\varepsilon,h^\delta}_t\Big)-\bar{b}\Big(\bar{X}_{t}^{h^\delta},\mathscr{L}_{\bar{X}^0_{t}}\Big)\Big]dt\\ &+\Big[\sigma\Big(X^{\delta,\varepsilon,h^\delta}_{t},\mathscr{L}_{X^{\delta,\varepsilon}_{t}}\Big)-\sigma\Big(\bar{X}^{h^\delta}_{t},\mathscr{L}_{\bar{X}^{0}_{t}}\Big)\Big]P_1\dot{h}_t^\delta dt\\
&+\sqrt{\delta}\sigma\Big(X^{\delta,\varepsilon,h^\delta}_{t},\mathscr{L}_{X^{\delta,\varepsilon}_{t}}\Big)dW_t^1,\\
X_{0}^{\delta,\varepsilon,h^\delta}-\bar{X}^{h^\delta}_0=0.&
\end{aligned}\right.
\end{eqnarray*}
Using It\^{o}'s formula we have
\begin{eqnarray}\label{I11}
|X^{\delta,\varepsilon,h^\delta}_t-\bar{X}^{h^\delta}_{t}|^{2}=\!\!\!\!\!\!\!\!&& 2\int_{0} ^{t}\langle b\big(X^{\delta,\varepsilon,h^\delta}_s,\mathscr{L}_{X^{\delta,\varepsilon}_s},Y^{\delta,\varepsilon,h^\delta}_s\big)-\bar{b}\big(\bar{X}_{s}^{h^\delta},\mathscr{L}_{\bar{X}_{s}^0}\big), X^{\delta,\varepsilon,h^\delta}_s-\bar{X}_{s}^{h^\delta}\rangle ds\nonumber \\
 \nonumber \\
 \!\!\!\!\!\!\!\!&&+2\int_{0} ^{t}\Big\langle\Big[\sigma\big(X^{\delta,\varepsilon,h^\delta}_{s},\mathscr{L}_{X^{\delta,\varepsilon}_{s}}\big)-\sigma\big(\bar{X}^{h^\delta}_{s},\mathscr{L}_{\bar{X}^{0}_{s}}\big)\Big]P_1\dot{h}_s^\delta,X^{\delta,\varepsilon,h^\delta}_s-\bar{X}_{s}^{h^\delta}\Big\rangle ds
 \nonumber \\
 \!\!\!\!\!\!\!\!&& +\delta\int_{0} ^{t}\left\|\sigma\big(X^{\delta,\varepsilon,h^\delta}_{s},\mathscr{L}_{X^{\delta,\varepsilon}_{s}}\big)\right\|^2ds
\nonumber \\
 \!\!\!\!\!\!\!\!&& +2\sqrt{\delta}\int_{0} ^{t}\left\langle  \sigma\big(X^{\delta,\varepsilon,h^\delta}_{s},\mathscr{L}_{X^{\delta,\varepsilon}_{s}}\big)dW_s^1,  X^{\delta,\varepsilon,h^\delta}_s-\bar{X}_{s}^{h^\delta} \right\rangle \nonumber\\
=:\!\!\!\!\!\!\!\!&&\mathscr{K}_{1}(t)+\mathscr{K}_{2}(t)+\mathscr{K}_{3}(t)+\mathscr{K}_{4}(t).
\end{eqnarray}
For the term  $\mathscr{K}_{2}(t)$,  it follows that
\begin{eqnarray}\label{170}
\EE\Big[\sup_{t\in[0,T]}\mathscr{K}_{2}(t)\Big]
\leq \!\!\!\!\!\!\!\!&&\EE\Big[\Big(\int_{0}^{T}\|\sigma(X^{\delta,\varepsilon,h^\delta}_{s},\mathscr{L}_{X^{\delta,\varepsilon}_{s}})-\sigma(\bar{X}^{h^\delta}_{s},\mathscr{L}_{\bar{X}^{0}_{s}})\|^2|X^{\delta,\varepsilon,h^\delta}_s-\bar{X}_{s}^{h^\delta}|^2ds\Big)^{\frac{1}{2}}
\nonumber\\
\!\!\!\!\!\!\!\!&&\cdot
\Big(\int_{0}^{T}\|P_1\|^2|\dot{h}^{\delta}_s|^2ds\Big)^{\frac{1}{2}}\Big]
\nonumber\\
\leq \!\!\!\!\!\!\!\!&&C_M\EE\Big(\sup_{s\in[0,T]}|X^{\delta,\varepsilon,h^\delta}_s-\bar{X}_{s}^{h^\delta}|^2\int_{0}^{T}\|\sigma(X^{\delta,\varepsilon,h^\delta}_{s},\mathscr{L}_{X^{\delta,\varepsilon}_{s}})-\sigma(\bar{X}^{h^\delta}_{s},\mathscr{L}_{\bar{X}^{0}_{s}})\|^2ds\Big)^{\frac{1}{2}}
\nonumber\\
\leq \!\!\!\!\!\!\!\!&&
C_{M}\EE\Big[\int_{0}^{T}\Big(\mathbb{W}_{2}(\mathscr{L}_{X^{\delta,\varepsilon}_{s}},\mathscr{L}_{\bar{X}_{s}^0})^2+|X^{\delta,\varepsilon,h^\delta}_s-\bar{X}_{s}^{h^\delta}|^2\Big)ds\Big]
\nonumber\\
\!\!\!\!\!\!\!\!&&+
\frac{1}{8}\EE\Big(\sup_{s\in[0,T]}|X^{\delta,\varepsilon,h^\delta}_s-\bar{X}_{s}^{h^\delta}|^2\Big),
\end{eqnarray}
where we used the fact $h^\delta\in \mathcal{A}_M$ in the second inequality and Young's inequality in the last step.

For the term $\mathscr{K}_{3}(t)$,  a priori estimates (\ref{X}) and (\ref{23}) imply that
\begin{eqnarray}\label{150}
\EE\Big[\sup_{t\in[0,T]}\mathscr{K}_{3}(t)\Big]\leq\!\!\!\!\!\!\!\!&&C\delta\EE\int_{0} ^{T}\big(1+|X^{\delta,\varepsilon,h^\delta}_{s}|^2+\mathscr{L}_{X^{\delta,\varepsilon}_{s}}(|\cdot|^2)\big)ds
\nonumber \\
\leq\!\!\!\!\!\!\!\!&&C_{T}\delta\big(1+|x|^{2}+|y|^{2}\big).
\end{eqnarray}

For the term $\mathscr{K}_{4}(t)$, by Burkholder-Davis-Gundy's inequality, it holds that
\begin{eqnarray}
\!\!\!\!\!\!\!\!&&\EE\Big[\sup_{t\in[0,T]}\mathscr{K}_{4}(t)\Big]
\nonumber \\
\leq \!\!\!\!\!\!\!\!&&
C\sqrt{\delta}\EE\left[\int_0^T\|\sigma\big(X^{\delta,\varepsilon,h^\delta}_{t},\mathscr{L}_{X^{\delta,\varepsilon}_{t}}\big)\|^2
|X^{\delta,\varepsilon,h^\delta}_t-\bar{X}_{t}^{h^\delta}|^2dt\right]^{\frac{1}{2}}
\nonumber \\
\leq\!\!\!\!\!\!\!\!&&\frac{1}{8}\EE\Big[\sup_{t\in[0,T]}|X^{\delta,\varepsilon,h^\delta}_t-\bar{X}_{t}^{h^\delta}|^2\Big]
+C\delta\EE\int_0^T\big(|X^{\delta,\varepsilon,h^\delta}_t|^2+\mathscr{L}_{X^{\delta,\varepsilon}_{t}}(|\cdot|^2)\big)dt
\nonumber\\ \leq\!\!\!\!\!\!\!\!&&\frac{1}{8}\EE\Big[\sup_{t\in[0,T]}|X^{\delta,\varepsilon,h^\delta}_t-\bar{X}_{t}^{h^\delta}|^2\Big]
+C_{T}\delta\big(1+|x|^{2}+|y|^{2}\big),\label{p4}
\end{eqnarray}
where we used Young's inequality in the second step.

For the term $\mathscr{K}_1(t)$, we introduce the following decomposition
\begin{eqnarray}
\mathscr{K}_1(t)=\!\!\!\!\!\!\!\!&&2\int_{0} ^{t}\big\langle b\big(X^{\delta,\varepsilon,h^\delta}_s,\mathscr{L}_{X^{\delta,\varepsilon}_s},Y^{\delta,\varepsilon,h^\delta}_s\big)
-b\big(X^{\delta,\varepsilon,h^\delta}_{s(\Delta)},\mathscr{L}_{X^{\delta,\varepsilon}_{s(\Delta)}},\bar{Y}^{\delta,\varepsilon}_s\big), X^{\delta,\varepsilon,h^\delta}_s-\bar{X}_{s}^{h^\delta}\big\rangle ds\nonumber \\
\nonumber \\
 \!\!\!\!\!\!\!\!&&+ 2\int_{0} ^{t}\big\langle b\big(X^{\delta,\varepsilon,h^\delta}_{s(\Delta)},\mathscr{L}_{X^{\delta,\varepsilon}_{s(\Delta)}},\bar{Y}^{\delta,\varepsilon}_s\big)-\bar{b}\big(X^{\delta,\varepsilon,h^\delta}_{s(\Delta)},\mathscr{L}_{X^{\delta,\varepsilon}_{s(\Delta)}}\big), X^{\delta,\varepsilon,h^\delta}_s-\bar{X}_{s}^{h^\delta}\big\rangle ds\nonumber \\
 \!\!\!\!\!\!\!\!&& + 2\int_{0} ^{t}\big\langle \bar{b}\big(X^{\delta,\varepsilon,h^\delta}_{s(\Delta)},\mathscr{L}_{X^{\delta,\varepsilon}_{s(\Delta)}}\big)-
 \bar{b}\big(X^{\delta,\varepsilon,h^\delta}_{s},\mathscr{L}_{X^{\delta,\varepsilon}_s}\big), X^{\delta,\varepsilon,h^\delta}_s-\bar{X}_{s}^{h^\delta}\big\rangle ds\nonumber \\
 \!\!\!\!\!\!\!\!&& + 2\int_{0} ^{t}{}\big\langle
 \bar{b}\big(X^{\delta,\varepsilon,h^\delta}_{s},\mathscr{L}_{X_s^{\delta,\varepsilon}}\big)
 -\bar{b}\big(\bar{X}_{s}^{h^\delta},\mathscr{L}_{\bar{X}^0_{s}}\big), X^{\delta,\varepsilon,h^\delta}_s-\bar{X}_{s}^{h^\delta}\big\rangle ds\nonumber \\
=:\!\!\!\!\!\!\!\!&&\mathscr{K}_{11}(t)+\mathscr{K}_{12}(t)+\mathscr{K}_{13}(t)+\mathscr{K}_{14}(t).  \label{p5}
\end{eqnarray}
By \eref{COX}, \eref{F3.10} and \eref{3.14}, one can derive that
\begin{eqnarray}
\!\!\!\!\!\!\!\!&&\EE\Big[\sup_{t\in[0,T]}(\mathscr{K}_{11}(t)+\mathscr{K}_{13}(t))\Big]
\nonumber \\
\leq\!\!\!\!\!\!\!\!&&
C\EE\int_0^T|X^{\delta,\varepsilon,h^\delta}_t-\bar{X}_{t}^{h^\delta}|^2dt
+C\EE\int_0^T|X^{\delta,\varepsilon,h^\delta}_t-X^{\delta,\varepsilon,h^\delta}_{t(\Delta)}|^2dt
\nonumber \\
\!\!\!\!\!\!\!\!&&
+C\EE\int_0^T|Y^{\delta,\varepsilon,h^\delta}_t-\bar{Y}_{t}^{\delta,\varepsilon}|^2dt
+C\int_0^T\mathbb{W}_{2}(\mathscr{L}_{X^{\delta,\varepsilon}_t},\mathscr{L}_{X^{\delta,\varepsilon}_{t(\Delta)}})^2dt
\nonumber\\\leq\!\!\!\!\!\!\!\!&&
C_{T}(1+|x|^{2}+|y|^{2})\Big(\frac{\varepsilon}{\delta}+\Delta\Big)+C\EE\int_0^T|X^{\delta,\varepsilon,h^\delta}_t-\bar{X}_{t}^{h^\delta}|^2dt
\label{p15}
\end{eqnarray}
and
\begin{eqnarray}
\!\!\!\!\!\!\!\!&&\EE\Big[\sup_{t\in[0,T]}\mathscr{K}_{14}(t)\Big]
\nonumber \\
\leq\!\!\!\!\!\!\!\!&&
C\EE\int_0^T|X^{\delta,\varepsilon,h^\delta}_t-\bar{X}_{t}^{h^\delta}|^2dt+C\int_0^T\mathbb{W}_{2}(\mathscr{L}_{X^{\delta,\varepsilon}_{t}},\mathscr{L}_{\bar{X}_{t}^0})^2dt.
\label{p16}
\end{eqnarray}
For the term $\mathscr{K}_{12}(t)$, we again rewrite it as
\begin{eqnarray}
\mathscr{K}_{12}(t)=\!\!\!\!\!\!\!\!&& 2\int_{0} ^{t}\big\langle b\big(X^{\delta,\varepsilon,h^\delta}_{s(\Delta)},\mathscr{L}_{X^{\delta,\varepsilon}_{s(\Delta)}},\bar{Y}^{\delta,\varepsilon}_s\big)-\bar{b}\big(X^{\delta,\varepsilon,h^\delta}_{s(\Delta)},\mathscr{L}_{X^{\delta,\varepsilon}_{s(\Delta)}}\big), X^{\delta,\varepsilon,h^\delta}_s-{X}^{\delta,\varepsilon,h^\delta}_{s(\Delta)}\big\rangle ds\nonumber \\
 \!\!\!\!\!\!\!\!&& + 2\int_{0} ^{t}\big\langle b\big(X^{\delta,\varepsilon,h^\delta}_{s(\Delta)},\mathscr{L}_{X^{\delta,\varepsilon}_{s(\Delta)}},\bar{Y}^{\delta,\varepsilon}_s\big)-\bar{b}\big(X^{\delta,\varepsilon,h^\delta}_{s(\Delta)},\mathscr{L}_{X^{\delta,\varepsilon}_{s(\Delta)}}\big), X_{s(\Delta)}^{\delta,\varepsilon,h^\delta}-{\bar{X}}_{s(\Delta)}^{h^\delta}\big\rangle ds\nonumber \\
 \!\!\!\!\!\!\!\!&& + 2\int_{0} ^{t}\big\langle b\big(X^{\delta,\varepsilon,h^\delta}_{s(\Delta)},\mathscr{L}_{X^{\delta,\varepsilon}_{s(\Delta)}},\bar{Y}^{\delta,\varepsilon}_s\big)-\bar{b}\big(X^{\delta,\varepsilon,h^\delta}_{s(\Delta)},\mathscr{L}_{X^{\delta,\varepsilon}_{s(\Delta)}}\big), \bar{X}_{s(\Delta)}^{h^\delta}-{\bar{X}}_{s}^{h^\delta}\big\rangle ds\nonumber \\
=:\!\!\!\!\!\!\!\!&&\mathscr{V}_{1}(t)+\mathscr{V}_{2}(t)+\mathscr{V}_{3}(t).  \label{p17}
\end{eqnarray}
Due to (\ref{X}), (\ref{23}), (\ref{F3.10}) and (\ref{3.13a}), we can get that
\begin{eqnarray}
\!\!\!\!\!\!\!\!&&\EE\Big[\sup_{t\in[0,T]}\mathscr{V}_1(t)\Big]\nonumber\\\leq\!\!\!\!\!\!\!\!&&
C\left[\EE\int_0^T(1+|X^{\delta,\varepsilon,h^\delta}_{t(\Delta)}|^2+\mathscr{L}_{X^{\delta,\varepsilon}_{t(\Delta)}}(|\cdot|^2)+|\bar{Y}^{\delta,\varepsilon}_{t}|^2)dt\right]^{1/2}
\nonumber\\\!\!\!\!\!\!\!\!&&
\cdot\left[\EE\int_0^T|X^{\delta,\varepsilon,h^\delta}_t-{X}_{t(\Delta)}^{\delta,\varepsilon,h^\delta}|^2dt\right]^{1/2}
\nonumber\\ \leq \!\!\!\!\!\!\!\!&&
C_{T}\Delta^{\frac{1}{2}}(1+|x|^{2}+|y|^{2}).\label{p9}
\end{eqnarray}
Similarly, in view of Lemma \ref{l5} we also have
\begin{eqnarray}
\EE\Big[\sup_{t\in[0,T]}\mathscr{V}_3(t)\Big]\leq
C_{T}\Delta^{\frac{1}{2}}(1+|x|^{2}+|y|^{2}).\label{p10}
\end{eqnarray}
Consequently, combining \eref{I11}-\eref{p10}  implies
\begin{eqnarray}
\!\!\!\!\!\!\!\!&&\EE\Big[\sup_{t\in[0,T]}|X^{\delta,\varepsilon,h^\delta}_t-\bar{X}_{t}^{h^\delta}|^2\Big]
\nonumber\\
\leq \!\!\!\!\!\!\!\!&& \EE\Big[\sup_{t\in[0,T]}\mathscr{V}_2(t)\Big]+ C_{T}(\Delta^{\frac{1}{2}}+\delta+\frac{\varepsilon}{\delta})(1+{|x|^{2}}+{|y|^{2}})+
\frac{1}{4}\EE\Big[\sup_{s\in[0,T]}|X^{\delta,\varepsilon,h^\delta}_s-\bar{X}_{s}^{h^\delta}|^2\Big]
\nonumber\\
\!\!\!\!\!\!\!\!&&
+ C_{M}\EE\Big[\int_{0}^{T}\Big(\mathbb{W}_{2}(\mathscr{L}_{X^{\delta,\varepsilon}_{s}},\mathscr{L}_{\bar{X}_{s}^0})^2+|X^{\delta,\varepsilon,h^\delta}_s-\bar{X}_{s}^{h^\delta}|^2\Big)ds\Big].\label{p11}
\end{eqnarray}

\textbf{Step 2}. In this step, the main aim is to deal with the term $\mathscr{V}_2(t)$.  By dividing the time interval, we have
\begin{eqnarray}
 |\mathscr{V}_2(t)|\leq\!\!\!\!\!\!\!\!&&2\sum_{k=0}^{[t/\Delta]-1}\left|\int_{k\Delta} ^{(k+1)\Delta}\left\langle b\big(X^{\delta,\varepsilon,h^\delta}_{s(\Delta)},\mathscr{L}_{X^{\delta,\varepsilon}_{s(\Delta)}},\bar{Y}^{\delta,\varepsilon}_s\big)-\bar{b}\big(X^{\delta,\varepsilon,h^\delta}_{s(\Delta)},\mathscr{L}_{X^{\delta,\varepsilon}_{s(\Delta)}}\big), X_{s(\Delta)}^{\delta,\varepsilon,h^\delta}-{\bar{X}}_{s(\Delta)}^{h^\delta}\right\rangle ds\right|\nonumber \\
 \!\!\!\!\!\!\!\!&& +2\left|\int_{t(\Delta)} ^{t}\left\langle b\big(X^{\delta,\varepsilon,h^\delta}_{s(\Delta)},\mathscr{L}_{X^{\delta,\varepsilon}_{s(\Delta)}},\bar{Y}^{\delta,\varepsilon}_s\big)-\bar{b}\big(X^{\delta,\varepsilon,h^\delta}_{s(\Delta)},\mathscr{L}_{X^{\delta,\varepsilon}_{s(\Delta)}}\big), X_{s(\Delta)}^{\delta,\varepsilon,h^\delta}-{\bar{X}}_{s(\Delta)}^{h^\delta}\right\rangle ds\right|\nonumber \\
=:\!\!\!\!\!\!\!\!&&\mathscr{V}_{21}(t)+\mathscr{V}_{22}(t).  \label{p18}
\end{eqnarray}
Using (\ref{19}), (\ref{X}), (\ref{23}) and (\ref{3.13a}), similar to the estimate (\ref{p9}), we have
\begin{eqnarray}
\EE\big(\sup_{t\in[0,T]}\mathscr{V}_{22}(t)\big)\leq
C_T\Delta^{\frac{1}{2}}(1+|x|^{2}+|y|^{2}).\label{p13}
\end{eqnarray}
Finally, by Young's inequality, the term $\mathscr{V}_{21}(t)$ can be controlled as follows,
\begin{eqnarray}
\!\!\!\!\!\!\!\!&&\mathbb{E}\big(\sup_{t\in[0,T]}\mathscr{V}_{21}(t)\big)\nonumber\\
\leq\!\!\!\!\!\!\!\!&&2\mathbb{E}\sum_{k=0}^{[T/\Delta]-1}\Big|\int_{k\Delta} ^{(k+1)\Delta}\Big\langle b(X^{\delta,\varepsilon,h^\delta}_{k\Delta},\mathscr{L}_{X^{\delta,\varepsilon}_{k\Delta}},\bar{Y}^{\delta,\varepsilon}_s)-\bar{b}(X^{\delta,\varepsilon,h^\delta}_{k\Delta},\mathscr{L}_{X^{\delta,\varepsilon}_{k\Delta}}), X_{k\Delta}^{\delta,\varepsilon,h^\delta}-{\bar{X}}_{k\Delta}^{h^\delta}\Big\rangle ds\Big|
\nonumber \\
\leq\!\!\!\!\!\!\!\!&&\frac{C_T}{\Delta}\max_{0\leq k\leq[T/\Delta]-1}\mathbb{E}\Big|\Big\langle\int_{k\Delta} ^{(k+1)\Delta} \Big(b(X^{\delta,\varepsilon,h^\delta}_{k\Delta},\mathscr{L}_{X^{\delta,\varepsilon}_{k\Delta}},\bar{Y}^{\delta,\varepsilon}_s)-\bar{b}(X^{\delta,\varepsilon,h^\delta}_{k\Delta},\mathscr{L}_{X^{\delta,\varepsilon}_{k\Delta}})\Big)ds, \nonumber \\
&&\quad\quad\quad\quad\quad\quad\quad\quad\quad\quad X_{k\Delta}^{\delta,\varepsilon,h^\delta}-{\bar{X}}_{k\Delta}^{h^\delta}\Big\rangle \Big|
\nonumber \\
\leq\!\!\!\!\!\!\!\!&&\frac{C_T\varepsilon}{\Delta}\max_{0\leq k\leq[T/\Delta]-1}\mathbb{E}\Big[\Big|\int_{0} ^{\frac{\Delta}{\varepsilon}} \Big(b(X^{\delta,\varepsilon,h^\delta}_{k\Delta},\mathscr{L}_{X^{\delta,\varepsilon}_{k\Delta}},\bar{Y}^{\delta,\varepsilon}_{s\varepsilon+k\Delta})-\bar{b}(X^{\delta,\varepsilon,h^\delta}_{k\Delta},\mathscr{L}_{X^{\delta,\varepsilon}_{k\Delta}})\Big)ds\Big|  \nonumber \\
&&\quad\quad\quad\quad\quad\quad\quad\quad\quad\quad\cdot \Big|X_{k\Delta}^{\delta,\varepsilon,h^\delta}-{\bar{X}}_{k\Delta}^{h^\delta}\Big| \Big]
\nonumber \\
\leq\!\!\!\!\!\!\!\!&&\frac{C_T\varepsilon^2}{\Delta^2}\max_{0\leq k\leq[T/\Delta]-1}\mathbb{E}\Big|\int_{0} ^{\frac{\Delta}{\varepsilon}} \Big(b(X^{\delta,\varepsilon,h^\delta}_{k\Delta},\mathscr{L}_{X^{\delta,\varepsilon}_{k\Delta}},\bar{Y}^{\delta,\varepsilon}_{s\varepsilon+k\Delta})-\bar{b}(X^{\delta,\varepsilon,h^\delta}_{k\Delta},\mathscr{L}_{X^{\delta,\varepsilon}_{k\Delta}})\Big)ds\Big|^2 \nonumber \\
\!\!\!\!\!\!\!\!&&
+ \frac{1}{4}\mathbb{E}\Big[\sup_{t\in[0,T]}|X^{\delta,\varepsilon,h^\delta}_t-{\bar{X}}_{t}^{h^\delta}|^2\Big]
\nonumber \\
=\!\!\!\!\!\!\!\!&&
\frac{C_T\varepsilon^2}{\Delta^2}\max_{0\leq k\leq[T/\Delta]-1}\Big[\int_{0} ^{\frac{\Delta}{\varepsilon}} \int_{r} ^{\frac{\Delta}{\varepsilon}}\Psi_k(s,r)dsdr \Big]+\frac{1}{4}\mathbb{E}\Big[\sup_{t\in[0,T]}|X^{\delta,\varepsilon,h^\delta}_t-{\bar{X}}_{t}^{h^\delta}|^2\Big],\label{w6}
\end{eqnarray}
where for any $0\leq r\leq s\leq \frac{\Delta}{\varepsilon}$,
\begin{eqnarray*}
\Psi_k(s,r):=\!\!\!\!\!\!\!\!&&\EE\left[\left\langle b\big(X^{\delta,\varepsilon,h^\delta}_{k\Delta},\mathscr{L}_{X^{\delta,\varepsilon}_{k\Delta}},\bar{Y}^{\delta,\varepsilon}_{s\varepsilon+k\Delta}\big)-\bar{b}\big(X^{\delta,\varepsilon,h^\delta}_{k\Delta},\mathscr{L}_{X^{\delta,\varepsilon}_{k\Delta}}\big),
\right.\right.\nonumber \\
\!\!\!\!\!\!\!\!&&\left.\left.~~~~~~b\big(X^{\delta,\varepsilon,h^\delta}_{k\Delta},\mathscr{L}_{X^{\delta,\varepsilon}_{k\Delta}},\bar{Y}^{\delta,\varepsilon}_{r\varepsilon+k\Delta}\big)-\bar{b}\big(X^{\delta,\varepsilon,h^\delta}_{k\Delta},\mathscr{L}_{X^{\delta,\varepsilon}_{k\Delta}}\big)\right\rangle\right].
\end{eqnarray*}
Following the similar argument as in the proof of (\ref{w2}), one can easily deduce that
\begin{eqnarray}
\Psi_k(s,r)\leq\!\!\!\!\!\!\!\!&& C_T(1+|x|^{2}+|y|^{2})e^{-(s-r)\beta}.\label{w5}
\end{eqnarray}
Then, via \eref{w6} and \eref{w5}, we obtain
\begin{eqnarray}
\EE\big(\sup_{t\in[0,T]}\mathscr{V}_{21}(t)\big)
\leq\!\!\!\!\!\!\!\!&&
C_T(1+|x|^{2}+|y|^{2})
\big(\frac{\varepsilon^2}{\Delta^2}+\frac{\varepsilon}{\Delta}\big)
+\frac{1}{4}\EE\Big[\sup_{t\in[0,T]}|X^{\delta,\varepsilon,h^\delta}_t-{\bar{X}}_{t}^{h^\delta}|^2\Big].~~~\label{w7}
\end{eqnarray}
Inserting (\ref{p13}) and \eref{w7} into \eref{p11}, by Gronwall's lemma, it follows that
\begin{eqnarray}\label{w8}
\EE\Big[\sup_{t\in[0,T]}|X^{\delta,\varepsilon,h^\delta}_t-\bar{X}_{t}^{h^\delta}|^2\Big]
\leq \!\!\!\!\!\!\!\!&& C_{M,T}(1+|x|^{2}+|y|^{2})
\big(\frac{\varepsilon^2}{\Delta^2}+\frac{\varepsilon}{\Delta}+\Delta^{\frac{1}{2}}+\frac{\varepsilon}{\delta}+\delta\big)
\nonumber \\
\!\!\!\!\!\!\!\!&&
+ C_{M}\int_0^T\mathbb{W}_{2}(\mathscr{L}_{X^{\delta,\varepsilon}_{s}},\mathscr{L}_{\bar{X}_{s}^0})^2ds.
\end{eqnarray}
Therefore let $\Delta=\delta^{\frac{2}{3}}$ in \eref{w8}, we conclude
\begin{eqnarray}\label{w9}
\EE\Big[\sup_{t\in[0,T]}|X^{\delta,\varepsilon,h^\delta}_t-\bar{X}_{t}^{h^\delta}|^2\Big]\leq \!\!\!\!\!\!\!\!&& C_{M,T}(1+|x|^{2}+|y|^{2})\big(\frac{\varepsilon}{\delta}+\varepsilon^{\frac{1}{3}}+\delta\big)
\nonumber \\
\!\!\!\!\!\!\!\!&&+
C_{M}\int_0^T\EE|X^{\delta,\varepsilon}_{s}-\bar{X}_{s}^0|^2ds.
\end{eqnarray}
Applying the Chebyshev's inequality, (\ref{w9}) and (\ref{31}), for any $\varepsilon_0>0$ we have
\begin{eqnarray*}
\!\!\!\!\!\!\!\!&&\mathbb{P}\Big(d\Big(\mathcal{G}^\varepsilon\big(\sqrt{\varepsilon}W_{\cdot}
+\int_0^{\cdot}\dot{ h}^\delta_sds\big),\mathcal{G}^0\big(\int_0^{\cdot}\dot{ h}^\delta_sds\big)\Big)>\varepsilon_0\Big)
\nonumber \\
= \!\!\!\!\!\!\!\!&&\mathbb{P}\Big(\sup_{t\in[0,T]}|X^{\delta,\varepsilon,h^\delta}_t-\bar{X}_{t}^{h^\delta}|>\varepsilon_0\Big)
\nonumber \\
\leq \!\!\!\!\!\!\!\!&&\frac{\EE\left[\sup_{t\in[0,T]}|X^{\delta,\varepsilon,h^\delta}_t-\bar{X}_{t}^{h^\delta}|^2\right]}{\varepsilon_0^2}\rightarrow 0,~\text{as}~\delta\to0,
\end{eqnarray*}
where we used the Lemma \ref{p3} and the assumption (\ref{h5}) in Theorem \ref{t3} in the last step.

Hence the proof is complete. \hspace{\fill}$\Box$
\end{proof}

Now we are able to finish the proof of Theorem \ref{t3}.

\vspace{2mm}
\textbf{Proof of Theorem \ref{t3}}. Note that Propositions \ref{t4} and \ref{t5} imply conditions (i) and (ii) in Hypothesis \ref{h2} holds respectively, then by Lemma \ref{app1} we get  $\{X^{\delta,\varepsilon}\}_{\delta>0}$ satisfies the LDP in $C([0,T];\RR^n)$ with a good rate function $I$ defined in (\ref{rf}). \hspace{\fill}$\Box$

\section{Appendix}
\setcounter{equation}{0}
 \setcounter{definition}{0}

\subsection{Proof of \eref{supZvare}}\label{app2}
In this part, we give a detailed proof of \eref{supZvare}. It is sufficient to prove that
\begin{eqnarray*}
\EE\left[\sup_{t\in [0,T]}|X_{t}^{\varepsilon}-\bar{X}_{t}|^2\right]\leq C_T\big(1+\EE|\xi|^4+\EE|\zeta|^4\big)\varepsilon.
\end{eqnarray*}
In fact,  note that
\begin{eqnarray*}
X_{t}^{\varepsilon}-\bar{X}_{t}=\!\!\!\!\!\!\!\!&&\int_{0}^{t}\left[b(X_{s}^{\varepsilon},\mathscr{L}_{X_{s}^{\varepsilon}},Y_{s}^{\varepsilon})-\bar{b}(\bar{X}_{s},\mathscr{L}_{\bar X_{s}})\right]ds\\
&&+\int_{0}^{t}\left[\sigma(X^{\varepsilon}_{s},\mathscr{L}_{X_{s}^{\varepsilon}})-\sigma(\bar{X}_{s},\mathscr{L}_{\bar X_{s}})\right]dW^{1}_s\\
=\!\!\!\!\!\!\!\!&&\int_{0}^{t}\left[b(X_{s}^{\varepsilon},\mathscr{L}_{ X_{s}^{\varepsilon}},Y_{s}^{\varepsilon})-\bar{b}(X^{\varepsilon}_{s},\mathscr{L}_{X^{\varepsilon}_{s}})\right]ds\\
&&+\int_{0}^{t}\left[\bar{b}(X^{\varepsilon}_{s},\mathscr{L}_{X^{\varepsilon}_{s}})-\bar{b}(\bar{X}_{s},\mathscr{L}_{\bar{X}_{s}})\right]ds\\
&&+\int_{0}^{t}\left[\sigma(X^{\varepsilon}_{s}, \mathscr{L}_{X^{\varepsilon}_{s}})-\sigma( \bar{X}_{s},\mathscr{L}_{\bar X_{s}})\right]dW^{1}_s.
\end{eqnarray*}
By Burkholder-Davis-Gundy's inequality, for any $t\in [0, T]$ we have
\begin{eqnarray*}
\!\!\!\!\!\!\!\!&&\mathbb{E}\Big[\sup_{s\in[0,t]}\Big|\int_{0}^{s}\left[\sigma(X^{\varepsilon}_{r}, \mathscr{L}_{X^{\varepsilon}_{r}})-\sigma( \bar{X}_{r},\mathscr{L}_{\bar X_{r}})\right]dW^{1}_r\Big|^2\Big]
\nonumber\\
~\leq\!\!\!\!\!\!\!\!&&C\mathbb{E}\Big[\int_0^t\|\sigma(X^{\varepsilon}_{r}, \mathscr{L}_{X^{\varepsilon}_{r}})-\sigma( \bar{X}_{r},\mathscr{L}_{\bar X_{r}})\|^2dr\Big]
\nonumber\\
~\leq\!\!\!\!\!\!\!\!&&C_T\mathbb{E}\int_{0}^{t}\left|X_{r}^{\varepsilon}-\bar{X}_{r}\right |^2 dr.
\end{eqnarray*}
Then we know that for any $t\in [0, T]$,
\begin{eqnarray*}
\EE\left[\sup_{s\in [0, t]}\left|X_{s}^{\varepsilon}-\bar{X}_{s}\right |^2\right]\leq\!\!\!\!\!\!\!\!&&C\EE\left[\sup_{s\in[0,t]}\left|\int_{0}^{s}b(X_{r}^{\varepsilon},\mathscr{L}_{ X_{r}^{\varepsilon}},Y_{r}^{\varepsilon})-\bar{b}(X^{\varepsilon}_{r},\mathscr{L}_{X^{\varepsilon}_{r}})dr\right|^2\right]\nonumber\\
&&+C_T\EE\int_{0}^{t}\left|X_{r}^{\varepsilon}-\bar{X}_{r}\right |^2 dr.
\end{eqnarray*}
Then Grownall's inequality implies that
\begin{eqnarray*}
\EE\left[\sup_{t\in [0, T]}\left|X_{t}^{\varepsilon}-\bar{X}_{t}\right |^2\right]\leq\!\!\!\!\!\!\!\!&&C_T\EE\left[\sup_{t\in[0,T]}\left|\int_{0}^{t}b(X_{s}^{\varepsilon},\mathscr{L}_{ X_{s}^{\varepsilon}},Y_{s}^{\varepsilon})-\bar{b}( X^{\varepsilon}_{s},\mathscr{L}_{X^{\varepsilon}_{s}})ds\right|^2\right].
\end{eqnarray*}
Recall
\begin{eqnarray*}
\!\!\!\!\!\!\!\!&&\int^t_0b(X_{s}^{\varepsilon},\mathscr{L}_{ X_{s}^{\varepsilon}},Y_{s}^{\varepsilon})-\bar{b}( X^{\varepsilon}_{s},\mathscr{L}_{X^{\varepsilon}_{s}})ds
=-\int^t_0 \mathcal{L}_{2}(X_{s}^{\varepsilon},\mathscr{L}_{X^{\varepsilon}_{s}})\Phi(X_{s}^{\varepsilon},\mathscr{L}_{X^{\varepsilon}_{s}},Y^{\varepsilon}_{s})ds \\
=\!\!\!\!\!\!\!\!&&\varepsilon\Big[\Phi(\xi,\mathscr{L}_{\xi},\zeta)-\Phi(X_{t}^{\varepsilon},\mathscr{L}_{X^{\varepsilon}_{t}},Y^{\varepsilon}_{t})
+\int^t_0 \EE\left[b(X^{\varepsilon}_s,\mathscr{L}_{ X^{\varepsilon}_{s}}, Y^{\varepsilon}_s)\partial_{\mu}\Phi(x,\mathscr{L}_{X^{\varepsilon}_{s}},y)(X^{\varepsilon}_s)\right]\mid_{x=X_{s}^{\varepsilon},y=Y^{\varepsilon}_{s}}ds\\
&&+\int^t_0 \frac{1}{2}\EE \text{Tr}\left[\sigma\sigma^{*}(X^{\varepsilon}_s,\mathscr{L}_{ X^{\varepsilon}_{s}})\partial_z\partial_{\mu}\Phi(x,\mathscr{L}_{X^{\varepsilon}_{s}},y)(X^{\varepsilon}_s)\right]\mid_{x=X_{s}^{\varepsilon},y=Y^{\varepsilon}_{s}}ds\\
&&+\int^t_0 \mathcal{L}_{1}(\mathscr{L}_{X^{\varepsilon}_{s}},Y^{\varepsilon}_{s})\Phi(X_{s}^{\varepsilon},\mathscr{L}_{X^{\varepsilon}_{s}},Y^{\varepsilon}_{s})ds+M^{1,\varepsilon}_t\Big]+\sqrt{\varepsilon}M^{2,\varepsilon}_t,
\end{eqnarray*}
where $\mathcal{L}_{1}(\mathscr{L}_{X^{\varepsilon}_{s}},Y^{\varepsilon}_{s})$ is defined by (\ref{inf2}), and
\begin{eqnarray*}
&&M^{1,\varepsilon}_t=\int^t_0 \partial_x \Phi(X_{s}^{\varepsilon},\mathscr{L}_{X^{\varepsilon}_{s}},Y_{s}^{\varepsilon})\cdot \sigma(X^{\varepsilon}_s,\mathscr{L}_{ X^{\varepsilon}_{s}}) dW^1_s,\\
&&M^{2,\varepsilon}_t=\int^t_0 \partial_y \Phi(X_{s}^{\varepsilon},\mathscr{L}_{X^{\varepsilon}_{s}},Y_{s}^{\varepsilon})\cdot g(X^{\varepsilon}_s,\mathscr{L}_{ X^{\varepsilon}_{s}}, Y^{\varepsilon}_s) dW^2_s.
\end{eqnarray*}
Then by Lemma \ref{PMY} and Proposition \ref{P3.6}, we immediately obtain that
\begin{eqnarray*}
\EE\left[\sup_{t\in [0, T]}\left|X_{t}^{\varepsilon}-\bar{X}_{t}\right |^2\right]\leq C_T\big(1+\EE|\xi|^4+\EE|\zeta|^4\big)\varepsilon.
\end{eqnarray*}
We complete the proof. \hspace{\fill}$\Box$

\subsection{Continuity of $\partial_x \bar{b}$ and $\partial_\mu \bar{b}$}
In this subsection, we aim to prove the continuity of the coefficients $\partial_x \bar{b}(\cdot,\cdot)$ and $\partial_\mu \bar{b}(x,\cdot)(\cdot)$.

\begin{lemma} \label{C1}For any fixed $x\in\RR^n$, if sequence $\{(\mu_n,z_n)\}_{n\geq1}\subset \mathscr{P}_2\times \RR^n$ satisfy $\mu_n\to\mu$ in $\mathbb{W}_2$ and $|z_n- z|\to 0$ as $n\to \infty$, then we have
\begin{eqnarray}
&&\lim_{n\to\infty}\big\|\partial_\mu \bar{b}(x,\mu_n)(z_n)-\partial_\mu \bar{b}(x,\mu)(z)\big\|=0,\label{Conbarb1}\\
&&\lim_{n\to\infty}\big\|\partial_x \bar{b}(z_n,\mu_n)-\partial_x \bar{b}(z,\mu)\big\|=0.\label{Conbarb2}
\end{eqnarray}
\end{lemma}
\begin{proof}
We only prove \eref{Conbarb1} here, and the proof of \eref{Conbarb2} is very similar.

Using the definition of $\bar{b}(x,\mu)$ and the ergodicity of the frozen equation (see \eref{Ergodicity}), we have for any $(x,\mu,y)\in\RR^n\times\mathscr{P}_2\times \RR^m$,
\begin{eqnarray*}
\bar{b}(x,\mu)
=\lim_{t\to\infty}\tilde{\EE}b(x,\mu,Y^{x,\mu,y}_t),
\end{eqnarray*}
then under the condition ${\mathbf{A\ref{A2}}}$, we get
\begin{eqnarray*}
\partial_\mu \bar{b}(x,\mu)(z)
=\lim_{t\to\infty}\tilde{\EE}\Big[\partial_\mu b(x,\mu,Y^{x,\mu,y}_t)(z)+\partial_y b(x,\mu,Y^{x,\mu,y}_t)\partial_{\mu} Y^{x,\mu,y}_t(z)\Big],
\end{eqnarray*}
where $Y^{x,\mu,y}$ is the solution of frozen equation (\ref{FEQ2}) and denote its Lions derivative by $\partial_\mu Y^{x,\mu,y}_t(z)$ which exists and satisfies
\begin{equation}\left\{\begin{array}{l}\label{37}
\displaystyle
d\partial_{\mu} Y^{x,\mu,y}_t(z)=\partial_{\mu} f(x,\mu,Y^{x,\mu,y}_{t})(z)dt+\partial_y f(x,\mu,Y^{x,\mu,y}_{t})\cdot\partial_{\mu} Y^{x,\mu,y}_t(z)dt\vspace{2mm}\\
\quad\quad\quad\quad\quad\quad\displaystyle~~+\left[\partial_{\mu} g( x,\mu,Y^{x,\mu,y}_t)(z)+\partial_{y} g(x,\mu,Y^{x,\mu,y}_t)\cdot\partial_{\mu} Y^{x,\mu,y}_t(z)\right]d\tilde{W}_{t}^{2},\vspace{2mm}\nonumber\\
\partial_{\mu} Y^{x,\mu,y}_0(z)=0.\nonumber\\
\end{array}\right.
\end{equation}
By a straightforward calculation, there exists a constant $C>0$ such that
\begin{equation}\label{38}
\sup_{t\geq 0,x,z\in \RR^n,\mu\in \mathscr{P}_2,y\in\RR^m,}\tilde{\EE}\|\partial_{\mu} Y^{x,\mu,y}_t(z)\|^4\leq C.
\end{equation}
Then by the uniform boundedness of $\partial_{\mu}b(x,\mu,y)(z)$ and $\partial_y b(x,\mu,y)$, we have
\begin{eqnarray}
\sup_{x\in\RR^n,\mu\in\mathscr{P}_2,z\in\RR^n}\|\partial_\mu \bar{b}(x,\mu)(z)\|<\infty.\label{Boundbarb}
\end{eqnarray}
In order to prove \eref{Conbarb1}, it is sufficient to show that for any sequence $\{(\mu_n,z_n)\}_{n\geq1}\subset \mathscr{P}_2\times \RR^n$ with $\mu_n\to\mu$ in $\mathbb{W}_2$ and $|z_n- z|\to 0$ as $n\to \infty$, the following two statements hold
\begin{eqnarray}
&&\lim_{n\to\infty}\sup_{t\geq 0}\tilde{\EE}\left\|\partial_\mu b(x,\mu_n,Y^{x,\mu_n,y}_t)(z_n)-\partial_\mu b(x,\mu,Y^{x,\mu,y}_t)(z)\right\|=0,\label{F6.6}\\
&&\lim_{n\to\infty}\sup_{t\geq 0}\tilde{\EE}\!\left\|\partial_y b(x,\mu_n,Y^{x,\mu_n,y}_t)\!\cdot\partial_{\mu} Y^{x,\mu_n,y}_t(z_n)\!-\partial_y b(x,\mu,Y^{x,\mu,y}_t)\!\cdot\partial_{\mu} Y^{x,\mu,y}_t(z)\right\|=0.\label{F6.7}
\end{eqnarray}
By \eref{A40} we have
\begin{eqnarray*}
&&\sup_{t\geq 0}\tilde{\EE}\left\|\partial_\mu b(x,\mu_n,Y^{x,\mu_n,y}_t)(z_n)-\partial_\mu b(x,\mu,Y^{x,\mu,y}_t)(z)\right\|\nonumber \\
\leq \!\!\!\!\!\!\!\!&&\sup_{t\geq 0}\tilde{\EE}\big\|\partial_\mu b(x,\mu_n,Y^{x,\mu_n,y}_t)(z_n)-\partial_\mu b(x,\mu_n,Y^{x,\mu,y}_t)(z_n)\big\|
\nonumber \\
\!\!\!\!\!\!\!\!&&+\sup_{t\geq 0}\tilde{\EE}\big\|\partial_\mu b(x,\mu_n,Y^{x,\mu,y}_t)(z_n)-\partial_\mu b(x,\mu,Y^{x,\mu,y}_t)(z)\big\|
\nonumber \\
\leq \!\!\!\!\!\!\!\!&&C\sup_{t\geq 0}\tilde{\EE}\big|Y^{x,\mu_n,y}_t-Y^{x,\mu,y}_t\big|^{\gamma_1}+\sup_{y\in\RR^m}\big\|\partial_\mu b(x,\mu_n,y)(z_n)-\partial_\mu b(x,\mu,y)(z)\big\|.
\end{eqnarray*}
Then using \eref{13} and the uniform continuity of $\partial_\mu b(\cdot,\cdot,\cdot)(\cdot)$, we obtain \eref{F6.6}.

Using the boundedness of $\|\partial_{y}b\|$ and $\|\partial^2_{yy}b\|$, we get
\begin{eqnarray*}
&&\sup_{t\geq 0}\tilde{\EE}\left\|\partial_y b(x,\mu_n,Y^{x,\mu_n,y}_t)\cdot\partial_{\mu} Y^{x,\mu_n,y}_t(z_n)-\partial_y b(x,\mu,Y^{x,\mu,y}_t)\cdot\partial_{\mu} Y^{x,\mu,y}_t(z)\right\|\\
\leq \!\!\!\!\!\!\!\!&&\sup_{t\geq 0}\tilde{\EE}\big\|\partial_y b(x,\mu_n,Y^{x,\mu_n,y}_t)\cdot\partial_\mu Y^{x,\mu_n,y}_t(z_n)-\partial_y b(x,\mu_n,Y^{x,\mu,y}_t)\cdot\partial_\mu Y^{x,\mu_n,y}_t(z_n)\big\|
\nonumber \\
\!\!\!\!\!\!\!\!&&+\sup_{t\geq 0}\tilde{\EE}\big\|\partial_y b(x,\mu_n,Y^{x,\mu,y}_t)\cdot\partial_\mu Y^{x,\mu_n,y}_t(z_n)-\partial_y b(x,\mu,Y^{x,\mu,y}_t)\cdot\partial_\mu Y^{x,\mu_n,y}_t(z_n)\big\|
\nonumber \\
\!\!\!\!\!\!\!\!&&+\sup_{t\geq 0}\tilde{\EE}\big\|\partial_y b(x,\mu,Y^{x,\mu,y}_t)\cdot\partial_\mu Y^{x,\mu_n,y}_t(z_n)-\partial_y b(x,\mu,Y^{x,\mu,y}_t)\cdot\partial_\mu Y^{x,\mu,y}_t(z)\big\|\\
\leq \!\!\!\!\!\!\!\!&&C\sup_{t\geq 0}\left[\tilde{\EE}\big|Y^{x,\mu_n,y}_t-Y^{x,\mu,y}_t\big|^{2}\right]^{1/2}\sup_{t\geq 0}\left[\tilde{\EE}\big\|\partial_\mu Y^{x,\mu_n,y}_t(z_n)\big\|^2\right]^{1/2}\nonumber \\
\!\!\!\!\!\!\!\!&&+C\sup_{y\in\RR^m}\left\|\partial_y b(x,\mu_n,y)-\partial_y b(x,\mu,y)\right\|\left[\tilde{\EE}\big\|\partial_\mu Y^{x,\mu_n,y}_t(z_n)\big\|^2\right]^{1/2}\nonumber \\
\!\!\!\!\!\!\!\!&&+\sup_{t\geq 0}\tilde{\EE}\big\|\partial_\mu Y^{x,\mu_n,y}_t(z_n)-\partial_\mu Y^{x,\mu,y}_t(z)\big\|\\
=:\!\!\!\!\!\!\!\!&&\Pi_1^n+\Pi_2^n+\Pi_3^n.
\end{eqnarray*}

For the terms $\Pi_1^n$ and $\Pi_2^n$, by \eref{13}, \eref{38} and the uniform continuity of $\partial_y b(\cdot,\cdot,\cdot)(\cdot)$, it follows that
\begin{eqnarray}
\Pi_1^n+\Pi_2^n\rightarrow 0, \quad \text{as}\quad n\rightarrow \infty.\label{Conv1}
\end{eqnarray}

For the term $\Pi_3^n$, by a straightforward computation there exists $\beta\in (0,\gamma)$ such that
\begin{eqnarray*}
&&\tilde{\EE}\|\partial_\mu Y^{x,\mu_n,y}_t(z_n)-\partial_\mu Y^{x,\mu,y}_t(z)\|^2\\
\leq\!\!\!\!\!\!\!\!&& C\int^t_0 e^{-\beta(t-s)}\tilde{\EE}\|\partial_\mu f(x,\mu_n,Y^{x,\mu_n,y}_s)(z_n)-\partial_\mu f(x,\mu,Y^{x,\mu,y}_s)(z)\|^2ds\\
&&+C\int^t_0 e^{-\beta(t-s)}\tilde{\EE}\left[\|\partial_y f(x,\mu_n,Y^{x,\mu_n,y}_s)-\partial_y f(x,\mu,Y^{x,\mu,y}_s)\|^2\|\partial_\mu Y^{x,\mu,y}_s(z)\|^2\right]ds\\
&&+C\int^t_0 e^{-\beta(t-s)}\tilde{\EE}\|\partial_\mu g(x,\mu_n,Y^{x,\mu_n,y}_s)(z_n)-\partial_\mu g(x,\mu,Y^{x,\mu,y}_s)(z)\|^2ds\\
&&+C\int^t_0 e^{-\beta(t-s)}\tilde{\EE}\left[\|\partial_y g(x,\mu_n,Y^{x,\mu_n,y}_s)-\partial_y g(x,\mu,Y^{x,\mu,y}_s)\|^2\|\partial_\mu Y^{x,\mu,y}_s(z)\|^2\right]ds.
\end{eqnarray*}
Then by assumption $\mathbf{A\ref{A2}}$ and a similar argument above, we can prove that
\begin{eqnarray}
\Pi_3^n\rightarrow 0, \quad \text{as}\quad n\rightarrow \infty.\label{Conv2}
\end{eqnarray}
Note that \eref{Conv1} and \eref{Conv2} imply that \eref{F6.7} holds. The proof is complete. \hspace{\fill}$\Box$
\end{proof}

\subsection{Weak uniqueness of \eref{e5}}\label{subsection 6.3}

For any $u=\left(
                                              \begin{array}{c}
                                                u_1 \\
                                                u_2 \\
                                              \end{array}
                                            \right)
$ with $u_1,u_2\in\RR^n$, $\tilde{\pi}\in \mathscr{P}_2(\RR^{2n})$ with $\tilde{\mu}_1(A)=\tilde{\pi}(A\times \RR^n)$ and $\tilde{\mu}_2(A)=\tilde{\pi}(\RR^n\times A)$, $A\in \mathscr{B}(\RR^n)$, we denote
$$B(u,\tilde{\pi})=\left(
               \begin{array}{c}
                 \bar{b}(u_1,\tilde{\mu}_1) \\
                 B_2(u,\tilde{\mu})\\
               \end{array}
             \right),
\quad \Sigma(u,\tilde{\pi})=\left(
                          \begin{array}{cc}
                            \sigma(u_1,\tilde{\mu}_1) &0 \\
                            \Sigma_{21}(u,\tilde{\pi})& \Theta(u_1,\tilde{\mu}_1) \\
                          \end{array}
                        \right),
$$
where
$$B_2(u,\tilde{\pi}):=\partial_x \bar{b}(u_1,\tilde{\mu}_1)\cdot u_2+\int_{\RR^n\times \RR^n} [\partial_{\mu}\bar{b}(u_1,\tilde{\mu}_1)(z_1)\cdot z_2]\tilde{\pi}(dz_1,dz_2) $$
and
$$\Sigma_{21}(u,\tilde{\pi}):=\partial_x \sigma(u_1,\tilde{\mu}_1)\cdot u_2 +\int_{\RR^n\times \RR^n} [\partial_{\mu}\sigma(u_1,\tilde{\mu}_1)(z_1)\cdot z_2]\tilde{\pi}(dz_1,dz_2).$$

Recall equations \eref{1.3} and \eref{e5},
denote $U_t=\left(
                                            \begin{array}{c}
                                              \bar{X}_t \\
                                              Z_t \\
                                            \end{array}
                                          \right)
$, $\bar{W}_t=\left(
                \begin{array}{c}
                  W^1_t \\
                  W_t \\
                \end{array}
              \right)
$, then we immediately get
\begin{eqnarray}\label{Coupled equation}
d U_t=B(U_t, \mathscr{L}_{U_t})dt+\Sigma(U_t, \mathscr{L}_{U_t})d\bar{W}_t,
\end{eqnarray}
which is a new McKean-Vlasov SDE. By the discussion in subsection \ref{S4.3}, $\tilde{U}_t=\left(
                                                                                                                                                                                                   \begin{array}{c}
                                                                                                                                                                                                     \hat{\bar{X}} \\
                                                                                                                                                                                                     \hat{\eta}\\
                                                                                                                                                                                                   \end{array}
                                                                                                                                                                                                 \right)
$ is a weak solution of \eref{Coupled equation} under probability space $(\hat{\Omega}\times \tilde{\Omega},\hat{\mathscr{F}}\times \tilde{\mathscr{F}},\hat{\PP}\times \tilde{\PP})$. Then by the modified Yamada-Watanabe principle (see \cite[Lemma 2.1]{HW2021}), the weak uniqueness of equation \eref{Coupled equation} can be proved if we can prove equation \eref{Coupled equation} and the following SDEs
\begin{eqnarray}
d \bar{U}_t=B(\bar{U}_t, \tilde{\mu}_t)dt+\Sigma(\bar{U}_t, \tilde{\mu}_t)d\bar{W}_t\label{U+}
\end{eqnarray}
both have strong uniqueness for initial value $U_0$ with $\mathscr{L}_{U_0}=\tilde{\mu}_0$ , where $\tilde{\mu}_t=\mathscr{L}_{\tilde{U}_t}$.

To do this,  we assume that there exist two strong solutions $U_t=\left(
                                            \begin{array}{c}
                                              \bar{X}_t \\
                                              Z_t \\
                                            \end{array}
                                          \right)
$ and $U'_t=\left(
                                            \begin{array}{c}
                                              \bar{X}'_t \\
                                              Z'_t \\
                                            \end{array}
                                          \right)
$ for equation \eref{Coupled equation} with the same initial value $U_0$. Note that $\bar{X}$ and $\bar{X}'$ are the solution of equation \eref{1.3} with the same initial value, then by \cite[Theorem 2.1]{WFY}, equation \eref{1.3} has strong uniqueness, i.e.,  $\bar{X}_t=\bar{X}'_t$, $\PP$-a.s.. for any $t\geq 0$. Hence, for any $t\geq 0$, $\PP$-a.s.,
\begin{eqnarray*}
Z_t-Z'_t=\!\!\!\!\!\!\!\!&& \int^t_0 \left\{\partial_x\bar{b}(\bar{X}_s,\mathscr{L}_{\bar{X}_s})\cdot (Z_s-Z'_s)
+\EE\Big[\partial_{\mu} \bar{b}(u,\mathscr{L}_{\bar{X}_t})(\bar{X}_t)\cdot (Z_s-Z'_s)\Big]\Big|_{u=\bar{X}_s}\right\}ds \nonumber\\
 \!\!\!\!\!\!\!\!&& +\int^t_0\left\{\partial_x\sigma(\bar{X}_s,\mathscr{L}_{\bar{X}_s})\cdot (Z_s-Z'_s)+\EE\Big[\partial_{\mu} \sigma(u,\mathscr{L}_{\bar{X}_s})(\bar{X}_s)\cdot (Z_s-Z'_s)\Big]\Big|_{u=\bar{X}_s}\right\}dW_s^1.
\end{eqnarray*}
By the uniform boundedness of $\partial_x \bar{b}(x,\mu)$, $\partial_{\mu} \bar{b}(x,\mu)$, $\partial_x\sigma(x,\mu)$ and $\partial_{\mu}\sigma(x,\mu)(z)$, we have
\begin{eqnarray*}
\EE|Z_t-Z'_t|^2\leq\!\!\!\!\!\!\!\!&& C\int^t_0 \EE|Z_s-Z'_s|^2 ds,
\end{eqnarray*}
Then by Gronwall's inequality, we get $\EE|Z_t-Z'_t|^2=0$. Hence $Z_t=Z'_t$, $\PP$-a.s., which implies $U_t=U'_t$, $\PP$-a.s.~for any $t\geq 0$. Hence equation \eref{Coupled equation} is strong uniqueness. The  strong uniqueness of equation \eref{U+} can be proved by the same argument. The proof is complete.

\vspace{5mm}
\noindent\textbf{Acknowledgements} { The authors would like to thank the referees for their very constructive suggestions
and valuable comments. W. Hong is supported by NSFC (No.~12171354);  S. Li is supported by NSFC (No.~12001247) and
 NSF of Jiangsu Province (No.~BK20201019); W. Liu is supported by NSFC (No.~12171208, 11831014, 12090011);  X. Sun is supported by NSFC (No.~12271219, 11931004), the QingLan Project of Jiangsu Province and the Priority Academic Program Development of Jiangsu Higher Education Institutions.}

\vspace{5mm}
\noindent\textbf{Statements and Declarations}

\noindent The authors have no relevant financial or non-financial interests to disclose.
 The authors have no competing interests to declare that are relevant to the content of this article.
All authors certify that they have no affiliations with or involvement in any organization or entity with any financial interest or non-financial interest in the subject matter or materials discussed in this manuscript.
 The authors have no financial or proprietary interests in any material discussed in this article.
 All data generated or analysed during this study are included in this article.

\end{document}